\def\draftdate{April 14, 2022}

\documentclass{amsart}
\usepackage{amssymb}
\usepackage{xy}

\newcommand{\tact}{\zeta}

\newcommand{\bs}{\backslash}

\newcommand{\subdot}{_{\bullet}}

\def\sbreak{\texorpdfstring{\protect\ssbreak}{\ }} 
\let\ssbreak\

\let\iso\cong
\let\sma\wedge
\newcommand{\dR}{\mathbf{R}}
\newcommand{\dL}{\mathbf{L}}
\newcommand{\piS}[1][{*}]{\pi_{#1}^{\bS}}

\newcommand{\smal}{\wedge^{\dL}}

\newcommand{\tH}{\hat H\mathstrut}
\newcommand{\oleq}{\mathop{\leq}}
\newcommand{\ogeq}{\mathop{\geq}}

\newcommand{\half}{\frac12}

\renewcommand{\to}{\mathchoice{\longrightarrow}{\rightarrow}{\rightarrow}{\rightarrow}}
\newcommand{\from}{\mathchoice{\longleftarrow}{\leftarrow}{\leftarrow}{\leftarrow}}

\newcommand{\overto}[1]{\xrightarrow{\,#1\,}}
\newcommand{\overfrom}[1]{\xleftarrow{\,#1\,}}

\newcommand{\tEG}[1][{G}]{\widetilde{E#1}}
\newcommand{\tET}{\tEG[\bT]}
\newcommand{\tEOG}[1]{\widetilde{E_{#1}G}\mathstrut}
\newcommand{\tEOT}[1]{\widetilde{E_{#1}\bT}\mathstrut}

\newcommand{\tWG}{\widetilde{WG}\mathstrut^{\Delta}}

\newcommand{\CC}{HM}
\newcommand{\CCdg}{\CC_{*}}

\newcommand{\uplabel}[1]{\mvuplabel{#1}{.75pt}}

\newcommand{\mvuplabel}[2]{\hbox to 0pt{\hskip #2
\vrule height1.5em width.5pt depth0pt
\raise1.8em\hbox to 0pt{\hskip-.5ex $\scriptstyle #1$\hss}\hss}}

\newcommand{\dnlabel}[1]{\hbox to 0pt{\hskip .75pt
\vrule height0em width.5pt depth1.5em
\raise-2.2em\hbox to 0pt{\hskip-.5ex $\scriptstyle #1$\hss}\hss}}

\newcommand{\subseg}[2]{%
\overbrace{\hbox to #2{\hss%
\vbox to 8pt{\vss\hrule height2pt width2pt depth1pt}%
\vbox to 8pt{\vss\hrule height1pt width#2 depth0pt}%
\vbox to 8pt{\vss\hrule height2pt width2pt depth1pt}%
\vbox to 0pt{\vss\hrule height0pt width0pt depth8pt}%
\hss}}^{#1}}
\newcommand{\subsegnolabel}[1]{%
\hbox to #1{\hss%
\vbox to 8pt{\vss\hrule height2pt width2pt depth1pt}%
\vbox to 8pt{\vss\hrule height1pt width#1 depth0pt}%
\vbox to 8pt{\vss\hrule height2pt width2pt depth1pt}%
\vbox to 0pt{\vss\hrule height0pt width0pt depth8pt}%
\hss}}

\newcommand{\subsegnoend}[2]{%
\overbrace{\hbox to #2{\hss%
\vbox to 8pt{\vss\hrule height2pt width2pt depth1pt}%
\vbox to 8pt{\vss\hrule height1pt width#2 depth0pt}%
\vbox to 8pt{\vss\hrule height1pt width2pt depth0pt}%
\vbox to 0pt{\vss\hrule height0pt width0pt depth8pt}%
\hss}}^{#1}}

\newcommand{\subsegnoel}[1]{%
\hbox to #1{\hss%
\vbox to 8pt{\vss\hrule height2pt width2pt depth1pt}%
\vbox to 8pt{\vss\hrule height1pt width#1 depth0pt}%
\vbox to 8pt{\vss\hrule height1pt width2pt depth0pt}%
\vbox to 0pt{\vss\hrule height0pt width0pt depth8pt}%
\hss}}

\newcommand{\subsegnobeg}[2]{%
\overbrace{\hbox to #2{\hss%
\vbox to 8pt{\vss\hrule height1pt width2pt depth0pt}%
\vbox to 8pt{\vss\hrule height1pt width#2 depth0pt}%
\vbox to 8pt{\vss\hrule height2pt width2pt depth1pt}%
\vbox to 0pt{\vss\hrule height0pt width0pt depth8pt}%
\hss}}^{#1}}

\newcommand{\subsegnobl}[1]{%
\hbox to #1{\hss%
\vbox to 8pt{\vss\hrule height1pt width2pt depth0pt}%
\vbox to 8pt{\vss\hrule height1pt width#1 depth0pt}%
\vbox to 8pt{\vss\hrule height2pt width2pt depth1pt}%
\vbox to 0pt{\vss\hrule height0pt width0pt depth8pt}%
\hss}}

\DeclareMathAlphabet{\catsymbfont}{U}{rsfs}{m}{n}

\newcommand{\aC}{{\catsymbfont{C}}}
\newcommand{\Cat}{\mathop{\aC\!\mathrm{at}}\nolimits}
\newcommand{\aD}{{\catsymbfont{D}}}
\newcommand{\aF}{{\catsymbfont{F}}}
\newcommand{\aI}{{\catsymbfont{I}}}
\newcommand{\aS}{{\catsymbfont{S}}}
\newcommand{\Sp}{\aS}
\newcommand{\SpG}[1][{G}]{\Sp^{#1}}

\newcommand{\SpGbS}{\SpG\bs\bS}
\newcommand{\SpbS}{\Sp\bs\bS}
\newcommand{\Stab}{\Ho(\Sp)}
\newcommand{\StabG}[1][{G}]{\Ho(\Sp^{#1})}
\newcommand{\StabGB}[1][{G}]{\Ho^{B}(\Sp^{#1})}

\newcommand{\aM}{{\catsymbfont{M}}}
\newcommand{\Mod}{\mathop {\aM\mathrm{od}}\nolimits}
\newcommand{\aX}{{\catsymbfont{X}}}
\newcommand{\aY}{{\catsymbfont{Y}}}

\newcommand{\bC}{{\mathbb{C}}}

\newcommand{\bN}{{\mathbb{N}}}
\newcommand{\bO}{{\mathbb{O}}}
\newcommand{\bR}{{\mathbb{R}}}
\newcommand{\bS}{{\mathbb{S}}}
\newcommand{\bT}{{\mathbb{T}}}
\newcommand{\bW}{{\mathbb{W}}}
\newcommand{\bZ}{{\mathbb{Z}}}

\newcommand{\Wk}{\bW k}

\newcommand{\oA}{\mathcal{A}}
\newcommand{\Ass}{\mathop{\oA\mathrm{ss}}\nolimits}
\newcommand{\oAss}{\overline{\Ass}}
\newcommand{\oC}{\mathcal{C}}
\newcommand{\Com}{\mathop{\oC\mathrm{om}}\nolimits}

\newcommand{\oL}{\mathcal{L}}

\newcommand{\oO}{\mathcal{O}}
\newcommand{\oCO}{\oO^{\Xi}_{1}}
\newcommand{\ooO}{\overline{\oO}}
\newcommand{\ooCo}{\overline{\oC}_{1}}

\newcommand{\Ncyc}{N^{cy}}

\newcommand{\Spcat}{\Cat^{\Sp}}
\newcommand{\SpcatHk}{\Cat^{\Sp}_{Hk}}

\newcommand{\dgcatk}{\Cat^{\mathrm{dg}}_{k}}

\def\quickop#1{\expandafter\DeclareMathOperator\csname
#1\endcsname{#1}}
\quickop{id}\quickop{Id}
\quickop{Tor}\quickop{Ext}
\quickop{Hom}\quickop{Tot}
\quickop{holim}\quickop{hocolim}\quickop{hofib}
\quickop{Ho}
\quickop{cy}
\quickop{colim}
\quickop{op}
\quickop{End}
\quickop{Gr}
\quickop{LGr}
\quickop{BH}
\quickop{std}
\quickop{Cof}
\quickop{Bal}

\DeclareMathOperator*\lcolim{colim}
\DeclareMathOperator*\lhocolim{hocolim}

\numberwithin{equation}{section}

\newtheorem{thm}[equation]{Theorem}
\newtheorem{main}{Theorem}
\newtheorem{Acor}{Corollary}
\newtheorem{cor}[equation]{Corollary}

\newtheorem{lem}[equation]{Lemma}
\newtheorem{prop}[equation]{Proposition}
\theoremstyle{definition}
\newtheorem{defn}[equation]{Definition}
\newtheorem{hyp}[equation]{Hypothesis}
\newtheorem{cons}[equation]{Construction}

\newtheorem{notn}[equation]{Notation}

\theoremstyle{remark}
\newtheorem{rem}[equation]{Remark}
\newtheorem{example}[equation]{Example}

\xyoption{arrow}
\xyoption{matrix}
\xyoption{cmtip}
\SelectTips{cm}{}

\newcommand{\term}[1]{\textit{#1}}

\newdir{ >}{{}*!/-5pt/\dir{>}}

\begin{document}

\title[Strong K\"unneth Theorem for $TP_{*}$]
{The strong K\"unneth theorem for topological periodic cyclic homology}

\author{Andrew J. Blumberg}
\address{Department of Mathematics, Columbia University, 
New York, NY \ 10027}
\email{blumberg@math.columbia.edu}
\thanks{The first author was supported in part by NSF grants
DMS-1151577, DMS-1812064, DMS-2104420}
\author{Michael A. Mandell}
\address{Department of Mathematics, Indiana University,
Bloomington, IN \ 47405}
\email{mmandell@indiana.edu}
\thanks{The second author was supported in part by NSF grants
DMS-1505579, DMS-1811820, DMS-2104348}

\date{\draftdate} 
\subjclass[2010]{Primary 19D55.}
\keywords{K\"unneth theorem, topological Hochschild homology, periodic
cyclic homology, Tate cohomology.}

\begin{abstract}
Topological periodic cyclic homology (i.e., $\mathbb{T}$-Tate fixed
points of 
$THH$) has the structure of a strong symmetric monoidal functor of
smooth and proper dg categories over a perfect field of finite
characteristic. 
\end{abstract}

\maketitle

\addtocontents{toc}{\protect\setcounter{tocdepth}{1}}

\tableofcontents

\let\ssbreak\\



\section*{Introduction} 

Trace methods have produced powerful tools for computing algebraic
$K$-theory.  In these methods, one obtains information about the
$K$-theory spectrum by
mapping it to more computable theories like topological Hochschild
homology ($THH$) and topological cyclic homology ($TC$).  As the name
suggests, $THH$ is a ``topological'' analogue of Hochschild homology,
where tensor product is replaced with smash product and we work over
the sphere spectrum instead of the integers,
resulting in a theory which is the same rationally but richer
at finite primes.  Despite the name, $TC$ is not the
topological analogue of cyclic homology, but is more closely related to
negative cyclic homology.  In contrast to the algebraic setting, this
construction does not provide evident topological analogues of positive or
periodic cyclic homology.  Since $TC$ has been so successful for
calculations in 
algebraic $K$-theory, topologists have focused comparatively
little attention on such analogues until recently.

In~\cite{Hesselholt-Periodic}, Hesselholt studied a topological analogue
for periodic cyclic homology, motivated by the Deninger
program~\cite{Deninger-Main} and the recent progress on it by
Connes-Consani~\cite{ConnesConsani}.  Deninger proposed an approach to
the classical Riemann hypothesis in analogy with Deligne's proof of
Weil's Riemann hypothesis.  A basic ingredient in this approach is a
suitable infinite-dimensional cohomology theory and some kind of
endomorphism playing the role of the Frobenius automorphism; there is
then a conjectural expression for the zeta function in terms of
regularized determinants.  Using periodic cyclic homology,
Connes-Consani~\cite[1.1]{ConnesConsani} produced a cohomology theory
and expression for the product of Serre's archimedean local factors of
the Hasse-Weil zeta function of a smooth projective variety over a
number field in terms of a regularized determinant for an endomorphism
coming from the action of $\lambda$-operations on this cohomology
theory.  Hesselholt~\cite{Hesselholt-Periodic} defined a 
theory $TP$ as the Tate $\bT$-fixed points of
$THH$.  We refer to this theory as \term{topological
periodic cyclic homology} as it is the topological analogue of
periodic cyclic homology, though we note that it is not always itself
periodic.  Using  
$TP$ in place of the Connes-Consani theory, Hesselholt establishes a
non-archimedean version of their results, producing a
regularized determinant expression for the Hasse-Weil zeta function
of a smooth projective variety over a finite field.

The purpose of this paper is to prove an important structural property
of $TP$.  It is
well-known to experts that $TP$ has the natural structure of a lax
symmetric monoidal functor from dg categories (or spectral categories)
over a commutative ring $R$ (or commutative $\bS$-algebra $R$) to the
derived category of $TP(R)$-modules.  In other words, $TP$ satisfies a 
lax K\"unneth formula.  
We establish a strong K\"unneth formula
for $TP$ when restricted to smooth and proper dg categories over a
perfect field of characteristic $p>0$.

\begin{main}[Strong K\"unneth Formula]\label{main:Kunneth}
\label{MAIN:KUNNETH}
Let $k$ be a perfect field of characteristic $p>0$ and let $\aX$ and $\aY$
be $k$-linear dg categories.  The natural lax symmetric monoidal transformation 
\[
TP(\aX) \smal_{TP(k)} TP(\aY) \to TP(\aX \otimes_{k} \aY)
\]
is a weak equivalence when $\aX$ and $\aY$ are smooth and proper over $k$.
\end{main}

Here $\smal$ denotes the derived smash product:  
This theorem is a derived category statement rather than a homotopy groups
statement.  We get a homotopy groups statement from the K\"unneth
spectral sequence~\cite[IV.4.7]{EKMM}, which has $E^{2}$-term
\[
E^{2}_{*,*}=\Tor^{TP_{*}(k)}_{*,*}(TP_{*}(\aX),TP_{*}(\aY))
\]
and converges strongly to $\pi_{*}(TP(\aX)\smal_{TP(k)}TP(\aY))$.  
When $k$ is a perfect field of characteristic $p$, $\pi_{*}TP(k)\iso
\Wk[v,v^{-1}]$, where $\Wk$ denotes the $p$-typical Witt vectors,
and $v$ is an element in degree $2$.  In particular,
$\Tor^{TP_{*}(k)}_{s,*}$ vanishes for $s>1$ and the spectral sequence
degenerates to a short exact sequence.  We state this as the following
corollary. 

\begin{Acor}
Let $k$ be a perfect field of characteristic $p>0$ and let $\aX$ and $\aY$
be smooth and proper $k$-linear dg categories.  Then there is a
natural short exact sequence
\[
0\rightarrow TP_{*}(\aX)\mathbin{\mathop{\otimes}\limits_{\hspace{0pt}TP_{*}(k)}\hspace{-3.75pt}}TP_{*}(\aY)\rightarrow TP_{*}(\aX\otimes_{k}\aY)\rightarrow
\Tor^{TP_{*}(k)}_{*-1,*}(TP_{*}(\aX),TP_{*}(\aY))\rightarrow 0.
\]
This exact sequence splits, but not naturally.
\end{Acor}

We also get a strong K\"unneth theorem on homotopy groups after inverting
$p$; indeed, for the Hesselholt work on $TP$, the statements of the
main results only involve $TP_{*}[1/p]$.  Since inverting $p$ commutes with
the smash product and $TP_{*}(k)[1/p]$ is a ``graded field'' in the sense
that all graded modules over it are free (in particular, $\Tor_{s,*}=0$
for $s>0$), we have the following immediate corollary.

\begin{Acor}
Let $k$ be a perfect field of characteristic $p>0$ and let $\aX$ and $\aY$
be $k$-linear dg categories.  The natural map on homotopy groups with
$p$ inverted
\[
TP_{*}(\aX)[1/p] \otimes_{TP_{*}(k)[1/p]} TP_{*}(\aY)[1/p] \to TP_{*}(\aX \otimes_{k} \aY)[1/p]
\]
is an isomorphism when $\aX$ and $\aY$ are smooth and proper over $k$.
\end{Acor}

Along the way to proving Theorem~\ref{main:Kunneth}, we prove the
following finiteness result for $TP$. This result is surprising in light of
the relationship between $TP$ and the \'etale cohomology of the de
Rham-Witt sheaves~\cite[6.8]{Hesselholt-Periodic} and the extreme
non-finiteness of this cohomology for supersingular $K3$
surfaces~\cite[\S II.7.2]{Illusie-DRW}.  (We thank Lars Hesselholt for
calling our attention to this example.)

\begin{main}\label{main:finite}\label{MAIN:FINITE}
Let $k$ be a perfect field of characteristic $p>0$ and let $\aX$ be a
smooth and proper $k$-linear dg category.  Then $TP(\aX)$ is a small
$TP(k)$-module; in particular $TP_{*}(\aX)$ is finitely generated over
$TP_{*}(k)$. 
\end{main}

We also prove analogues of Theorems~\ref{main:Kunneth}
and~\ref{main:finite} for $C_{p}$-Tate $THH$.  For $THH$, as we review
in Section~\ref{sec:THH}, the analogue of Theorem~\ref{main:Kunneth}
is well-known to hold in much wider generality (without the smooth and
proper hypotheses, or the field hypothesis), but
Theorem~\ref{main:finite} appears to be new even in this case.  Since
the other cases take this case as input, we include the statement here
in the generality we prove it in Section~\ref{sec:spthhfg}.  We
emphasize that the statement is non-equivariant.

\begin{main}\label{main:THHfg}\label{MAIN:THHFG}
Let $R$ be a commutative ring orthogonal spectrum and $\aX$ a
smooth and proper $R$-spectral category.  Then $THH(\aX)$ is a small
$THH(R)$-module.
\end{main}

Theorem~\ref{main:THHfg} is easy to deduce from the strong K\"unneth
theorem for $THH$ and the interpretation of smooth and proper in terms
of duality~\cite[5.4.2]{Toen-Lectures},
\cite[3.7]{BGT}, but we provide a simplicial argument
in Section~\ref{sec:spthhfg}. 

An interesting application for Theorem~\ref{main:Kunneth} comes from
the theory of noncommutative motives.  Periodic cyclic homology plays
an important role in Kontsevich's ideas on noncommutative motives over
a field of characteristic $0$, as developed by
Tabuada~\cite{Tabuada-ULECT}, serving as a noncommutative replacement
for de Rham cohomology.  The fact that $TP_*[1/p]$ satisfies a
K\"unneth formula allows $TP$ to take the place of periodic cyclic
homology in
the setting of noncommutative motives over a perfect field of
characteristic $p$.  Notably, Tabuada~\cite{Tabuada-Num} (see also~\cite{Tabuada-Weil}) has recently
used our Theorem~\ref{main:Kunneth} to show that the numerical
Grothendieck group of a smooth and proper variety (the generalization
of the group of algebraic cycles up to numerical equivalence) is
finitely generated.  Tabuada also applies Theorem~\ref{main:Kunneth}
to establish a conjecture of Kontsevich~\cite{Kontsevich-talk} that
the category of numerical noncommutative motives is abelian
semi-simple, generalizing work of
Janssen~\cite{Janssen-MotivesNumerical} for numerical motives.

Finally, this paper also achieves some technical results that may be of
interest to homotopy theorists.  Some highlights include:
\begin{itemize}
\item Construction of an explicit $A_{\infty}$ coalgebra structure,
parametrized by the little $1$-cubes operad, for cellular approximations
of the diagonal for the cell structure on the geometric realization of
a simplicial set and the analogue for a simplicial
space. (See Section~\ref{sec:NegFilt}.)
\item Construction of models for the co-family universal space
$\tEG$ with action by arbitrary $A_{\infty}$ or
$E_{\infty}$ operads (see Section~\ref{sec:lax}) and a
filtration equivalent to the standard filtration (as the cofiber of
$EG\to S^{0}$)
that respects the operad structure (see Section~\ref{sec:PosFilt}). 
\item Construction of a point-set lax symmetric monoidal model of the
Tate construction (see Section~\ref{sec:lax}) and a point-set lax
monoidal filtered version of the Tate construction (see
Section~\ref{sec:Moore}).
\item Construction of the balanced smash product for right and left
modules over $A_{\infty}$ ring orthogonal spectra, and homotopical
comparison for some common $A_{\infty}$ operads as the operad varies. (See
Section~\ref{sec:compare}.)
\item Study of the $\bT$-analogue of the Hesselholt-Madsen
construction of the Tate spectral sequence, which (unlike the case of
a finite group) differs at $E^{2}$ from the Greenlees Tate
spectral sequence, before becoming isomorphic at $E^{3}$.  (See
Sections~\ref{sec:TSS} and~\ref{sec:CCS}--\ref{sec:HMvsG}.)
\end{itemize}

The reader interested in homotopy theory might wonder about
generalizing Theorems~\ref{main:Kunneth} and~\ref{main:finite} for
commutative ring spectra more general than $Hk$ and to other closed
subgroups of $\bT$ beyond $\{1\}$, $C_{p}$, and $\bT$.  Besides connectivity
of $Hk$, the proof of Theorem~\ref{main:Kunneth} depends on
Theorem~\ref{main:finite} and the observation in
Proposition~\ref{prop:kfinite} that the canonical map $THH(Hk)\to Hk$
makes $Hk$ a small $THH(Hk)$-module in the Borel category of
equivariant $THH(Hk)$-modules.  Theorem~\ref{main:finite} depends on
the fact that the commutative ring $THH_{*}(Hk)$ is Noetherian and the
commutative rings $TP_{*}(k)$ and $\pi_{*}(THH(k))^{tC_{p}}$ are a
graded PID and a graded field, respectively; for the theorem, it is
enough that they have finite global dimension.  For other
subgroups $C_{p^{n}}<\bT$ with $n>1$, $\pi_{*}(THH(k))^{tC_{p}}$ has
infinite global dimension.

\subsection*{Acknowledgments}
This project was suggested by Lars Hesselholt based on conversations
with Gon\c{c}alo Tabuada; we thank both for their interest and
contributions.  The authors learned the terminology of ``Borel
equivalences'' and ``Borel category'' in conversations with Mark
Behrens.  The second author thanks Ayelet Lindenstrauss and Teena
Gerhardt for conversations in November 2008 related to
Section~\ref{sec:spthhfg} and specifically Remark~\ref{rem:TA}. 

\section{Orthogonal \texorpdfstring{$G$}{G}-spectra and the Tate fixed points}\label{sec:spec}

This section sets out some conventions for the rest of the paper and
reviews certain aspects of the homotopy theory of equivariant
orthogonal spectra and the construction of the Tate fixed point
spectra.  Throughout the paper, $G$ denotes a compact Lie 
group, but starting in Section~3, we specialize to the case when $G$
is a finite group or $G=\bT$, 
the circle group of unit complex numbers.

Let $\Sp$ denote the category of non-equivariant orthogonal spectra
and $\SpG$ the category of orthogonal $G$-spectra, which
we understand to be indexed on all finite-dimensional orthogonal
$G$-representations (at least up to isomorphism).  We understand
\term{weak equivalence} (when used as an unmodified specific term) to
refer to the weak equivalences in the stable model structure
of~\cite[III.4.2]{MM}, or equivalently, the positive stable model
structure~\cite[III.5.3]{MM}; the weak equivalences are the maps that
induce isomorphisms on homotopy groups $\pi_{*}^{H}$ for all $H<G$,
defined by
\[
\pi_{q}^{H}X=\lcolim_{V<U}\lcolim_{n\geq \max\{0,-q\}}
\pi_{q+n} ((\Omega^{V}(X(\bR^{n}\oplus V)))^{H}).
\]
In this notation, $H<G$ means $H$ is a closed subgroup of $G$ (not
necessarily proper), $U$ denotes a fixed infinite dimensional $G$-inner
product space containing a representative of each finite dimensional
$G$-representation, and $V<U$ means that $V$ is a finite dimensional
$G$-stable vector subspace of $U$.  The homotopy category obtained by
inverting these weak equivalences is called the \term{stable category} (or
$G$-stable category or $G$-equivariant stable category) and denoted
here as $\StabG$.  We call the objects of $\StabG$ \term{spectra} (or
$G$-spectra or $G$-equivariant spectra).

The Tate construction is most naturally viewed as a functor not from
the stable category but from a localization called the 
``Borel stable category'', which is formed by inverting the ``Borel
equivalences''.

\begin{defn}\label{defn:Borel}
A map of orthogonal $G$-spectra is a \term{Borel equivalence} if it
induces an isomorphism on $\pi_{*}=\pi_{*}^{\{e\}}$ (where $\{e\}$ denotes the
trivial subgroup of $G$).  The \term{Borel colocal model structure} on orthogonal
$G$-spectra is the $\aF$-model structure of~\cite[IV.6.5ff]{MM} for
$\aF=\{\{e\}\}$, the \term{Borel local model structure} is the
Bousfield $\aF$-module structure of~\cite[IV.6.3ff]{MM} for
$\aF=\{\{e\}\}$, and the \term{Borel category} $\StabGB$ is the
homotopy category obtained by inverting the Borel equivalences.
\end{defn}

For orthogonal $G$-spectra $X$ and $Y$, the relationship between maps
in the stable category and maps in the Borel category is given by the formula~\cite[IV.6.11]{MM}
\begin{multline*}
\StabGB(X,Y)\iso \StabG(X\sma EG_{+},Y\sma EG_{+})\\
\iso
\StabG(X\sma EG_{+},Y)\iso
\StabG(X,Y^{EG}).
\end{multline*}
As always, $EG$ denotes the universal space, i.e., a free $G$-CW
space whose underlying non-equivariant space is contractible.  The
notation $Y^{EG}$ denotes the derived $G$-spectrum of unbased
non-equivariant maps from $EG$ to $Y$.  The last two isomorphisms indicate
that the (co)localization $\StabG\to \StabGB$ has both a left and a right
adjoint: The left is modeled by the functor $(-)\sma EG_{+}$ and the
right is modeled by applying fibrant replacement in the stable or
positive stable model
structure followed by the point-set mapping $G$-spectrum functor $(-)^{EG}$.

The previous formula indicates that the Borel category may be viewed
as the full subcategory of the $G$-stable category consisting of the
free $G$-spectra.  The Borel category can also be constructed as the
homotopy category of $G$-objects in orthogonal spectra, giving rise to
the Lewis-May slogan ``free $G$-spectra live in the trivial
universe''.  To be precise, let $\Sp[G]$ denote the category where an
object is an orthogonal spectrum $X$ together with an associative and
unital structure map $G_{+}\sma X\to X$ and a morphism is a map of
orthogonal spectra preserving the action maps.  The category $\Sp[G]$
then has a model structure with weak equivalences and fibrations the
underlying weak equivalences and fibrations of orthogonal spectra; the
cofibrations are retracts of cell complexes built using free $G$-cells on
spheres.
Let $\kappa \colon \Sp[G]\to \SpG$ be the functor that fills in the
non-trivial representations by the formula
\[
\kappa X(V):=\aI(\bR^{\dim V},V)\sma_{O(\dim V)}X(\bR^{\dim V})
\]
where $\aI(\bR^{\dim V},V)$ denotes the $G$-space of
non-equivariant linear isometries from $\bR^{\dim V}$ to $V$.
As observed in~\cite[V.1.5]{MM}, $\kappa$ is an equivalence of
categories with inverse the functor that forgets the non-trivial
representation indexes.  Since the formula for the homotopy groups of
an orthogonal $G$-spectrum in the case $H=\{e\}$ simplifies to
\[
\pi_{q}X=\lcolim_{n\geq \max\{0,-q\}}
\pi_{q+n} X(\bR^{n})
\]
(as the map $\lcolim_{n}\pi_{q+n} X(\bR^{n})\to \lcolim_{n}
\pi_{q+n} \Omega^{V}X(\bR^{n}\oplus V)$ is an isomorphism for all
$V<U$),
$\kappa$ sends weak equivalences to Borel equivalences and
$\kappa^{-1}$ sends Borel equivalences to weak equivalences.  Thus,
$\kappa$ and $\kappa^{-1}$ also induce equivalences on homotopy
categories $\Ho(\Sp[G])\simeq \StabGB$ (with $\kappa$ a Quillen left
adjoint in both model structures in Definition~\ref{defn:Borel}). 

If $A$ is an associative ring orthogonal $G$-spectrum, we also have
Borel local and colocal model structures on the category $\Mod_{A}$ of
equivariant $A$-modules.  We call the homotopy category the
\term{Borel derived category}, denoted $\Ho^{B}(\Mod_{A})$.  Just as
in the base case of $A=\bS$, we have
\begin{multline*}
\Ho^{B}(\Mod_{A})(X,Y)\iso \Ho(\Mod_{A})(X\sma EG_{+},Y\sma EG_{+})\\
\iso \Ho(\Mod_{A})(X\sma EG_{+},Y)\iso \Ho(\Mod_{A})(X,Y^{EG}).
\end{multline*}
As usual, we can interpret $\Ho^{B}(\Mod_{A})(X,Y)$ as the $\pi_{0}$ of a derived mapping spectrum 
$\dR^{B}F^{G}_{A}(X,Y)$. We can construct $\dR^{B}F_{A}(X,Y)$ as the
derived $G$-fixed points of the derived mapping $G$-spectrum 
\[
\dR F_{A}(X\sma EG_{+},Y\sma EG_{+})\simeq 
\dR F_{A}(X\sma EG_{+},Y)\simeq 
\dR F_{A}(X,Y^{EG_{}}),
\]
or equivalently, as the homotopy $G$-fixed points of the derived
mapping $G$-spectrum $\dR F_{A}(X,Y)$.

We now define the Tate fixed point functor,
following~\cite{GreenleesMay-Tate}.  In the definition,
$\tEG$ denotes a based $G$-space of the homotopy type of the cofiber
of the map $EG_{+} \to S^{0}$; we will become more picky about models
in later sections.

\begin{defn}\label{defn:TateFP}
Define the \term{Tate fixed points} to be the functor
\[
(-)^{tG}\colon \StabGB\to \Stab
\]
from the Borel category to the stable
category constructed as the composite of derived functors
\[
X^{tG}=(X^{EG}\sma \tEG)^{G}
\]
where $(-)^{EG}$ is the right derived inclusion functor $\StabGB\to
\StabG$.  We write $\pi^{tG}_{*}X$ for $\pi_{*}(X^{tG})$.
\end{defn}

We can construct a point-set model for the Tate fixed point functor
using a fibrant replacement functor $R_{G}$ in the (positive) stable model
structure in orthogonal $G$-spectra; we will be more specific about
$R_{G}$ in later sections.

\begin{cons}\label{cons:PSTate}
For a chosen fibrant replacement functor $R_{G}$ in the stable or
positive stable model
structure on orthogonal $G$-spectra and a chosen model for $S^{0}\to
\tEG$, let $T_{G}\colon \SpG\to \Sp$ be the composite point-set functor
\[
T_{G}(X)=(R_{G}((R_{G}X)^{EG}\sma \tEG))^{G}.
\]
\end{cons}

As constructed, $T_{G}$ sends Borel equivalences of 
orthogonal $G$-spectra to weak equivalences of orthogonal spectra 
and its derived functor is canonically naturally isomorphic to $(-)^{tG}$.

\begin{rem}\label{rem:choices}
Construction~\ref{cons:PSTate} gives us a lot of latitude in the
choices, which we take advantage of in later sections.  We pause to
note that any choices give naturally weakly equivalent functors by an
essentially unique natural transformation (where ``essentially
unique'' means a contractible space of choices): The choice of fibrant
replacement functor is unique up to essentially unique natural weak
equivalence, the space of equivariant self-maps of the $G$-space $EG$ is
contractible, and the space of equivariant based self-maps of the
based $G$-space $\tEG$ is based homotopy equivalent to $S^{0}$.
\end{rem}

\section{A lax K\"unneth theorem for Tate fixed points}\label{sec:lax}

In this section, we prove a lax version of the K\"unneth theorem for
the Tate fixed points (for an arbitrary compact Lie group $G$).  These
results are regarded as well-known, but the exposition of this section
sets up some of the constructions and notation we need in later
sections.  We deduce our stable
category results from stronger point-set results that we expect to be
useful in future work.

Construction~\ref{cons:PSTate}, our point-set model $T_{G}$ for the Tate fixed
point functor, left the choice of the fibrant replacement functor
$R_{G}$ and the model for $\tEG$ unspecified.  Our strategy in this
section is to make choices so that $T_{G}$ becomes a lax symmetric
monoidal functor.  The first step is to choose $R_{G}$ to be
lax symmetric monoidal. 

\begin{lem}\label{lem:smRG}
Let $G$ be a compact Lie group.  The positive stable model structure
on orthogonal $G$-spectra admits a topologically enriched lax symmetric monoidal
fibrant replacement functor.
\end{lem}

As the details of the construction are unimportant, we give the proof
in Section~\ref{sec:smRG}.

We next choose a model for $\tEG$.
McClure~\cite[\S2]{McClure-EinftyTate} showed that any model of $\tEG$
has the structure of an equivariant algebra over some non-equivariant
$E_{\infty}$ operad (an $N_{\infty}$ operad in the terminology
of~\cite{BlumbergHill-NormsTransfers}).  Turning this around, given a
non-equivariant $E_{\infty}$ or $A_{\infty}$ operad $\oO$, we would
like a tractable model of $\tEG$ which is an $\oO$-algebra in based
$G$-spaces (using $\sma$ as the symmetric monoidal product).  We need
the flexibility to consider varying $\oO$: in
this section, using Boardman's linear isometries operad $\oL$ lets us
take advantage of the symmetric monoidal structures considered
in~\cite{BlumbergHill-GSymmetric}, and in future sections we need to
use the Boardman-Vogt little $1$-cubes operad $\oC_{1}$ (and certain
variants).

In the following, we take $\oO$ to be an $E_{\infty}$ or $A_{\infty}$
operad (in unbased non-equivariant spaces), and for convenience of
exposition, we require $\oO$ to satisfy $\oO(0)=*$.  (If not, 
the discussion below changes in that $I$ and $S^{0}$ get replaced
with $\oO(0)\times I$ and $\oO(0)_{+}$, respectively.) For a based
space $X$, we denote by $\bO X$ the free $\bO$-algebra on $X$ in based
spaces, $\bO X=\bigvee\oO(n)_{+}\sma_{\Sigma_{n}}X^{(n)}$, where $(n)$
denotes the $n$th smash power.

\begin{cons}\label{cons:tEOG}
Let $\tEOG{\oO}$ be the $\oO$-algebra formed as the pushout in
$\oO$-algebras 
\[
S^{0}\from \bO(G\times \partial I)_{+}\to \bO(G\times I)_{+}
\]
where the rightward map is induced by the inclusion and the leftward
map is induced by the map $G\times \partial I\to S^{0}$ sending $G\times
\{1\}$ to the basepoint and $G\times \{0\}$ to the non-basepoint.
\end{cons}

It is clear that $\tEOG{\oO}$ is non-equivariantly contractible
because the basepoint lies in the unit component.  Neglecting the
$\oO$-algebra structure, the underlying $G$-space has a filtration by
$G$-equivariant cofibrations induced by homogeneous degree in the
elements of $I$,
\begin{equation}\label{eq:firstfiltration}
S^{0}=E_{0}\subset E_{1}\subset \cdots \subset \tEOG{\oO},
\end{equation}
with quotients 
\[
E_{n}/E_{n-1}\iso \oO(n)_{+}\sma_{\Sigma_{n}}(G^{n}_{+}\sma I^{n}/\partial I^{n}).
\]
Indeed, $E_{n}$ is the pushout in unbased spaces of the evident map
\[
\oO(n)\times_{\Sigma_{n}}(G^{n}\times  \partial (I^{n}))\to E_{n-1}
\]
and the inclusion
\[
\oO(n)\times_{\Sigma_{n}}(G^{n}\times  \partial I^{n})
\to \oO(n)\times_{\Sigma_{n}}(G^{n}\times I^{n}).
\]
Because $\oO(n)$ is $\Sigma_{n}$-free, it follows that for any
nontrivial $H<G$, the fixed point subspace $E^{H}$ is exactly the
subspace $S^{0}$.  This shows that $\tEOG{\oO}$ is a model of $\tEG$.  

The result of this construction is that when we use a lax symmetric
monoidal fibrant replacement functor to define $T_{G}$, we now have a
natural map
\begin{equation}\label{eq:operadassoc}
\oO(n)_{+}\sma T_{G}X_{1}\sma \cdots \sma T_{G}X_{n}
\to
T_{G}(X_{1}\sma\cdots \sma X_{n}),
\end{equation}
compatible with the operadic multiplication.  
This map is also natural in the operad $\oO$.

Using $\tEOG{\oL}$ for Boardman's linear
isometries operad $\oL$ allows us to construct a point-set lax
symmetric monoidal model of $(-) \sma \tEG$ as follows.  We
use the model of spectra given by the symmetric monoidal category of
$S$-modules inside the weak symmetric monoidal category of
$\oL(1)$-spectra in orthogonal spectra constructed
by Blumberg-Hill
in~\cite{BlumbergHill-GSymmetric} following the ideas of EKMM.
Specifically, let $\SpG{}[\oL(1)]$ denote the category with objects
orthogonal $G$-spectra $X$ equipped with an associative and unital
action map $\oL(1)_+ \sma X \to X$, and morphisms the maps $X \to Y$
compatible with the action.  Recall that this is a \emph{weak}
symmetric monoidal category, that is, a category equipped with a
product which satisfies all the axioms of a symmetric monoidal
category except that the unit map need not be an isomorphism.  As in
EKMM, there is a full symmetric monoidal subcategory of the unital
objects in $\SpG{}[\oL(1)]$, which we write as $\SpG_{\BH}$
(denoted in~\cite{BlumbergHill-GSymmetric} as
$G\aS_{\bR^{\infty}}$), the analogue of EKMM
$S$-modules.  Furthermore, \cite[4.8]{BlumbergHill-GSymmetric}
constructs a strong
symmetric monoidal functor $J \colon \SpG{}[\oL(1)] \to
\SpG_{\BH}$ and a natural weak equivalence $J\to \Id$.

Both $\SpG{}[\oL(1)]$ and $\SpG_{\BH}$ admit model structures
with weak equivalences determined by the underlying equivalences in
$\SpG$; although~\cite{BlumbergHill-GSymmetric} works with the stable
equivalences, the arguments clearly hold for the Borel equivalences as
well.  In $\SpG{}[\oL(1)]$, the fibrations are also determined by the
forgetful functor to $\SpG$.  
Note that the lax
symmetric monoidal fibrant replacement functor $R_{G}$ on $\SpG$ induces
a lax symmetric monoidal fibrant replacement functor on
$\SpG{}[\oL(1)]$: as the fibrations in $\SpG{}[\oL]$ are the
maps whose underlying maps in $\SpG$ are fibrations, $R_{G} X$ is clearly
fibrant, and $R_{G}$ inherits a lax
symmetric monoidal structure on $\SpG{}[\oL(1)]$.

The point of introducing this setup is that for the model $\tEOG{\oL}$,
$\Sigma^{\infty}\tEOG{\oL}$ is 
a commutative monoid in the weak symmetric monoidal structure on
$\oL(1)$-spectra and therefore 
\[
J((-)\sma \tEOG{\oL})\iso J((-)\sma \Sigma^{\infty}\tEOG{\oL})
\]
is a strong symmetric monoidal functor from orthogonal $G$-spectra to
the category of $J(\Sigma^{\infty}\tEOG{\oL})$-modules in $\SpG_{\BH}$.
The following theorem is an immediate consequence of the construction.

\begin{thm}\label{thm:pssymmon}
Taking $R_{G}$ to be a lax symmetric monoidal fibrant replacement functor
and the model $\tEOG{\oL}$ for $\tEG$, the functor $T_{G}$ of
Construction~\ref{cons:PSTate} lifts to a lax symmetric monoidal functor
$\SpG\to \SpG{}[\oL]$ and $J\circ T_{G}$ has the compatible structure
of a lax symmetric monoidal functor $\SpG\to \SpG_{\BH}$.
\end{thm}

\begin{notn}\label{notn:TwO}
For any $A_{\infty}$ or $E_{\infty}$ operad $\oO$ in unbased spaces
with $\oO(0)=*$, write $T^{\oO}_{G}$ for the functor $T_{G}$ of
Construction~\ref{cons:PSTate} using a lax symmetric monoidal fibrant
replacement functor and the model $\tEOG{\oO}$ for $\tEG$.
Write $JT^{\oL}_{G}\colon \SpG\to \SpG_{\BH}$ for the composite
functor $J\circ T^{\oL}_{G}$ in the previous theorem.
\end{notn}

Note that the functor $J$ preserves all weak equivalences and so the
functor $JT^{\oL}_{G}$ induces a lax symmetric monoidal functor from the Borel
stable category to the stable category; in these terms, the
previous theorem specializes to the following homotopy category
statement.

\begin{cor}\label{cor:pssymmon}
The Tate fixed point functor $(-)^{tG}\colon \StabGB\to \Stab$ has a 
lax symmetric monoidal structure.
\end{cor}

More generally, if $A$ is a commutative ring orthogonal $G$-spectrum
(or even an $\oL$-algebra in orthogonal $G$-spectra), the previous
theorem specializes to show that Tate fixed
point functor restricts to a lax symmetric monoidal functor
from the Borel derived category of $A$-modules to the derived category
of $JT^{\oL}_{G}A$-modules.

\section{The Tate spectral sequences}\label{sec:TSS}

In this section, we review the construction and properties of certain
``Tate'' spectral sequences computing $\pi^{tG}_{*}$.  We discuss two
different spectral sequences, one based on the Greenlees Tate
filtration of~\cite[\S1]{Greenlees-Tate} and another introduced by
Hesselholt-Madsen in~\cite[\S4.3]{HMAnnals}.  In the case of finite
group $G$, these spectral sequences agree from $E^{2}$ on and in the
case when $G=\bT$, they agree from $E^{3}$ on.  The former
spectral sequence is aesthetically superior, but we can prove sharper
multiplicativity properties for the latter, and we need these sharper
properties for the work in Section~\ref{sec:main}.  We begin with the
case where $G$ is a finite group as that is conceptually simpler.

For $G$ a finite group and $P_{*}$ a projective $\bZ[G]$-resolution of
$\bZ$ in which each $P_{n}$ is finitely generated, Tate cohomology is
constructed by putting together $P_{*}$ and the dual (contragradient)
$\bZ[G]$-complex $Q^{*}=\Hom(P_{*},\bZ)$.  If $X$ is a $\bZ[G]$-module, the Tate
cohomology $\tH^{-*}_{G}(X)$ is the homology of the $G$-fixed point chain
complex of the chain complex
\begin{equation}\label{eq:GT}
\cdots \from Q^{n}\otimes X\from \cdots \from Q^{0}\otimes X\from
P_{0}\otimes X\from \cdots \from P_{n}\otimes X\from \cdots
\end{equation}
using the augmentation $P_{0}\to \bZ$ and its dual $\bZ\to Q^{0}$.
(Here $Q^{0}\otimes X$ sits in degree $0$.)  Alternatively, let
$\tilde P_{*}$ denote the augmented resolution: $\tilde P_{0}=\bZ$,
$\tilde P_{n}=P_{n+1}$ for $n>0$.  The Tate cohomology can also be
calculated as the homology of the $G$-fixed point chain complex of 
\begin{equation}\label{eq:HMT}
\Tot^{\oplus}(\Hom(P_{*},X)\otimes \tilde P_{*})\iso Q^{*}\otimes \tilde P_{*}\otimes X
\end{equation}
(total complex formed with $\oplus$).
The construction~\eqref{eq:GT} has a simpler form, particularly when $G$
is a cyclic group and we take $P_{*}$ to be the minimal resolution;
for the construction~\eqref{eq:HMT}, it is easy to construct the
cup product in Tate cohomology, using the (unique up to chain homotopy)
chain maps $P_{*}\to P_{*}\otimes P_{*}$ and $\tilde P_{*}\otimes
\tilde P_{*} \to \tilde P_*$ consistent with the maps $\bZ\to \bZ\otimes \bZ$ and
$\bZ\otimes \bZ\to \bZ$, respectively.

The Greenlees Tate filtration and Hesselholt-Madsen Tate filtration do
the analogous constructions in topology.  For any $G$-CW model of $EG$, the cellular
chain complex of $EG$ is a $\bZ[G]$-resolution $P_{*}$ of $\bZ$.  
For an orthogonal $G$-spectrum $X$, we consider the tower 
of orthogonal $G$-spectra 
\begin{multline*}
\cdots \to F({EG/EG_{n}},X)\to F({EG/EG_{n-1}},X)\to \cdots\\
\cdots  \to
F({EG/EG_{0}},X)\to F({EG/EG_{-1}},X)=X^{EG},
\end{multline*}
where $F(-,X)$ denotes the orthogonal $G$-spectrum of
(non-equivariant) maps and where we understand $EG_{-1}$ as the empty
set.  When $X$ is fibrant, this is a tower of 
fibrations whose inverse limit $F({EG/EG},X)$ is trivial.
Taking the model of $\tEG$ given as the homotopy cofiber
of the map $EG_{+}\to S^{0}$, 
\[
\tEG=(EG\times I)\cup_{(EG\times \partial I)}S^{0}
\]
we get a filtration with $\tEG_{n}$
($n\geq 0$) the homotopy cofiber of $(EG_{n-1})_{+}\to S^{0}$ and
$\tEG_{0}=S^{0}$; the cellular chain complex of this filtration is
$\tilde P_{*}$ in the notation above.  

We get the Greenlees Tate filtration~\cite[\S1]{Greenlees-Tate}
(cf.~\cite[4.3.6]{HMAnnals}) by continuing the tower above using the
filtration on $\tEG$:
\[
X^{EG}=X^{EG}\sma \tEG_{0}\to X^{EG}\sma \tEG_{1}\to \cdots \to X^{EG}\sma \tEG_{n}\to \cdots .
\]
In the spectral sequence associated to this $\bZ$-indexed sequence,
the $E^{1}$-term is canonically naturally isomorphic to
\[
\cdots \from Q^{n}\otimes \pi_{*}X\from \cdots \from Q^{0}\otimes \pi_{*}X\!\from
P_{0}\otimes \pi_{*}X\from \cdots \from P_{n}\otimes \pi_{*}X\from \cdots,
\]
precisely the complex of~\eqref{eq:GT}.  Because each homotopy cofiber
is $G$-free and has $G$-free homotopy groups, the homotopy groups of
the $G$-fixed point spectra are the $G$-fixed points of the homotopy
groups.  The spectral sequence associated to the induced
$\bZ$-indexed sequence on $G$-fixed points therefore has $E^{1}$-term
the complex 
\[
\cdots \from (Q^{0}\otimes \pi_{*}X)^{G}\from
(P_{0}\otimes \pi_{*}X)^{G}\from \cdots 
\]
of $G$-fixed points, and so has $E^{2}$-term given by
$E^{2}_{i,j}=\tH^{-i}_{G}(\pi_{j}X)$.   Because\break $\holim_{n}
F(EG/EG_{n-1},X)^{G}\simeq *$, this spectral sequence is conditionally
convergent (in the sense of~\cite[5.10]{Boardman-SpectralSequences})
to the colimit $\pi^{tG}_{*}X$. 

\begin{defn}\label{defn:GTfilt}
Let $G$ be a finite group and let $X$ be an orthogonal $G$-spectrum.
The \term{Greenlees Tate filtration} on $X^{tG}$ is the
$\bZ$-indexed sequence
\[
\cdots\to X^{tG}_{-1}\to X^{tG}_{0}\to X^{tG}_{1}\to \cdots,
\]
where $X^{tG}_{n}$ is the composite of derived functors
\[
X^{tG}_{n}=\begin{cases}
F({EG/EG_{-n-1}},X)^{G}&\text{if }n< 0\\
(X^{EG}\sma \tEG_{n})^{G}&\text{if }n\geq 0
\end{cases}
\]
and the maps are induced by the maps $EG/EG_{n-1}\to EG/EG_{n}$ and $\tEG_{n}\to
\tEG_{n+1}$.  The \term{Greenlees Tate spectral} sequence is the
associated conditionally convergent spectral sequence
\[
E^{2}_{i,j}=\tH^{-i}_{G}(\pi_{j}X)\Longrightarrow \pi^{tG}_{i+j}(X).
\]
\end{defn}

The Hesselholt-Madsen Tate filtration~\cite[\S4.3]{HMAnnals}
follows the pattern 
of~\eqref{eq:HMT}.  Because we use a version of this filtration for
the arguments in the remainder of the paper, it is convenient to set
up a point-set model for it rather than just a construction in the
stable category.  For an orthogonal $G$-spectrum $X$, using the chosen
lax symmetric monoidal fibrant replacement functor $R_{G}$, let
\begin{equation}\label{eq:Tij}
T_{G}X_{i,j}=(R_{G}(F(EG/EG_{j-1},R_{G}X)\sma \tEG_{i}))^{G}
\end{equation}
for any $i,j\geq 0$. If the filtration on $EG$ and $\tEG$ comes from a
$G$-CW structure as above, then $T_{G}X_{0,j}$ is a point-set model for
$X^{tG}_{-j}$ and $T_{G}X_{i,0}$ is a point-set model for $X^{tG}_{i}$
in the Greenlees Tate filtration.  We have canonical maps
$T_{G}X_{i,j}\to T_{G}X_{i',j'}$ for any $i\leq i'$, $j\geq j'$,
making $T_{G}X_{-,-}$ a functor from $(\bN,\oleq)\times (\bN,\ogeq)$ to
orthogonal $G$-spectra.  Using the functor minus, $-\colon (\bN,\oleq)\times
(\bN,\ogeq)\to (\bZ,\oleq)$, for any $n\in \bZ$, define the point-set
functor
\begin{equation}\label{eq:Thocolim}
\bar T_{n} X=\lhocolim_{i-j\leq n} T_{G}X_{i,j}
\end{equation}
from $(\bZ,\leq)$ to orthogonal spectra
using the categorical bar construction as the point-set model for
$\hocolim$.  Then for $n<n'$, we have an induced map $\bar T_{n} X\to \bar
T_{n'} X$. The consistent system of 
maps $T_{G}X_{i,j}\to T_{G}X$ induce a consistent system of maps
$\bar T_{n}X\to T_{G}X$ and a weak equivalence $\hocolim \bar
T_{n}X\to T_{G}X$.  Hesselholt-Madsen~\cite[4.3.4]{HMAnnals} proves
that the $E^{1}$-term of the spectral sequence associated to this
$\bZ$-indexed sequence is canonically naturally isomorphic as a chain complex
to 
\[
(\Tot^{\oplus}(\Hom(P_{*},\pi_{*}X)\otimes \tilde P_{*})^{G}
\]
(specifically, $E^{1}_{i,j}$ is the degree $i$ part of the total
complex for $\pi_{j}X$), and
\[
E^{2}_{i,j}\iso \tH^{-i}_{G}(\pi_{j}X).
\]
Moreover, Hesselholt-Madsen~\cite[4.3.6]{HMAnnals} constructs a zigzag
of maps consistent with the abutment to $\pi^{tG}_{*}X$ and inducing
an isomorphism on $E^{2}$ between this
spectral sequence and the Greenlees Tate spectral sequence, so this
spectral sequence also conditionally converges to the colimit
$\pi^{tG}_{*}X$. 

\begin{defn}\label{defn:HMTfilt}
Let $G$ be a finite group and let $X$ be an orthogonal $G$-spectrum.
The \term{Hesselholt-Madsen Tate filtration} on $X^{tG}$ is the
$\bZ$-indexed sequence
\[
\cdots\to \bar T_{-1}X\to \bar T_{0}X\to \bar T_{1}X\to \cdots,
\]
constructed in the previous paragraph. The \term{Hesselholt-Madsen Tate spectral} sequence is the
associated conditionally convergent spectral sequence whose
$E^{1}$-term is
\[
E^{1}_{i,j}=(\Tot^{\oplus}(\Hom(P_{*},\pi_{j}X)\otimes \tilde P_{*})^{G}_{i}
\]
(for $P_{*}$ the cellular chain complex of $EG$, a locally finite free
$\bZ[G]$-resolution of $\bZ$); it is isomorphic to the Greenlees Tate
spectral sequence from $E^{2}$ onward.
\end{defn}

As constructed, the Hesselholt-Madsen Tate filtration is a filtration in
the traditional sense of $\colim_{n} \bar T_{n}X \overto{\simeq}X^{tG}$.
Recall that a map $f\colon A\to B$ in any topologically enriched
category is called an \term{$h$-cofibration} when it satisfies the
\term{homotopy extension property}: Given a map $g\colon B\to C$, a
homotopy of $g\circ f\colon A\to C$ may be extended to a homotopy of
$g\colon B\to C$.  For orthogonal spectra, this is equivalent to the
map
\[
B\cup_{A\sma \{0\}_{+}}A\sma I_{+}\to B\sma I_{+}
\]
admitting a retraction, and so $h$-cofibrations are preserved by many
point-set constructions including smash product and gluing.
They also have important homotopy colimit properties including:
\begin{itemize}
\item If $A\to B$ is an $h$-cofibration of orthogonal spectra
and $A\to B$ is any map of orthogonal spectra, then the pushout
$B\cup_{A}C$ represents the homotopy pushout (left derived functor of
pushout).  In particular, the
quotient $B/A$ represents the left derived quotient, the homotopy
cofiber in the stable category.
\item If $A_{0}\to A_{1}\to \dotsb $ is a system of $h$-cofibrations
of orthogonal spectra,
then the colimit represents the homotopy colimit.
\end{itemize}
It is well-known and straightforward
to prove that the categorical bar construction model of the homotopy
colimit has the property that the inclusion of a subcategory induces
an $h$-cofibration on homotopy colimits.  In the particular case of
the Hesselholt-Madsen Tate filtration, this is the following observation.

\begin{prop}\label{prop:hcof}
The maps $\bar T_{n}X\to \bar T_{n+1}X$ in the point-set model of the
Hesselholt-Madsen Tate filtration above are $h$-cofibrations.
\end{prop}

Technically, both the filtrations and the $E^{1}$-terms of the
spectral sequences in Definitions~\ref{defn:GTfilt}
and~\ref{defn:HMTfilt} depend on the choice of $G$-CW structure on
$EG$.  While it is tempting to use one that gives the minimal
resolution, we will use the $G$-CW structure on $EG$ coming from the standard
two-sided bar construction model:  We use the model of $EG$
constructed as the geometric realization of the simplicial $G$-space 
\[
B_{n}(G,G,*)=G\times (G^{n})\times *
\]
with the usual face and degeneracy maps (induced by
multiplication/projection and by inclusion of the unit, respectively)
and the $G$-action on the lefthand factor.  The geometric realization
filtration is evidently the cellular filtration of a $G$-CW
structure.  The
following two lemmas are well-known but in particular follow from the
more delicate study of the structure we perform in
Sections~\ref{sec:PosFilt} and ~\ref{sec:NegFilt}.

\begin{lem}\label{lem:stddiag}
For the standard bar construction model of $EG$, the diagonal map
$EG\to EG\times EG$ admits an equivariant cellular approximation that
is homotopy coassociative and homotopy counital through equivariant
cellular maps. 
\end{lem}

\begin{lem}\label{lem:stdmult}
For the standard bar construction model of $EG$, the space $\tEG$
admits a pairing $\tEG\sma \tEG\to \tEG$ that is equivariant and
filtered and is homotopy unital and associative though equivariant
filtered maps.
\end{lem}

Since the space of equivariant maps $\tEG\sma \tEG\to \tEG$ is
homotopy discrete with components corresponding to the trivial map
and a weak equivalence, the multiplication on $\tEG$ in
Lemma~\ref{lem:stdmult} is compatible with the multiplication
constructed on the models in Section~\ref{sec:lax}.  In particular, the
maps in the previous two lemmas induce the same pairing in the
stable category $T_{G}(X)\smal T_{G}(Y)\to T_{G}(X\smal Y)$ as
Corollary~\ref{cor:pssymmon}. 

We can use the preceding two lemmas to give the Hesselholt-Madsen Tate
spectral sequence a natural lax monoidal structure starting at the
$E^{1}$-term.  The equivariant cellular model of the diagonal map $EG
\to EG \times EG$ induces equivariant maps
\[
EG_{n-1}\to (EG_{i-1}\times EG)\cup (EG\times EG_{j-1})\subset EG\times EG
\]
for $i+j=n$, which in turn induce equivariant maps
\begin{multline*}
EG/EG_{i+j-1}\to 
(EG\times EG)/((EG_{i-1}\times EG)\cup (EG\times EG_{j-1}))\\
\iso (EG/EG_{i-1})\sma (EG/EG_{j-1})
\end{multline*}
and equivariant maps
\begin{multline*}
F(EG/EG_{i-1},R_{G}X)\sma 
F(EG/EG_{j-1},R_{G}Y)\\
\to
F(EG/EG_{i-1}\sma EG/EG_{j-1},R_{G}X\sma R_{G}Y)\\
\to
F(EG/EG_{i+j-1},R_{G}(X\sma Y)).
\end{multline*}
Likewise the equivariant cellular model for the
multiplication $\tEG \sma \tEG \to \tEG$ induces
\[
(X\sma \tEG_{i})\sma (Y\sma \tEG_{j})\to (X\sma Y)\sma \tEG_{i+j}.
\]
Returning to $T_{G}$, we then obtain natural maps
\[
T_{G}X_{i,j}\sma T_{G}Y_{i',j'}\to T_{G}(X\sma Y)_{i+i',j+j'}.
\]
These are functorial in $i,i',j,j'$ and define maps on the
Hesselholt-Madsen Tate filtration
\[
\bar T X_{m}\sma \bar T Y_{n}\to \bar T(X\sma Y)_{m+n},
\]
respecting the $(\bN,\oleq)\times (\bN,\oleq)$-structure on both
sides.  This in turn induces a pairing 
\begin{equation}\label{eq:HMTpair}
E^{1}_{*,*}(X)\otimes E^{1}_{*,*}(Y)\to E^{1}_{*,*}(X\smal Y)
\end{equation}
converging to the standard pairing on $\pi^{tG}_{*}$.
The homotopy coassociativity of $EG$ and homotopy associativity of
$\tEG$ then give us filtered homotopies between the two maps
\[
\bar TW_{\ell}\sma \bar TX_{m}\sma \bar TY_{n}
\to \bar T(W\sma X\sma Y)_{\ell+m+n}
\]
and so both associations induce the same map
\[
E^{1}_{*,*}(W)\otimes E^{1}_{*,*}(X)\otimes E^{1}_{*,*}(Y)\to 
E^{1}_{*,*}(W\smal X\smal Y).
\]
Similar observations apply to the unit using the canonical map $\bS\to
T_{G}\bS_{0,0}$.  Hesselholt-Madsen~\cite[4.3.5]{HMAnnals} shows that
such a pairing induces the usual pairing on Tate cohomology on the
$E^{2}$-term.

We can describe this pairing algebraically as follows.
Lemma~\ref{lem:stddiag} endows the resolution $P_{*}$ with an equivariant
coassociative and counital comultiplication $P_{*}\to P_{*}\otimes
P_{*}$ and Lemma~\ref{lem:stdmult} endows $\tilde P$ with an equivariant
differential graded algebra structure.  In fact, for our construction
in Section~\ref{sec:NegFilt}, the coalgebra structure on $P_{*}$ is
induced by the Alexander-Whitney map.  The resulting differential graded
algebra plays a central role in our formulas, and we use the following
notation. 

\begin{notn}\label{notn:CCfg}
Let $\CCdg$ be the differential graded algebra
\[
(\Tot^{\oplus}(P_{*},\bZ)\otimes \tilde P)^{G},
\]
where $P_{*}$
denotes the cellular chain complex of the standard bar construction
model of $EG$ and where the multiplication is induced by Lemmas~\ref{lem:stddiag}
and~\ref{lem:stdmult} and the unit by the augmentation $P_{0}\to \bZ$
and isomorphism $\bZ\to \tilde P_{0}$.  
Let $\CC_{*,*}$ be the 
bigraded ring which is $\CC_{*}$ concentrated in degree $0$ in the
second grading (the \term{internal} grading).
\end{notn}

The pairing on $E^{1}$-terms~\eqref{eq:HMTpair} makes the $E^{1}$-term
for $\bS$ into a differential graded algebra, which acts on both sides
on the $E^{1}$-term for any orthogonal $G$-spectrum.  By the formula
for the $E^{1}$-term in Definition~\ref{defn:HMTfilt}, we see that as
a bigraded ring, the $E^{1}$-term for $\bS$ is $\CC_{*}\otimes
\pi_{*}\bS$.  In particular, the $E^{1}$-term for every orthogonal
$G$-spectrum is naturally a bimodule over $\CC_{*,*}$.  

Using the fact that $P_{0}=\bZ[G]$ and $\tilde P_{0}=\bZ$, we get a natural
map 
\[
\pi_{*}X\to (\Hom(P_{0},\pi_{*}X)\otimes \tilde P_{0})^{G}\to E^{1}_{0,*}(X)
\]
where $x\in \pi_{j}X$ goes to the unique equivariant homomorphism 
$P_{0}\to \pi_{j}X$ sending $1$ to $x$.  It is easy to see that this
is a monoidal natural transformation.  The following proposition is
then clear from inspection of the multiplication.

\begin{prop}\label{prop:CCfg}
For any orthogonal $G$-spectrum $X$, the $E^{1}$-term of the
Hesselholt-Madsen Tate spectral sequence is a bimodule over
$\CC_{*,*}$.  The map of left modules from $\CC_{*,*}\otimes \pi_{*}X$ to
the $E^{1}$-term for $X$ is a map of
$\CC_{*,*}$-bimodules and an isomorphism.
\end{prop}

Regarding $\CC_{*,*}\otimes \pi_{*}(-)$ as a functor from $G$-spectra
to $\CC_{*,*}$-bimodules,
the unit of $\CC_{*,*}\otimes \pi_{*}\bS$ and the pairing
\[
(\CC_{*,*}\otimes \pi_{*}(X))\otimes (\CC_{*,*}\otimes \pi_{*}(Y))\to
\CC_{*,*}\otimes \pi_{*}(X\sma Y)
\]
induced by the multiplication of $\CC_{*,*}$ and the usual pairing on
$\pi_{*}$ makes $\CC_{*,*}\otimes \pi_{*}(-)$ into a lax monoidal
functor.  The isomorphism of the
previous proposition is then a monoidal natural transformation from this
functor to the $E^{1}$-term of the Hesselholt-Madsen spectral sequence.

We now turn to the case of the circle group $\bT$.  We write $\bC(1)$
for $\bC$ with the standard action of $\bT$ as the group of unit
complex numbers, $S(\bC(1)^{n})$ for the unit sphere in $\bC(1)^{n}$,
and $S^{\bC(1)^{n}}$ for the one-point compactification of
$\bC(1)^{n}$. When working in the $\bT$-equivariant stable category, we
write $S^{n\bC(1)}$ for the suspension spectrum of $S^{\bC(1)^{n}}$
for $n\in \bN$, and we extend this notation to representation spheres
$S^{n\bC(1)}$ for all $n\in \bZ$.  We have the standard bar
construction model for $E\bT$, which comes with a filtration from
geometric realization, but to better match the numbering in the finite
group case, we define $E\bT_{2n+1}=E\bT_{2n}$ and let $E\bT_{2n}$ be
the geometric realization $n$-skeleton; we use the corresponding
numbering for the filtration on $\tET$.  In this numbering, there are
well-known $\bT$-equivariant homotopy equivalences
\[
E\bT_{2n}\simeq S(\bC(1)^{n}), \qquad \tET_{2n}\simeq S^{\bC(1)^{n}}
\]
as well as $\bT$-equivariant homotopy equivalences
\[
E\bT/E\bT_{2n-1}=E\bT/E\bT_{2n-2}\simeq E\bT_{+}\sma S^{\bC(1)^{n}}.
\]
These $\bT$-equivariant homotopy equivalences give us natural
isomorphisms in the $\bT$-equivariant stable category 
\[
X^{E\bT/E\bT_{2n-1}}\simeq X^{E\bT}\smal S^{-n\bC(1)}
\qquad \text{and}\qquad
X^{E\bT}\sma \tET_{2n}\simeq X^{E\bT}\smal S^{n\bC(1)}.
\]
The \term{Greenlees Tate filtration} is then
\begin{multline}\label{eq:TGTfilt}
\cdots \to (X^{E\bT}\smal S^{-n\bC(1)})^{\bT}\to \cdots 
\to (X^{E\bT}\smal S^{-\bC(1)})^{\bT}\to
(X^{E\bT})^{\bT}\\
\to
(X^{E\bT}\smal S^{\bC(1)})^{\bT}\to \cdots \to
(X^{E\bT}\smal S^{n\bC(1)})^{\bT}\to \cdots
\end{multline}
(even indices only displayed, odd indices equal to previous even
index, with $(X^{ET})^{\bT}$ in index $0$).
Using the identification of the quotient $S^{\bC(1)}/S^{0}$ as
$\bT_{+}\sma S^{1}$, the associated graded spectra are 
\[
(\bT_{+}\sma X^{E\bT}\smal S^{(n-1)\bC(1)}\sma S^{1})^{\bT}
\simeq (\bT_{+}\sma \Sigma^{2n-1}X)^{\bT}
\simeq \Sigma^{2n}X,
\]
where the last weak equivalence is the Adams isomorphism.  The
\term{Greenlees Tate spectral sequence} is the spectral sequence associated to the $\bZ$-indexed
Greenlees Tate filtration~\eqref{eq:TGTfilt}.  It has
\begin{equation}\label{eq:TGTSSE2}
E^{1}_{i,j}=E^{2}_{i,j}=\pi_{j}X \Longrightarrow \pi^{t\bT}_{i+j}X,
\end{equation}
conditionally converging to the colimit $\pi^{t\bT}_{i+j}X$.  (In
general $E^{2r-1}=E^{2r}$ for all $r$ since the filtration is
concentrated in even indices.)
Moreover, because the associative and commutative pairing (in the
equivariant stable homotopy category) 
\[
S^{m\bC(1)}\smal S^{n\bC(1)}\overto{\simeq}S^{(m+n)\bC(1)}
\]
is compatible with the maps in the Greenlees Tate filtration, the 
Greenlees Tate spectral sequence has a natural associative and unital
pairing on $E^{1}$ compatible with the pairing on $\pi^{t\bT}_{*}$.  In
terms of the description of the $E^{1}$-term in~\eqref{eq:TGTSSE2},
the pairing on $E^{1}$ is induced by $\pi_{*}X\otimes \pi_{*}Y\to
\pi_{*}(X\smal Y)$.

Although the Greenlees Tate filtration and Greenlees Tate spectral
sequence are ideal for many purposes, we have not succeeded in making
the pairing coherent enough for the argument in
Section~\ref{sec:main}.  For that argument we work with the
Hesselholt-Madsen filtration, which does not have as a clean a
spectral sequence for $G=\bT$; in particular it does not agree at
$E^{2}$ with the Greenlees Tate spectral sequence
(see~\ref{thm:CCT}--\ref{thm:HMvsG} below).  We construct a 
point-set model for the Hesselholt-Madsen filtration exactly as in
Definition~\ref{defn:HMTfilt}, except using the doubled indexing for
the filtration on $E\bT$ and $\tET$ as in the preceding paragraph.

\begin{defn}\label{defn:THM}
For $G=\bT$, the \term{Hesselholt-Madsen Tate filtration} is the filtration
by $h$-cofibrations
\[
\cdots \to \bar T_{-1} X\to \bar T_{0} X\to \bar T_{1}X\to \cdots 
\]
where $\bar T_{n}X=\hocolim_{i-j\leq n}T_{\bT}X_{i,j}$ for 
\[
T_{\bT}X_{i,j}=(R_{\bT}(F(E\bT/E\bT_{j-1},R_{\bT}X)\sma \tET_{i}))^{\bT}.
\]
The \term{Hesselholt-Madsen Tate spectral sequence} is the spectral
sequence associated to this filtration.
\end{defn}

In the finite group case, conditional convergence of the
Hesselholt-Madsen spectral sequence followed from conditional
convergence of the Greenlees Tate spectral sequence, where it was
clear from construction.  For the case $G=\bT$, we prove it
separately.  As the details are not needed in the main argument, we
defer the proof of the following lemma to Section~\ref{sec:CCS}.

\begin{lem}\label{lem:HMTcc}
For $G=\bT$, the Hesselholt-Madsen Tate spectral sequence converges
conditionally to the colimit $\pi^{t\bT}_{*}X$. 
\end{lem}

Lemmas~\ref{lem:stddiag} and~\ref{lem:stdmult} hold for $G=\bT$
using these filtrations.  We therefore again obtain a monoidal
structure on the spectral sequence as in~\eqref{eq:HMTpair}.  The
$E^{1}$-term for $\bS$ is a bigraded ring that acts on both sides on
the $E^{1}$-term for an arbitrary $X$.  We make the following
computation in Section~\ref{sec:CCS}, where we specify the generators
precisely.

\begin{thm}\label{thm:CCS}
For $G=\bT$, the $E^{1}$-term for the Hesselholt-Madsen $\bT$-Tate spectral
sequence for $X=\bS$ is as a bigraded ring the free graded commutative
$\pi_{*}\bS$-algebra on a generator $x$ in bidegree $(2,0)$, a generator
$y$ in bidegree $(2,-1)$, and a generator $z$ in bidegree $(-2,0)$ subject
to the relation $y^{2}=0$.  
\end{thm}

In the previous statement, for the $\pi_{*}\bS$ action, we regard
elements of $\pi_{*}\bS$ as concentrated in degree 0 for the first
grading, i.e., elements of $\pi_{n}\bS$ are in bidegree $(0,n)$.  In
analogy with Notation~\ref{notn:CCfg}, we use the following notation
in the case $G=\bT$, so that the $E^{1}$-term of the Hesselholt-Madsen
Tate spectral sequence for $\bS$ is isomorphic as a bigraded ring to
$\CC_{*,*}\otimes \pi_{*}\bS$.

\begin{notn}\label{notn:CCT}
Let $\CC_{*,*}=\bZ[x,y,z]/y^{2}$ where $x$ is in bidegree $(2,0)$ and $y$
is in bidegree $(2,-1)$, and $z$ is in bidegree $(-2,0)$.
\end{notn}

For an arbitrary orthogonal $\bT$-spectrum $X$, using the canonical isomorphism 
\[
\CC_{*,*}\otimes \pi_{*}X\iso 
(\CC_{*,*}\otimes \pi_{*}\bS)\otimes_{\pi_{*}\bS} \pi_{*}X,
\]
the $E^{1}$-term of the Hesselholt-Madsen $\bT$-Tate spectral sequence for
$X$ naturally becomes a bimodule over $\CC_{*,*}$.  We have a canonical
map $\pi_{*}X\to E^{1}_{0,*}$ induced by the identification of 
\[
i^{*}X\simeq (R_{\bT}F(E\bT_{0},R_{\bT}X))^{\bT}
\]
as the homotopy cofiber of $T_{\bT}X_{0,1}\to T_{\bT}X_{0,0}$, where
$i^{*}$ denotes the underlying non-equivariant spectrum.  We
then get an induced a map of left $\CC_{*,*}$-modules from
$\CC_{*,*}\otimes \pi_{*}X$ to the $E^{1}$-term for $X$.  We prove in
Section~\ref{sec:CCS} the following theorem.

\begin{thm}\label{thm:CCT}
The map from $\CC_{*,*}\otimes \pi_{*}X$ to the $E^{1}$-term of the
Hesselholt-Madsen $\bT$-Tate spectral sequence for $X$ is an isomorphism of
$\CC_{*,*}$-bimodules and a monoidal transformation. 
\end{thm}

Non-multiplicatively, the previous theorem identifies the $E^{1}$-term of the
Hessel\-holt-Madsen Tate spectral sequence for an orthogonal
$\bT$-spectrum $X$ as having $E^{1}_{2i,j}$ a countable direct sum of
copies of $\pi_{j}X\oplus \pi_{j+1}X$ and $E^{1}_{2i-1,j}$ zero.
\par

Although the work above is all we need for the main results of the
paper, we can say more about the Hesselholt-Madsen $\bT$-Tate spectral
sequence and its relationship to the Greenlees $\bT$-Tate spectral
sequence.  We prove the following theorems in Section~\ref{sec:HMvsG}.

\begin{thm}\label{thm:HMvsG}
In the notation of Definition~\ref{defn:THM}, the inclusion of
$T_{\bT}X_{0,j}$ in $\bar T_{-j}X$ and $T_{\bT}X_{i,0}$ in $\bar
T_{i}X$ induce a (non-monoidal) filtered map from the Greenlees
filtration to the Hesselholt-Madsen filtration for $G=\bT$.  The
induced map on spectral sequences is split injective on $E^{1}=E^{2}$
and an isomorphism from $E^{3}$ onward.
\end{thm}

\begin{thm}\label{thm:HMd2}
In the Hesselholt-Madsen Tate spectral sequence for $G=\bT$, the
$d_{1}$ differential is $0$.  The $d_{2}$ differential satisfies the
Leibniz rule, is given on generators of $HM_{*,*}\otimes \pi_{*}(\bS)$ by 
\begin{align*}
d_{2}(x\otimes 1)&=1\otimes \eta\\
d_{2}(y\otimes 1)&=xz\otimes 1 +yz\otimes \eta -1\otimes 1\\
d_{2}(z\otimes 1)&=z^{2}\otimes \eta 
\end{align*}
and is given on elements $v$ of $\pi_{*}X$ by $d_{2}(1\otimes
v)=-z\otimes \tact v$, where $\tact$ denotes the action of the
fundamental class of $\pi_{1}\bT$ on $\pi_{*}X$. 
\end{thm}

The $E^{1}=E^{2}$ term of the Greenlees $\bT$-Tate spectral sequence
for $X$ is (multiplicatively) isomorphic to $\bZ[t,t^{-1}]\otimes
\pi_{*}X$, where $t$ is an element of $E^{1}_{-2,0}$ for $X=\bS$.  The
element $t$ has $d_{2}(t\otimes 1)=t^{2}\otimes \eta$ and is usually
chosen so that for $v\in \pi_{*}X$, $d_{2}(1\otimes v)=t\otimes \tact
v$.  With this choice of $t$ (which is now uniquely specified), the
map in Theorem~\ref{thm:HMvsG} then takes $t^{n}\otimes v$ to
\[
\begin{cases}
(-z)^{n}&n\geq 0\\
(-x)^{|n|}\otimes v-(-x)^{|n|-1}y\otimes \tact v&n<0.
\end{cases}
\]
From Theorem~\ref{thm:HMd2}, we see that the ring map
$HM_{*,*}\to \bZ[t,t^{-1}]$ sending $x$ to $-t^{-1}$, $y$ to $0$, and
$z$ to $-t$ induces a multiplicative map of spectral sequences from
the Hesselholt-Madsen $\bT$-Tate spectral sequence to the Greenlees
$\bT$-Tate spectral sequence that splits the (non-multiplicative) map
in Theorem~\ref{thm:HMvsG}.

\section{Topological periodic cyclic homology}

Having reviewed the definition and basic properties of the Tate fixed
point functor in the previous three 
sections, we now review the definition and basic properties of
topological periodic cyclic homology ($TP$).  
We begin with a very rapid review of relevant aspects of the theory of
$THH$ of spectral categories and $\Mod_{R}$-categories for $R$ a
commutative ring orthogonal spectrum.  We define $TP$ of spectral
categories and, as discussed below, we rely on
the equivalence of $k$-linear dg categories and $\Mod_{Hk}$-categories to
define $TP$ of dg categories.

Let $\Spcat$ denote the category of small spectral categories and
spectral functors: An object of $\Spcat$ is a small category enriched
in orthogonal spectra and a morphism in $\Spcat$ is an enriched
functor.  As explained in~\cite[\S 3]{BM-tc}, the topological
Hochschild-Mitchell complex $\Ncyc$ yields a functor from orthogonal
spectra to orthogonal $\bT$-spectra.  Since we are working
with Borel equivalences, the left derived functor provides an adequate
model for $THH$.  (In other contexts, the construction is more subtle;
see~\cite[\S5]{ABGHLM}.)

To make sense of the left derived functor $THH$, we say that a small
spectral category is \term{pointwise cofibrant} if each mapping
spectrum $\aC(x,y)$ is a cofibrant orthogonal spectrum (in the
standard model structure).  A spectral
functor $F \colon \aC \to \aD$ is a \term{pointwise equivalence} if
the induced function on object sets is the identity and each map of
spectra $\aC(x,y) \to \aD(x,y)$ is a weak equivalence.  The argument
of~\cite[2.7]{BM-tc} shows that there is an endofunctor $Q$ of
$\Spcat$ that lands in pointwise cofibrant small spectral categories
and a natural transformation $Q \to \Id$ through pointwise
equivalences.  The topological Hochschild-Mitchell complex $\Ncyc$
composed with $Q$ then takes pointwise equivalences of spectral
categories to Borel equivalences of orthogonal
$\bT$-spectra. 

Although pointwise cofibrant replacement is technically convenient, we
are typically interested in less rigid notions of equivalence on
$\Spcat$.  We say that a functor $F \colon \aC \to \aD$ is a \term{Dwyer-Kan
equivalence} (or \term{DK-equivalence}) if the induced functor $\pi_{0}(\aC) \to \pi_{0}(\aD)$
is essentially surjective and each map of orthogonal spectra $\aC(x,y)
\to \aD(Fx,Fy)$ is a weak equivalence~\cite[5.1]{BM-tc}.  We say that
a functor $F \colon \aC \to \aD$ is a Morita equivalence if the
induced functor on ``triangulated closures'' is a DK-equivalence
(see~\cite[\S 5]{BM-tc}).  Here the triangulated closure of a small
spectral category $\aC$ is a spectral category $\aC'$ such that
$\Ho(\aC')$ is the pre-triangulated closure of
$\Ho(\aC)$~\cite[5.5]{BM-tc}.

The composite $\Ncyc\circ Q$ sends Morita
equivalences of spectral categories to Borel equivalences of
orthogonal $\bT$-spectra~\cite[5.12]{BM-tc}.  We define topological Hochschild
homology ($THH$) to be the resulting left derived functor
\[
THH \colon \Ho^{M}(\Spcat) \to \StabGB[\bT],
\]
where $\Ho^{M}(\Spcat)$ denotes the homotopy category of spectral
categories obtained by formally inverting the Morita equivalences.
The work of Shipley~\cite[4.2.8--9]{ShipleyD} and
Patchkoria-Sagave~\cite[3.8]{PatchkoriaSagave}
(cf.~\cite[3.5]{BM-tc}) shows that $THH(\aC)$ coincides with the
classical B\"okstedt construction of topological Hochschild homology
as functors to $\StabGB[\bT]$.

\begin{defn}
\term{Topological periodic cyclic homology} 
\[
TP\colon \Ho^{M}(\Spcat)\to \Stab
\]
is the composite derived functor $TP(\aC)=THH(\aC)^{t\bT}$.
\end{defn}

The smash product of spectral categories endows $\Spcat$ with a
symmetric monoidal structure.  This smash product can be left derived using
the pointwise cofibrant replacement functor $Q$ \cite[4.1]{BGT2}.  The
standard Milnor product argument makes $\Ncyc$ a strong symmetric
monoidal functor, which passes to the homotopy category (inverting
Morita equivalences) to make $THH$ a strong symmetric monoidal
functor~\cite[6.8,6.10]{BGT2}.  As an immediate consequence of this discussion
and Corollary~\ref{cor:pssymmon}, we have:

\begin{prop}
$TP$ has the canonical structure of a lax symmetric
monoidal functor $\Ho^{M}(\Spcat)\to \Stab$. 
\end{prop}

For $R$ a commutative ring orthogonal spectrum, the definition of
$\Mod_{R}$-cat\-egory is completely analogous to the definition of
spectral category, using the category of $R$-modules in
place of the category of orthogonal spectra.  We denote the category
of $\Mod_{R}$-categories as $\Spcat_{R}$ and the homotopy category
obtained by formally inverting the Morita equivalences as
$\Ho^{M}(\Spcat_{R})$.  When $R$ is cofibrant as a commutative ring
orthogonal spectrum (which we can assume without loss of
generality), $\Ncyc$ preserves Morita equivalences between pointwise
cofibrant $\Mod_{R}$-categories (where we understand cofibrant in the
sense of the standard model structure on the category of $R$-modules).  We then have the following
proposition for $TP$ of $\Mod_{R}$-categories, using the model 
$JT^{\oL}_{\bT}\Ncyc(R)$ (in Notation~\ref{notn:TwO}; see also
Theorem~\ref{thm:pssymmon}). 

\begin{prop}
Let $R$ be a cofibrant commutative orthogonal ring spectrum.  Then
$TP$ has the canonical structure of a lax symmetric monoidal functor\break
$\Ho^{M}(\Spcat_{R})\to \Ho(\Mod_{TP(R)})$.
\end{prop}

Next, let $\dgcatk$ denote the category of small $k$-linear dg categories for a
commutative ring $k$: An object consists of a small category 
enriched in chain complexes of $k$-modules and a morphism is an
enriched functor.  Choosing a cofibrant commutative ring orthogonal
spectrum $Hk$ representing the Eilenberg-Mac Lane ring spectrum, we
have a symmetric monoidal equivalence of homotopy categories
$\Ho(\dgcatk)\simeq \Ho(\SpcatHk)$ and $\Ho^{M}(\dgcatk)\simeq
\Ho^{M}(\SpcatHk)$, q.v.~\cite[\S2]{BM-tc}.  Using this equivalence, we
obtain $TP$ as a lax symmetric monoidal functor 
\[
\Ho^{M}(\dgcatk)\to \Ho(\Mod_{TP(k)})
\]
where we write $TP(k)$ for $JT^{\oL}_{\bT}(\Ncyc(Hk))$.  

\section[Proof of Theorem~\ref{main:Kunneth}]{A filtration argument (Proof of Theorem~\ref{main:Kunneth})}\label{sec:main}\label{SEC:MAIN}

The purpose of this section is to give an outline of the proof of the
main theorem, assuming the existence of a point-set model of $TP$ with
a filtration that satisfies certain properties.  Since construction of
a model satisfying these properties is technically involved, it is
useful to abstract out the key principles of the argument here,
leaving the construction and verification of properties to
Sections~\ref{sec:start}--\ref{sec:end}.  The same basic outline also
proves an analogous theorem for $C_{p}$-Tate of $THH$ in the same
context ($k$-linear dg categories for $k$ a perfect field of
characteristic $p>0$).  Given the properties of the point-set model
described below, we prove the main theorem using a comparison of
spectral sequences argument.

For the purposes of this section, let $G$ be a finite group or the
circle group $\bT$ of unit complex numbers, and we consider a point-set functor
\[
T^{M}\colon \SpGbS\to \SpbS
\]
where $\SpGbS$ denotes the category of
orthogonal $G$-spectra with a structure map from $\bS$ and likewise
$\SpbS$ denotes the category of orthogonal spectra with a structure map
from $\bS$.  We intend to apply this functor to the $THH$ of
$Hk$-categories, where any choice of object of $\aX$ induces a map of
orthogonal $\bT$-spectra $\bS\to N^{\cy}(\aX)$. We ask for $T^{M}$ to come
with the following additional structure:
\begin{enumerate}
\item\label{e:monoidal} $T^{M}$ is a lax monoidal functor.
\item\label{e:filt} $T^{M}$ comes with a natural filtration by $h$-cofibrations
\[
\cdots \to T^{M}_{-1}\to T^{M}_{0}\to T^{M}_{1}\to \cdots
\]
with $\holim T^{M}_{-n}\simeq *$ and $\colim T^{M}_{n}=T^{M}$.
\item\label{e:fm} The filtration is compatible with the monoidal structure in the
sense that the structure map $\bS\to T^{M}$ has a natural lift to $T^{M}_{0}$
and we have natural maps
\[
T^{M}_{m}(X)\sma T^{M}_{n}(Y)\to T^{M}_{m+n}(X\sma Y)
\]
for all $m,n\in \bZ$, which induce in the colimit the lax monoidal
transformation $T^{M}(X)\sma T^{M}(Y)\to T^{M}(X\sma Y)$.
\end{enumerate}
Moreover, we require that as a functor to the stable category $T^{M}$ is
naturally isomorphic to $(-)^{tG}$ by an isomorphism taking the
filtration to the Hesselholt-Madsen Tate filtration, Definition~\ref{defn:THM}
(case $G=\bT$) or~\ref{defn:HMTfilt} (case $G$ finite), and
we fix such an isomorphism.  

The Hesselholt-Madsen Tate filtration comes with a map $\pi_{*}X$ into
$E^{1}_{0,*}$; it derives from a natural transformation from the
underlying non-equivariant orthogonal spectrum $i^{*}X$ to the $0$th
filtration quotient.  We will need the following additional technical
hypothesis about the filtration $T^{M}_{*}X$ and this map.

\begin{hyp}\label{hyp:gr0}
The given natural transformation in the stable category\break
$i^{*}X \to T^{M}_{0}(X)/T^{M}_{-1}(X)$ comes from a zigzag of point-set 
monoidal natural transformations.
\end{hyp}

Specifically, we construct in Section~\ref{sec:Moore} a zigzag of
point-set monoidal functors of the form
\[
i^{*}X\overto{\simeq}RR(X)\overfrom{\simeq} TT(X)\to T^{M}_{0}X/T^{M}_{-1}X
\]
where $RR$ and $TT$ are specific point-set monoidal functors
constructed there. 

Assuming the existence of the functor $T^{M}$ with properties above,
the remainder of the section discusses and outlines a proof of the following theorem.

\begin{thm}\label{thm:main}
Let $G=\bT$ or $C_{p}$ for a prime $p$.  Let $k$ be a perfect field of characteristic
$p$, and let $X$ and $Y$ be $G$-equivariant $N^{\cy}(Hk)$-modules
under $N^{\cy}(Hk)$ with
the property that 
$\pi_{*}(X)$ and $\pi_{*}(Y)$ are finitely generated as graded modules
over $THH_{*}(k)\iso \pi_{*}(N^{\cy}(Hk))$.  Then the induced map 
\[
T^{M}(X)\smal_{T^{M}(N^{\cy}Hk)} T^{M}(Y)\to
T^{M}(X \smal_{N^{\cy}(Hk)}Y)
\]
is a weak equivalence.
\end{thm}

This implies Theorem~\ref{main:Kunneth} and its analogue for $G=C_{p}$
as follows.  Let $\aX'$ and $\aY'$ be pointwise cofibrant $Hk$-spectral
categories modeling smooth and proper $k$-linear dg categories $\aX$
and $\aY$, and
take $X$ and $Y$ in the statement of Theorem~\ref{thm:main} to be
$N^{\cy}(\aX')$ and $N^{\cy}(\aY')$, using any object of $\aX$ and
$\aY$ to obtain the
structure maps $N^{\cy}(Hk)\to N^{\cy}(\aX')$ and  $N^{\cy}(Hk)\to
N^{\cy}(\aY')$.   Although the map in
Theorem~\ref{main:Kunneth} implicitly uses the lax symmetric monoidal
structure on $T_{G}$ constructed in Section~\ref{sec:lax}, we show in
Section~\ref{sec:compare} that the monoidal structure on the model
$T^{M}$ induces the same map in the stable category.  Since
$THH_{*}(k)\iso k[t]$ (with $t$ in degree $2$) is a Noetherian ring,
it follows from Theorem~\ref{main:THHfg}, which is proved
independently in Section~\ref{sec:spthhfg}, that the $THH_{*}$ of a
smooth and proper $k$-linear dg category is finitely generated over
$THH_{*}(k)$.  The K\"unneth theorem for $THH$,
\[
N^{\cy}(\aX')\sma_{N^{\cy}(Hk)}N^{\cy}(\aY')\simeq N^{\cy}(\aX'\sma_{Hk}\aY'),
\]
q.v.~Theorem~\ref{thm:thhkun}, gives us a weak equivalence 
\[
T^{M}(X \smal_{N^{\cy}(Hk)}Y)\simeq
T^{M}(N^{\cy}(\aX'\sma_{Hk}\aY'))\simeq TP(\aX\otimes \aY).
\]
Putting this all together reduces
Theorem~\ref{main:Kunneth} to the statement of Theorem~\ref{thm:main}.

We now outline the proof of Theorem~\ref{thm:main}.
For convenience, write $A=N^{\cy}(Hk)$. Without loss of generality, we can take $X$ and
$Y$ in the statement to be cofibrant in the category of $G$-equivariant
$A$-modules under $A$; then $X\sma_{A}Y$ represents the derived smash
product $X\smal_{N^{\cy}(Hk)}Y$.  The filtration on
$T^{M}(X\sma_{A}Y)$ induces a spectral sequence, the Hesselholt-Madsen Tate spectral
sequence for $X\sma_{A}Y$.  In our work below, we call this the
\term{righthand spectral sequence} as it is the codomain in a map of
spectral sequences.

We construct the \term{lefthand spectral sequence} as follows.  As we
review in Section~\ref{sec:filtmod}, the compatibility of the filtration with the lax monoidal
structure on $T^{M}$ allow us to construct $T^{M}X\sma_{T^{M}\!A}T^{M}Y$ as a filtered
object.  Actually, we are more interested in the derived smash product
$T^{M}X\smal_{T^{M}\!A}T^{M}Y$, and we argue in Section~\ref{sec:filtmod} that we can model the
derived smash product by a filtered object where the comparison map to
$T^{M}X\sma_{T^{M}\!A}T^{M}Y$ is filtered; we choose such a model and denote it as
$T^{M}X\smal_{T^{M}\!A}T^{M}Y$.  The filtration on $T^{M}X\smal_{T^{M}\!A}T^{M}Y$ then gives the
lefthand spectral sequence.  The compatibility of the filtration with
the lax monoidal structure of $T^{M}$ implies that the map
\begin{equation}\label{eq:Hand}
T^{M}X\smal_{T^{M}\!A}T^{M}Y\to T^{M}(X\sma_{A}Y)
\end{equation}
is filtered, and that gives us a map of spectral sequences from the
lefthand spectral sequence to the righthand spectral sequence.  The
main step in the proof of Theorem~\ref{thm:main} is the following
result, which we prove in Section~\ref{sec:Hand}.

\begin{thm}\label{thm:Hand}\label{THM:HAND}
The map of spectral sequences from the lefthand spectral sequence to
the righthand spectral sequence is an isomorphism on $E^{1}$-terms.
\end{thm}

To apply standard spectral sequence comparison techniques, we need the
following two results on the convergence of the spectral sequences. 

\begin{prop}\label{prop:rhss}
The righthand spectral sequence is a half-plane
spectral sequence with entering differentials, conditionally converging to
$\pi_{*}(T^{M}(X\sma_{A}Y))$.
\end{prop}

\begin{proof}
We have $\pi_{*}A\iso THH_{*}(k)\iso k[t]$ (with $t$ in degree 2) is
connective, and by hypothesis in Theorem~\ref{thm:main}, $\pi_{*}X$ and
$\pi_{*}Y$ are finitely generated over $\pi_{*}A$; they are therefore
bounded below.  It follows that 
$X\sma_{A}Y$ is bounded below and so the explicit formula
for the $E^{1}$-term in Theorem~\ref{thm:CCT} (in the case $G=\bT$) or
Proposition~\ref{prop:CCfg} (in the case $G=C_{p}$) shows that 
$E^{1}_{i,*}$ is bounded below with bound independent of $i$.  
Conditional convergence is Lemma~\ref{lem:HMTcc} in the case $G=\bT$
and~\cite[4.3.6]{HMAnnals} in the case $G=C_{p}$.
\end{proof}

In Section~\ref{sec:cc}, we prove the following lemma.

\begin{lem}\label{lem:lhss}\label{LEM:LHSS}
The lefthand spectral sequence is conditionally convergent, converging
to $\pi_{*}(T^{M}X\smal_{T^{M}\!A}T^{M}Y)$. 
\end{lem}

Theorem~\ref{thm:Hand}, Proposition~\ref{prop:rhss}, and
Lemma~\ref{lem:lhss} together imply that the map~\eqref{eq:Hand} is a weak
equivalence (by~\cite[7.2]{Boardman-SpectralSequences}), which proves
Theorem~\ref{thm:main}.

\section{Filtered modules over filtered ring orthogonal spectra}\label{sec:filtmod}

This section begins the argument for
Theorem~\ref{thm:Hand} and Lemma~\ref{lem:lhss}.
Theorem~\ref{thm:Hand} compares a spectral sequence constructed from a
smash product of Tate filtrations to the spectral sequence of the Tate
filtration of the
smash product.  To develop the tools to do this, we study the homotopy
theory of filtered modules over filtered ring orthogonal spectra and
the derived smash product.  We use the following terminology in this
section and the next two.

\begin{defn}
A \term{filtered spectrum} $X_{*}$ consists of a sequence of maps of
orthogonal spectra
\[
\cdots \to X_{-1}\to X_{0}\to X_{1}\to \cdots,
\]
or equivalently a functor from the poset $(\bZ,{\leq})$ to the
category of orthogonal spectra.  The category of filtered spectra has
objects the filtered spectra and maps the natural transformations.  We
make filtered spectra a symmetric monoidal category using the Day
convolution~\cite[\S21]{MMSS} for the symmetric monoidal product $+$ on
$(\bZ,{\leq})$.  Let $\bS_{*}$ denote the filtered spectrum satisfying
$\bS_{n}=*$ for $n<0$ and $\bS_{n}=\bS$ and maps the identity maps for
$n\geq 0$. 
\end{defn}

The Day convolution formulation of the smash product is a shortcut to
produce the strong symmetric monoidal structure on the category, but in this
case the construction is easy to describe.  Given filtered spectra
$X_{*}$ and $Y_{*}$, the smash product $Z_{*}=X_{*}\sma Y_{*}$ is
given by the formula
\[
Z_{n}=\bigcup_{i+j=n}X_{i}\sma Y_{j}
\]
or more formally,
\begin{multline*}
Z_{n}=\lcolim_{s\to -\infty,t\to\infty}
   X_{s}\sma Y_{n-s}
    \cup_{X_{s}\sma Y_{n-s-1}}
    X_{s+1}\sma Y_{n-s-1}
      \cup_{X_{s+1}\sma Y_{n-s-2}}\cdots\\ \cdots 
      \cup_{X_{t-2}\sma Y_{n-t+1}}
      X_{t-1}\sma Y_{n-t+1}
        \cup_{X_{t-1}\sma Y_{n-t}}
        X_{t}\sma Y_{n-t}.
\end{multline*}
The smash product $Z_{*}$ comes with canonical maps $X_{i}\sma
Y_{j}\to Z_{i+j}$ and is characterized by the property that maps of
filtered spectra $Z_{*}\to W_{*}$ are in natural bijective
correspondence with systems of maps $X_{i}\sma Y_{j}\to W_{i+j}$ that
make the evident diagrams in $i$ and $j$ commute.  The spectrum
$\bS_{*}$ is the unit for the smash product. 

\begin{defn}\label{defn:ringmod}
A \term{filtered associative ring spectrum} consists of a filtered
spectrum $A_{*}$, a map $\eta\colon \bS_{*}\to A_{*}$, and a map
$\mu\colon A_{*}\sma A_{*}\to A_{*}$ satisfying the usual monoid relations (unit
and associativity diagrams).  A \term{filtered left $A_{*}$-module} (resp.,
\term{filtered right $A_{*}$-module}) consists of a filtered spectrum $M_{*}$
and a map $\xi\colon A_{*}\sma M_{*}\to M_{*}$ (resp., $\xi \colon
M_{*}\sma A_{*}\to M_{*}$) satisfying the usual action relations (unit
and associativity diagrams).
\end{defn}

Using the characterization of the smash product, we obtain an external
formulation of filtered associative ring spectra: a filtered
associative ring spectrum consists of a filtered spectrum $A_{*}$
together with a map $\eta_{0}\colon \bS\to A_{0}$ and maps 
\[
\mu_{i,j}\colon A_{i}\sma A_{j}\to A_{i+j}
\]
such that the following unit diagrams
\[
\xymatrix@C+1pc{%
\bS\sma A_{n}\ar[r]^-{\eta_{0}\sma\id}\ar[dr]_-{\iso}
&A_{0}\sma A_{n}\ar[d]^-{\mu_{0,n}}
&A_{n}\sma A_{0}\ar[d]_-{\mu_{n,0}}
&A_{n}\sma \bS\ar[l]_-{\id\sma\eta_{0}}\ar[dl]^-{\iso}\\
&A_{n}&A_{n}
}
\]
and associativity and structure map diagrams 
\[
\xymatrix@C+2pc{%
A_{i}\sma A_{j}\sma A_{k}
  \ar[r]^-{\mu_{i,j}\sma \id}\ar[d]_-{\id\sma\mu_{j,k}}
&A_{i+j}\sma A_{k}\ar[d]^-{\mu_{i+j,k}}
\hspace{-2em}&%
A_{i}\sma A_{j}\ar[d]_-{\mu_{i,j}}\ar[r]^-{a_{i,m}\sma a_{j,n}}
&A_{m}\sma A_{n}\ar[d]^-{\mu_{m,n}}
\\%
A_{i}\sma A_{j+k}\ar[r]_-{\mu_{i,j+k}}&A_{i+j+k}
\hspace{-2em}&%
A_{i+j}\ar[r]_-{a_{i+j,m+n}}&A_{m+n}
}
\]
commute, where we have written $a_{k,l}\colon A_{k}\to A_{l}$ for the
structure maps of $A_{*}$.  Filtered left and right $A_{*}$-modules
admit a similar external formulation.

\begin{example}
Let $T^{M}_{*}$ be a functor satisfying the properties laid out in
Section~\ref{sec:main}.  For $A$ an associative ring orthogonal
$\bT$-spectrum, $T^{M}_{*}\!A$ has the canonical structure of a filtered
associative ring spectrum.  If $X$ and $Y$ are right and left
$A$-modules, then $T^{M}_{*}X$ and $T^{M}_{*}Y$ have the canonical
structure of right and left filtered $T^{M}_{*}\!A$-modules.
\end{example}

The argument in Section~\ref{sec:main} uses a filtered version of the
balanced smash product, which is constructed as follows.

\begin{defn}
Let $A_{*}$ be a filtered associative ring spectrum, let $M_{*}$ be a filtered
right $A_{*}$-module and let $N_{*}$ be a filtered left
$A_{*}$-module.  Define the balanced smash product
$M_{*}\sma_{A_{*}}N_{*}$ to be the coequalizer
\[
\xymatrix{%
M_{*}\sma A_{*}\sma N_{*}\ar[r]<.75ex>\ar[r]<-.75ex>&
M_{*}\sma N_{*}\ar[r]&M_{*}\sma_{A_{*}}N_{*}
}
\]
where one map is induced by the right $A_{*}$-action on $M_{*}$ and
the other is induced by the left $A_{*}$-action on $N_{*}$.  The
coequalizer is a filtered spectrum (typically with no extra
structure). 
\end{defn}

Since smash products and coequalizers commute with sequential
colimits, we see that 
\[
\colim (M_{*}\sma_{A_{*}}N_{*})\iso (\colim M_{*})\sma_{(\colim
A_{*})}(\colim N_{*}).
\]
In other words, the balanced smash product above gives a filtration on
the balanced smash product of the underlying unfiltered modules.

A filtered spectrum naturally gives a spectral sequence on homotopy
groups.  To avoid the ambiguity of writing $E^{1}_{i,j}(X)$ for the
$E^{1}$-term, we introduce the following notation.

\begin{defn}\label{defn:piGr}
For a filtered spectrum $X_{*}$, let
$\pi^{\Gr}_{i,j}X_{*}=\pi_{i+j}C(X_{i},X_{i-1})$, where
$C(X_{i},X_{i-1})$ denotes the homotopy cofiber of the structure map
$X_{i-1}\to X_{i}$.
\end{defn}

It is also useful to work directly with the orthogonal spectra
$C(X_{n},X_{n-1})$ whose homotopy groups represent $\pi^{\Gr}_{n,*}$,
the associated graded spectra of the filtration.  We follow the
traditional route of defining a point-set functor $\Gr$ with good structural
properties, and regard $C(X_{n},X_{n-1})$ as a model for the left
derived functor.  Before defining $\Gr$, we define the target
category. 

\begin{defn}
The category of graded spectra is the category of functors from the
discrete category $\bZ$ to orthogonal spectra: A \term{graded
spectrum} $X_{*}$ consists of a sequence of orthogonal spectra 
$X_{n}$ for each $n\in \bZ$; a map of graded spectra $X_{*}\to Y_{*}$
consists of a sequence of maps of orthogonal spectra  $X_{n}\to Y_{n}$
for $n\in \bZ$.  The category of graded spectra becomes a symmetric
monoidal category using the Day convolution for the symmetric monoidal
product $+$ on $\bZ$:
\[
(X_{*}\sma Y_{*})_{n}=\bigvee_{i+j=n}X_{i}\sma Y_{j}.
\]
\end{defn}

\begin{defn}
The \term{associated graded} functor $\Gr$ from filtered spectra to
graded spectra is defined by 
\[
(\Gr X_{*})_{n}:=\Gr_{n}(X_{*}):=X_{n}/X_{n-1}.
\]
\end{defn}

We emphasize that in the definition, $X_{n}/X_{n-1}$ denotes the
point-set quotient and not the homotopy cofiber.  The functor $\Gr$
has a right adjoint $Z$ that makes a graded spectrum into a filtered
spectrum using the trivial map for structure maps.  In particular
$\Gr$ preserves colimits.  It also clearly preserves smash products:

\begin{prop}\label{prop:Grmon}
$\Gr$ is a strong symmetric monoidal functor.
\end{prop}

We define a \term{graded associative ring spectrum} and \term{graded
left }and\term{ right modules} over a graded associative ring spectrum
in terms of the smash product with unit maps, multiplication, and
action maps in the usual way for graded spectra just as we did for
filtered spectra in Definition~\ref{defn:ringmod}.  The previous
proposition implies that $\Gr$ takes filtered associative ring spectra
to graded associative ring spectra and filtered left and right modules
over a filtered associative ring spectrum $A_{*}$ to graded left and
right modules over the graded associative ring spectrum $\Gr A_{*}$.
The previous proposition and the fact that $\Gr$ preserves colimits 
then implies that $\Gr$ preserves the balanced smash product.

\begin{prop}
Let $A_{*}$ be a filtered associative ring spectrum, let $M_{*}$ be a filtered
right $A_{*}$-module and let $N_{*}$ be a filtered left
$A_{*}$-module.  Then there is a canonical natural isomorphism 
\[
\Gr(M_{*}\sma_{A_{*}}N_{*})\iso (\Gr M_{*})\sma_{(\Gr A_{*})}(\Gr N_{*}).
\]
\end{prop}

We now turn to the homotopy theory of filtered and graded spectra. To
avoid confusion we do not use the unmodified phrase ``weak
equivalence'' in the context of filtered and graded spectra.  We always write
\term{objectwise weak equivalence} for a map $X_{*}\to Y_{*}$ that is
a weak equivalence $X_{n}\to Y_{n}$ for all $n$.  We call a map of
filtered spectra $X_{*}\to Y_{*}$ a
\term{total weak equivalence} when it induces a weak equivalence
\[
\lhocolim_{n\to \infty}X_{n}\to \lhocolim_{n\to\infty} Y_{n}.
\]
Objectwise weak equivalences are total weak equivalences, but not vice
versa.  However, the following result provides a useful converse under
an additional hypothesis.

\begin{prop}\label{prop:gradedequiv}
Let $f\colon X_{*}\to Y_{*}$ be a map of filtered spectra that induces an
isomorphism on $\pi^{\Gr}_{*,*}$.  Then $f$ is a total weak equivalence if
and only if it is an objectwise weak equivalence.
\end{prop}

\begin{proof}
We only need to show the direction that assumes $f$ is a total weak
equivalence; the converse is clear.
Write $X_{\infty}$ for $\hocolim X_{n}$ and $Y_{\infty}$ for $\hocolim
Y_{n}$.  The hypothesis is then that the map $C(X_{n+1},X_{n})\to
C(Y_{n+1},Y_{n})$ is a weak equivalence for all $n$; by induction, we see that
$C(X_{n+i},X_{n})\to C(Y_{n+i},Y_{n})$ is a weak equivalence for all
$n,i$, and passing to the homotopy colimit (in $i$), we see that $C(X_{\infty},X_{n})\to
C(Y_{\infty},Y_{n})$ is a weak equivalence for all $n$.  When $f$ is a
total weak equivalence, then $X_{\infty}\to Y_{\infty}$ is a weak equivalence,
and it follows that $X_{n}\to Y_{n}$ is a weak equivalence for all $n$.  
\end{proof}

Standard results~\cite[11.6.1]{Hirschhorn},\cite[4.1]{SSAlgMod}
provide a closed model structure on filtered spectra for the
objectwise weak equivalences.

\begin{prop}\label{prop:filt1}
The category of filtered spectra has a compactly generated topological
model structure in the sense of~\cite[5.9,5.12]{MMSS} with weak
equivalences and fibrations defined objectwise.
\end{prop}

\begin{prop}\label{prop:filt2}
The categories of filtered associative ring spectra and filtered left
and right modules over a filtered associative ring spectrum have
compactly generated topological closed model category structures with weak
equivalences and fibrations defined objectwise.
\end{prop}

We call the homotopy categories of the preceding model structures the
\term{filtered derived category} of left and right $A_{*}$-modules.
We have a derived smash product of an $A_{*}$-module and a filtered
spectrum, denoted $\smal$, which may be constructed by cofibrant
approximation of either object.  Using this, we have an evident notion
of homotopical module in the derived category.  For us, the case that
is most useful is that of homotopical left $B_{*}$-modules in the
filtered derived category of right $A_{*}$-modules, so we make this
definition explicitly; other types of homotopical modules are defined
analogously.

\begin{defn}
Let $A_{*}$ and $B_{*}$ be a filtered associative ring spectra.  A \term{homotopical
left $B_{*}$-module} in the filtered derived category of right
$A_{*}$-modules consists of a right $A_{*}$-module $N_{*}$ and a map
\[
\xi\colon  B_{*}\smal N_{*}\to N_{*}
\]
in the filtered derived category of right $A_{*}$-modules 
satisfying the usual unit and associativity conditions.
\end{defn}

The next proposition, a special case of~\cite[8.2]{LewisMandell2},
studies the homotopy theory of the balanced smash product.  It implies
in particular that the filtered derived smash product may be computed
by deriving either variable.

\begin{prop}\label{prop:filt3}
Let $A_{*}$ be a filtered associative ring spectrum.
The left derived bifunctor $\Tor^{A_{*}}(-,-)$ of the balanced
product $(-)\sma_{A_{*}}(-)$ exists and can be constructed by
cofibrant replacement of either variable.  In particular, for each
right module $M_{*}$ and 
left module $N_{*}$, $\Tor^{A_{*}}(-,N_{*})$ is the left
derived functor of $(-)\sma_{A_{*}}N_{*}$ and $\Tor^{A_{*}}(M_{*},-)$
is the left derived functor of $M_{*}\sma(-)$.
\end{prop}

Given this proposition, there is no source of confusion for which
derived functor $\smal$ denotes; we now switch to writing
$M_{*}\smal_{A_{*}}N_{*}$ in place of $\Tor^{A_{*}}(M_{*},N_{*})$.

As a consequence of the proposition, if $M_{*}$ has the structure of
a homotopical left $A_{*}$-module in the category of filtered right
$A_{*}$-modules, then $M_{*}\smal_{A_{*}}(-)$ has the natural
structure of a homotopical left $A_{*}$-module in filtered spectra.

For Propositions~\ref{prop:filt1}--\ref{prop:filt3}, analogous
statements hold in context of graded spectra with easier proofs. As
discussed above, the functor $\Gr$ from filtered spectra to graded
spectra has a right adjoint, and from the description above, it is
clear that the right adjoint preserves objectwise fibrations and objectwise weak
equivalences.  Thus, $\Gr$ is a Quillen left adjoint; its left derived
functor $\LGr$ exists and may be constructed by applying $\Gr$ to a
cofibrant replacement. Alternatively, as per the motivation for
introducing $\Gr$, we can use a homotopy cofiber construction.  For
work below it is useful to have a wide class of objects where the
point-set functor $\Gr$ models the derived functor.  We introduce the
following terminology.

\begin{defn}
A filtered spectrum $X_{*}$ is \term{reasonably filtered} when the
structure maps $X_{n}\to X_{n+1}$ are all $h$-cofibrations.
\end{defn}

In what follows, ``reasonably filtered'' will always refer to the
underlying filtered spectrum.  We elide ``reasonably filtered
filtered\dots'' to ``reasonably filtered\dots'' for filtered spectra,
filtered associative ring spectra, and filtered left and right
modules.

\begin{prop}
$\Gr$ preserves objectwise weak equivalences between reasonably filtered spectra.
\end{prop}

Cofibrant objects in the model category of filtered spectra are
reasonably filtered, and so $\Gr$ computes the derived functor $\LGr$
on all reasonably filtered spectra.  When $A_{*}$ is a reasonably
filtered associative ring spectrum, cofibrant left and right filtered
$A_{*}$-modules are reasonably filtered, and so the functors $\Gr$
from filtered left and right $A_{*}$-modules to graded left and right
$\Gr A_{*}$-modules have left derived functors, calculated by applying
$\Gr$ to reasonably filtered replacements.  We also have the following
observation about the balanced smash product.

\begin{prop}
Let $A_{*}$ be a reasonably filtered associative ring spectrum, let
$M_{*}$ be a cofibrant right $A_{*}$-module and let $N_{*}$
be a cofibrant left $A_{*}$-module.  Then the balanced
smash product $M_{*}\sma_{A_{*}}N_{*}$ is a reasonably filtered
spectrum. 
\end{prop}

Boardman~\cite[5.10]{Boardman-SpectralSequences} defines conditional convergence of
spectral sequences in terms
of $\lim$ and $\lim^{1}$.  In the current context of filtered spectra,
the spectral sequence associated to the filtration on $X_{*}$ is
conditionally convergent if and only if $\holim_{n} X_{-n}\simeq *$.
With this in mind, we make the following definition.

\begin{defn}\label{defn:condconv}
A filtered spectrum $X_{*}$ is \term{conditionally convergent} when\break $\holim_{n}
X_{-n}\simeq *$ 
\end{defn}

A first easy observation about conditional convergence of filtered
spectra is that it is invariant under objectwise weak equivalences, and so
conditional convergence may be studied in the homotopy category of
filtered spectra. 

Because homotopy limits commute with cofiber sequences and with other
homotopy limits, the following propositions are clear.

\begin{prop}
If $X_{*}$ and $Y_{*}$ are conditionally convergent, then for any map
$X_{*}\to Y_{*}$, the homotopy cofiber is conditionally convergent.
\end{prop}

\begin{prop}\label{prop:holim}
If $d\mapsto X_{*}(d)$ is a small diagram and each $X_{*}(d)$ is
conditionally convergent, then $\holim_{d}X_{*}(d)$ is conditionally
convergent. 
\end{prop}

For a filtered spectrum $X_{*}$, the suspension $\Sigma^{m}X_{*}$ is
defined objectwise,\break $(\Sigma^{m}X_{*})_{n}=\Sigma^{m}X_{n}$, where we
understand $\Sigma^{m}X=X\sma F_{-m}S^{0}$ when $m<0$.  (Here
$F_{-m}S^{0}$ is a particular cofibrant model of $S^{m}$; see
\cite[1.3]{MMSS}).  The shift $X_{*}[t]$ is defined by
$(X_{*}[t])_{n}=X_{n-t}$.  The following proposition is also clear.

\begin{prop}
If $X_{*}$ is conditionally convergent, then so is any suspension and shift.
\end{prop}

For a filtered associative ring spectrum $A_{*}$, a finite cell
filtered right $A_{*}$-module is a filtered right $A_{*}$-module that
can be built in finitely many stages using cofiber sequences involving
suspensions and shifts.

\begin{prop}\label{prop:smashfinite}
If $N_{*}$ is a conditionally convergent left $A_{*}$-module and
$M_{*}$ is a finite cell filtered right $A_{*}$-module, then
$M_{*}\smal_{A_{*}}N_{*}$ is conditionally convergent.
\end{prop}

\begin{proof}
$M_{*}\smal_{A_{*}}N_{*}$ is built in finitely many stages using
cofiber sequences involving suspensions and shifts of $N_{*}$.
\end{proof}

\section[Proof of Theorem~{\ref{thm:Hand}}]{Comparison of the lefthand and righthand spectral sequences\sbreak (Proof of Theorem~\ref{thm:Hand})}\label{sec:Hand}

This section is devoted to the proof of Theorem~\ref{thm:Hand}.  We
prove a slightly more general result: We assume that that $A$ is a
commutative ring orthogonal $\bT$-spectrum, and $X$ and $Y$ are
cofibrant $A$-modules under $A$.  (We can also take $G$ to be any
closed subgroup of $\bT$.)

Let $TY'_{*}$ be a filtered $T^{M}_{*}\!A$-module cofibrant
replacement for $T^{M}_{*}Y$.  In particular $\Gr TY'_{*}\to
\Gr T^{M}_{*}Y$ is a weak equivalence.
Theorem~\ref{thm:Hand} asserts that the filtered map
\begin{equation}\label{eq:Hand2}
T^{M}_{*}X\sma_{T^{M}_{*}\!A}TY'_{*}\to T^{M}_{*}(X\sma_{A}Y)
\end{equation}
induces a weak equivalence on $\pi^{\Gr}_{*,*}$.
Note that this is a map of homotopical left $T^{M}_{*}\!A$-modules
and hence the induced map on $\pi^{\Gr}_{*,*}$ is a map of left
$\pi^{\Gr}_{*,*}T^{M}_{*}\!A$-modules.  Theorem~\ref{thm:CCT} (in the
case $G=\bT$) or Proposition~\ref{prop:CCfg} (in the case $G$ is
finite) gives an isomorphism of bigraded rings
\[
\pi^{\Gr}_{*,*}T^{M}_{*}\!A\iso \CC_{*,*}\otimes \pi_{*}A.
\]
In particular~\eqref{eq:Hand2} is a map of left
$\CC_{*,*}$-modules. 

By Hypothesis~\ref{hyp:gr0}, we have a zigzag of monoidal functors,
which as indicated takes the form
\[
i^{*}\overto{\simeq} RR\overfrom{\simeq}TT\to \Gr_{0}T^{M}_{*},
\]
where $i^{*}$ denotes the forgetful functor to non-equivariant
orthogonal spectra.
Choosing cofibrant replacements $TTY'\to TT(Y)$ and $RRY'\to
RR(Y)$ in 
the categories of left $TT(A)$-modules and left $RR(A)$-modules,
respectively, we can then choose lifts $TTY'\to \Gr_{0}TY'$ and $TTY'\to
RRY'$ (in left $TT(A)$-modules) and $i^{*}Y\to RRY'$ (in left
$i^{*}A$-modules).  Since the natural transformations are monoidal, we
get a commutative diagram of graded spectra
\[
\xymatrix@R-1pc{%
i^{*}X\sma_{i^{*}A}i^{*}Y\ar[r]^{=}\ar[d]_-{\simeq}
&i^{*}X\sma_{i^{*}A}i^{*}Y\ar[d]^-{\simeq}\\
RR(X)\sma_{RR(A)}RRY'\ar[r]&RR(X\sma_{A}Y)\\
TT(X)\sma_{TT(A)}TTY'\ar[r]\ar[u]^-{\simeq}\ar[d]
&TT(X\sma_{A}Y)\ar[u]_-{\simeq}\ar[d]\\
\Gr  T^{M}_{*}X\sma_{\Gr T^{M}_{*}\!A}\Gr TY'\ar[r]
&\Gr (T^{M}_{*}(X\sma_{A}Y))
}
\]
where we view the ungraded spectra as concentrated in degree $0$.  On
the righthand side, since 
\[
\pi^{\Gr}_{*,*}(T^{M}_{*}(X\sma_{A}Y))\iso
\pi_{*,*}\Gr(T^{M}_{*}(X\sma_{A}Y))
\]
is a left
$\CC_{*,*}$-module, we get an induced map of $\CC_{*,*}$-modules
\[
\CC_{*,*}\otimes \pi_{*}(i^{*}X\sma_{i^{*}A}i^{*}Y)\to
\pi^{\Gr}_{*,*}(T^{M}(X\sma_{A}Y))
\]
which by construction (and Hypothesis~\ref{hyp:gr0}) is the usual
isomorphism.  We are now reduced to proving the following lemma.

\begin{lem}\label{lem:EMSS}\label{LEM:EMSS}
The induced map of bigraded abelian groups
\[
\CC_{*,*}\otimes \pi_{*}(i^{*}X\sma_{i^{*}A}i^{*}Y)\to 
\pi^{\Gr}_{*,*}(T^{M}_{*}X\sma_{T^{M}_{*}\!A}TY')
\]
is an isomorphism.
\end{lem}

\begin{proof}
We have trigraded Eilenberg-Moore spectral sequences to compute both
sides. On the left we take the tensor of the torsion free bigraded
abelian group $\CC_{*,*}$ with the Eilenberg-Moore spectral sequence
for the smash product $i^{*}X\sma_{i^{*}A}i^{*}Y$; this has $E^{2}$-term
\[
E^{2}_{i,j,*}=\CC_{j,*}\otimes \Tor^{\pi_{*}A}_{i,*}(\pi_{*}X,\pi_{*}Y).
\]
On the right the usual Eilenberg--More spectral sequence for the
balanced smash product of associated graded modules is naturally
trigraded with $E^{2}$-term
\[
E^{2}_{i,*,*}=\Tor^{\pi^{\Gr}_{*,*}A}_{i}(\pi^{\Gr}_{*,*}X,\pi^{\Gr}_{*,*}Y)
\iso 
\Tor^{\CC_{*,*}\otimes \pi_{*}A}_{i,*,*}(\CC_{*,*}\otimes \pi_{*}X,\CC_{*,*}\otimes \pi_{*}Y).
\]
Since the map is induced by maps of rings and modules, we get a
homomorphism of spectral sequences.  Since the isomorphism
$\CC_{*,*}\otimes \pi_{*}Z\iso \pi^{\Gr}_{*,*}Z$ of
Theorem~\ref{thm:CCT} or Proposition~\ref{prop:CCfg} is induced by the
same map $\pi_{*}Z\to \pi^{\Gr}_{*,*}T^{M}Z$ as in
Hypothesis~\ref{hyp:gr0}, the induced map on $E^{2}$-terms
\[
\CC_{*,*}\otimes \Tor^{\pi_{*}A}_{*,*}(\pi_{*}X,\pi_{*}Y)
\to 
\Tor^{\CC_{*,*}\otimes \pi_{*}A}_{*,*,*}(\CC_{*,*}\otimes \pi_{*}X,\CC_{*,*}\otimes \pi_{*}Y)
\]
is the evident isomorphism.
\end{proof}

\section[Proof of Lemma~{\ref{lem:lhss}}]{Conditional convergence of the lefthand spectral sequence\sbreak (Proof of Lemma~\ref{lem:lhss})}\label{sec:cc}

We employ the terminology of Section~\ref{sec:filtmod}.  In this
terminology, Lemma~\ref{lem:lhss} is precisely the assertion that the
filtered spectrum $T^{M}_{*}X\sma_{T^{M}_{*}\!A}TY'$ is conditionally
convergent.  The proof relies properties specific
to perfect fields of finite characteristic.

As indicated in Section~\ref{sec:filtmod}, our homotopical work is in
the Borel equivariant stable category, where we can specify an object
as an orthogonal spectrum indexed on $\{\bR^{n}\}$ with a (point-set)
$\bT$-action.  We denote by $i_{*}Hk$ the Eilenberg-Mac~Lane spectrum
$Hk$ (indexed on $\{\bR^{n}\}$) with
the trivial action.  For any fixed model of $Hk$ as a commutative ring
orthogonal spectrum, the augmentation map $N^{\cy}(Hk)\to i_{*}Hk$ is
a point-set
map of commutative ring orthogonal spectra with $\bT$-action, in
particular making $i_{*}Hk$ an $N^{\cy}(Hk)$-module.  On the point-set
level the map is induced by the multiplication map
$Hk\sma \dotsb \sma Hk\to Hk$, or viewing $N^{\cy}(Hk)$ as the tensor
$i_{*}Hk\otimes \bT$ in the point-set category of commutative ring orthogonal
spectra with $\bT$-action, it is induced by the map of $\bT$-spaces
$\bT\to \bT/\bT$.  The following proposition is specific to the case of
perfect fields of finite characteristic.

\begin{prop}\label{prop:kfinite}
Let $k$ be a perfect field of finite characteristic.  In the
Borel derived category of left $N^{\cy}(Hk)$-modules, $i_{*}Hk$ is finite.
\end{prop}

\begin{proof}
We have $THH_{*}(k)\iso k[t]$ (with $t$ in degree $2$), so it suffices
to show that there exists a map in the Borel stable category
$\Sigma^{2}\bS\to THH(k)$ sending the fundamental class to $t$, or
equivalently, that $t$ is in the image of the map
$\pi^{h\bT}_{*}THH(k)\to \pi_{*}THH(k)$.  The homotopy fixed point
spectral sequence is conditionally convergent in this case and
concentrated in even degrees, and so strongly convergent with
$E_{2}=E_{\infty}$.  In particular, the map $\pi^{h\bT}_{*}THH(k)\to
\pi_{*}THH(k)$ is surjective.
\end{proof}

Throughout this section, when we refer to $i_{*}Hk$ as an equivariant
$N^{\cy}(Hk)$-module, we always mean the structure in
Proposition~\ref{prop:kfinite}.  In the argument for
Lemma~\ref{lem:lhss}, we will use Postnikov towers built from $i_{*}Hk$.

\begin{prop}\label{prop:Post}
Let $A=N^{\cy}(Hk)$ and let $X$ be an equivariant $A$-module.  If
$\pi_{n}X=0$ for $n<N$ then there exists a tower of equivariant
$A$-modules 
\[
\cdots \to X_{m+1}\to X_{m} \to \cdots \to X_{N}\to X_{N-1}=*
\]
and a map of equivariant $A$-modules from $X$ to the system such that:
\begin{enumerate}
\item The map $X\to \holim X_{m}$ is a Borel equivalence.
\item Each homotopy fiber $Fib(X_{m+1}\to X_{m})$ is Borel
equivalent as an equivariant $A$-module to a wedge of copies of $\Sigma^{m+1} i_{*}Hk$.
\end{enumerate}
\end{prop}

\begin{proof}
By way of notation, recall from Section~\ref{sec:spec} that $\dR
F_{A}(-,-)$ denotes the equivariant derived mapping spectrum for the derived
category of left $A$-modules, an object of the $\bT$-equivariant
stable category.  In contrast, $\dR^{B}F^{\bT}_{A}(-,-)$ denotes the derived
mapping spectrum in the Borel derived category of left $A$-modules, an
object of the (non-equivariant) stable category, and the relationship
is that $\dR^{B}F^{\bT}_{A}(-,-)$ is the homotopy fixed point spectrum
of $\dR F_{A}(-,-)$.   We then have isomorphisms 
\begin{multline*}
\Ho^{B}(\Mod_{A})(X,\Sigma^{n}i_{*}Hk)\iso
\pi_{-n}\dR^{B}F^{\bT}_{A}(X,i_{*}Hk)
\iso \pi^{h\bT}_{-n}\dR F_{A}(X,i_{*}Hk)\\
\iso \pi^{h\bT}_{-n}\dR F_{i_{*}Hk}(i_{*}Hk\smal_{A}X,i_{*}Hk),
\end{multline*}
the last isomorphism induced by change of scalars.  By the Hurewicz
theorem for (non-equivariant) $A$-modules and the hypothesis that
$\pi_{n}X=0$ for $n<N$, we have that the map $\pi_{N}X\to
\pi_{N}(i_{*}Hk\smal_{A}X)$ is an isomorphism and $\pi_{n}(i_{*}Hk\smal_{A}X)=0$
for $n<N$.  We see that 
$\pi_{n}\dR F_{i_{*}Hk}(i_{*}Hk\smal_{A}X,i_{*}Hk)=0$
for $n>-N$, and so from the homotopy fixed point spectral sequence, we
deduce that the map
\begin{multline*}
\pi^{h\bT}_{-N}\dR F_{i_{*}Hk}(i_{*}Hk\smal_{A}X,i_{*}Hk)\to
\pi_{-N}\dR F_{i_{*}Hk}(i_{*}Hk\smal_{A}X,i_{*}Hk)\\
\iso\pi_{-N}\dR F_{Hk}(Hk\smal_{i^{*}A}i^{*}X,Hk)
\end{multline*}
is an isomorphism.  This constructs enough maps $X\to \Sigma^{N}i_{*}Hk$ to
produce a map $X\to \bigvee \Sigma^{N}i_{*}Hk$ that induces an isomorphism
on $\pi_{N}$.  The remainder of the construction is parallel to the
usual construction of the Postnikov tower, using the techniques above
to construct the maps.
\end{proof}

We are ready to prove Lemma~\ref{lem:lhss}.

\begin{proof}[Proof of Lemma~\ref{lem:lhss}]
We use the Postnikov tower construction of Proposition~\ref{prop:Post}
applied to $Y$.  By construction, each $Y_{m}$ is built in finitely
many fiber sequences from a wedge of copies of $\Sigma^{n}i_{*}Hk$.  These are finite
wedges since $\pi_{n}X$ is finitely generated over $k$.  We also know
from Proposition~\ref{prop:kfinite} that $i_{*}Hk$ is finite in the Borel
derived category of $A$-modules.  It follows 
that $T^{M}_{*}Y_{m}$ is a finite filtered $T^{M}_{*}\!A$-module, and by
Proposition~\ref{prop:smashfinite},
$T^{M}_{*}X\smal_{T^{M}_{*}\!A}T^{M}_{*}Y_{m}$ 
is conditionally convergent.  

Proposition~\ref{prop:filt3} allows us to choose a cofibrant filtered
$T^{M}_{*}\!A$-module replacement $TX'_{*}\to T^{M}_{*}X$ so that
$TX'_{*}\sma_{T^{M}_{*}\!A}(-)$ models
$T^{M}X\smal_{T^{M}_{*}\!A}(-)$.  Proposition~\ref{prop:holim} now
implies that
\[
\holim_{m} (TX'_{*}\sma_{T^{M}_{*}\!A}T^{M}_{*}Y_{m})
\]
is conditionally convergent.  We have a map of filtered spectra
\[
TX'_{*}\sma_{T^{M}_{*}\!A}T^{M}Y_{*}\to  
\holim_{m} (TX'\sma_{T^{M}_{*}\!A}T^{M}_{*}Y_{m})
\]
that induces an isomorphism on $\pi^{\Gr}_{*,*}$ since $\pi_{*}X$ is
bounded below (for example, see Lemma~\ref{lem:EMSS}).  The proof is
completed by showing that this map is an objectwise weak equivalence.
We do this by applying Proposition~\ref{prop:gradedequiv} and showing
that the map is a total weak equivalence.  To see that the map is a
total weak equivalence, it suffices to observe that $TX'$ is small as
a $T^{M}\!A$-module.  This is the content of Theorem~\ref{main:finite}
in the case of main interest when $X=N^{\cy}(\aX')$ and $G=\bT$, and
follows in the current generality by its generalization,
Theorem~\ref{thm:genfinite} (together with
Proposition~\ref{prop:finite}), proved in Section~\ref{sec:tpfg}. 
\end{proof}

\section{Constructing the filtered model: The positive filtration}
\label{sec:PosFilt}\label{sec:start}

In this section, we start the construction of the filtered functor
$T^{M}$ outlined in Section~\ref{sec:main}.  As in the construction of
the Hesselholt-Madsen Tate filtration in Section~\ref{sec:TSS}, we
construct the integrally graded filtration from a positive filtration,
arising from a filtration on $\tEG$, and a negative filtration,
arising from a filtration on $EG$.  In both cases, the technical
work is to construct the multiplicative structure we require.  This
section handles the work for the positive filtration and the next one
the work for the negative filtration.

The positive filtration on the Tate fixed points arises from the
$G$-cellular filtration on $\tEG$.  It should have $S^{0}$ in
filtration level zero and free $G$-cells in every positive degree in
the pattern specified by the cell structure of the standard model
discussed in Section~\ref{sec:TSS}.  Recall that when $G=\bT$, we
double the natural filtration degrees as explained there.
However, for the purposes of this section, we
will work with the natural filtration degrees in order to give a
uniform treatment, introducing a different notation to avoid confusion.

To construct a multiplicative version of this filtration, we use a
model $\tEOT{\oO}$ for an $A_{\infty}$ operad $\oO$ (with $\oO(0)=*$)
and construct a 
new filtration, the ``pseudocellular filtration'', which is coarser
than the homogeneous filtration~\eqref{eq:firstfiltration}.  Here we
understand an $A_{\infty}$ operad to have spaces the homotopy type of
CW complexes, the identity element to be a non-degenerate basepoint,
and to come with a weak equivalence of operads $\oO\to \Ass$, where
$\Ass$ denotes the associative operad $\Ass(n)=\Sigma_{n}$.  We have a
corresponding non-$\Sigma$ operad $\ooO$ where $\ooO(n)$ is the
component of $\oO(n)$ lying over the identity permutation in
$\Sigma_{n}$; then $\oO$ is canonically isomorphic to the induced
operad ($\oO(n)\iso \ooO(n)\times \Ass(n)$).  In this notation,
$\tEOG{\oO}$ is formed by starting with $S^{0}$ and iteratively gluing
on the cells $\ooO(n)\times G^{n}\times I^{n}$.

\begin{defn}\label{defn:secondfiltration}
Let $\oO$ be an $A_{\infty}$ operad with $\oO(0)=*$.
We define the \term{pseudocellular filtration}, an increasing filtration on $\tEOG{\oO}$, as
follows.  We put $S^{0}$ in filtration level $0$.  For an element
$(g_{1},\ldots,g_{n})\in G^{n}$, let 
\[
q(g_{1},\ldots,g_{n})=n-(\delta(g_{1},g_{2})+\cdots +\delta(g_{n-1},g_{n})),
\]
where $\delta(g,h)=1$ if $g=h$ and $0$ otherwise.  An
element of $\tEOG{\oO}$ in the image of 
\[
a,(g_{1},\ldots,g_{n}),(t_{1},\ldots,t_{n})\in \ooO(n)\times G^{n}\times I^{n}
\]
is placed in filtration level $q(g_{1},\ldots,g_{n})$.  For the
purposes of this section, we
write $\tEOG{\oO}^{pc}_{n}$ for the subspace in pseudocellular
filtration level $n$.  For $G$, finite, we define
$\tEOG{\oO}_{n}=\tEOG{\oO}^{pc}_{n}$, and for $G=\bT$, we define
$\tEOG{\oO}_{2n+1}=\tEOG{\oO}_{2n}=\tEOG{\oO}^{pc}_{n}$.
\end{defn}

Although different representatives of the same point may have
different $q$-values, the filtration is well-defined: The point lies
in all filtration levels of all representatives and all higher
levels.  The function $q$ subtracts from the homogeneous degree the
number of consecutive repeats; an element in homogeneous
filtration level $n$ is also in pseudocellular filtration level $n$
(and possibly lower).  Since $G$ acts diagonally on $G^{n}$, the
filtration is $G$-equivariant.  The filtration is multiplicative in
the following sense.

\begin{prop}\label{prop:posfiltmult}
The operad action maps $\ooO(m)_{+}\sma \tEOG{\oO}^{(m)}\to
\tEOG{\oO}$ preserve filtration using the pseudocellular filtration on
$\tEOG{\oO}$ and the smash power of the pseudocellular
filtration on $\tEOG{\oO}^{(m)}$.
\end{prop}

\begin{proof}
Representing an element $x$ of $\tEOG{\oO}^{(m)}$ as 
\[
a^{i},(g^{i}_{1},\ldots,g^{i}_{n_{i}}),(t_{1},\ldots,t_{n_{i}}) 
\]
for $i=1,\ldots,m$, and given an element $a\in \ooO(m)$, the
composition takes $a,x$ to an element $y$ of $\tEOG{\oO}$ that is
represented by 
\[
b,(h_{1},\ldots,h_{n}),(u_{1},\ldots,u_{n})
\]
where $b=a\circ (a_{1},\ldots,a_{i})$, $n=n_{1}+\cdots+n_{m}$, and the
$h_{k}$'s and $u_{k}$'s are the lists obtained by flattening the
arrays of $g^{i}_{j}$'s and $t^{i}_{j}$'s (respectively)
lexicographically with the lower index first.
The element $x$ is in filtration level
\[
\sum_{i=1}^{m} q(g_{1},\ldots,g_{n_{i}})
=\sum_{i=1}^{m}\left(n_{i}-\sum_{j=1}^{n_{i}-1}\delta(g^{i}_{j},g^{i}_{j+1})\right)
\]
which is at least as big as 
\[
\sum_{i=1}^{m}\left(n_{i}-\sum_{j=1}^{n_{i}-1}\delta(g^{i}_{j},g^{i}_{j+1})\right)-\sum_{i=1}^{m-1}\delta(g^{i}_{n_{i}},g^{i+1}_{1})
=n-\sum_{k=1}^{n-1}\delta(h_{k},h_{k+1}), 
\]
which filtration level contains $y$.
\end{proof}

To show that this filtration is equivalent to the standard one,
consider the standard (simplicial) filtration on $EG$ coming from the
bar construction $EG=B(G,G,*)$ where
\[
B_{n}(G,G,*)=G\times (G^{n}).
\]
We have the same filtration when we look at the $W$-construction, the
geometric realization of the simplicial set
\[
WG_{n}=G^{n+1},
\]
where the faces are induced by projection maps 
\[
d_{i}(g_{0},\ldots,g_{n})=(g_{0},\ldots,\hat g_{i},\ldots,g_{n})
\]
and the degeneracies by
diagonal maps 
\[
s_{i}(g_{0},\ldots,g_{n})=(g_{0},\ldots,g_{i-1},g_{i},g_{i},g_{i+1},\ldots,g_{n}).
\]
This has the diagonal $g$-action, and the map $WG\subdot\to B\subdot(G,G,*)$ defined by 
\[
g_{0},\ldots,g_{n}\mapsto g_{0},(g_{0}^{-1}g_{1},\ldots,g_{n-1}^{-1}g_{n})
\]
is a $G$-equivariant simplicial isomorphism.  We get another model
closer to the operadic construction by
regarding $WG\subdot$ as a $\Delta$-set, forgetting the degeneracies.
Let $WG^{\Delta }$ denote the $\Delta$-set geometric realization,
where we iteratively glue $WG_{n}\times \Delta[n]$ along $WG_{n}\times
\partial \Delta[n]$ without collapsing the degeneracies.  The $\Delta$-set
geometric realization gives a filtration but not one that corresponds
to the filtration on $WG$: The map $WG^{\Delta}\to WG$ is filtered and
an equivariant weak equivalence, but the filtration pieces do not map by weak
equivalences.  We can fix this by enlarging the filtration making use
of the forgotten degeneracies: We define the pseudocellular filtration
on $WG^{\Delta}$ by taking $WG^{\Delta}_{n}$ to consist of all
elements in the image of $WG_{m}\times \Delta[m]^{\Delta}$ for $m\leq
n$, where $\Delta[m]^{\Delta}$ denotes the $\Delta$-set geometric
realization of $\Delta[m]$.  Using this filtration, the map
$WG^{\Delta}\to WG$ is an equivariant weak equivalence on each
filtration level. 

\begin{defn}
Let $\tWG$ be the cofiber of the map $WG^{\Delta}_{+}\to S^{0}$
collapsing $WG^{\Delta}$ to the non-basepoint.  Let $\tWG_{n}$ be the
cofiber of $(WG^{\Delta}_{n-1})_{+}\to S^{0}$ for the pseudocellular
filtration on $WG^{\Delta}$, where we understand $WG^{\Delta}_{-1}$ to
be the empty set.
\end{defn}

The equivariant weak equivalence $WG^{\Delta}\to EG$ induces an
equivariant weak equivalence $\tWG\to \tEG$, and the following
proposition is clear by construction. 

\begin{prop}
The canonical map $\tWG\to \tEG$ is filtered for the standard
filtration on $\tEG$ and an equivariant weak equivalence on each
filtration level.
\end{prop}

Next we define a map $\tEOG{\oO}\to \tWG$.
Concretely, $\tWG$ is built from $S^{0}$ by iteratively attaching
$G^{n}\times \Delta[n-1]\times I$ along $G^{n}\times \partial
(\Delta[n-1]\times I)$ by the map that identifies an element
\[
(g_{1},\ldots,g_{n}),x,t
\]
of $G^{n}\times \partial (\Delta[n-1]\times I)\subset G^{n}\times \Delta[n-1]\times I $ with 
\[
(g_{1},\ldots,\hat g_{i},\ldots,g_{n}), y,t
\]
when $x=d^{i-1}(y)$ for some $y\in \Delta[n-2]$, with the basepoint if
$t=1$, and with the non-basepoint element of $S^{0}=\tWG_{0}$ when
$t=0$.  We can identify the image of $G^{m}\times \Delta[m]^{\Delta}$
in $WG^{\Delta}$ in terms of repeated coordinates.  Concisely, a
point is in $\tWG_{n}$ exactly when it has a representative
$(g_{1},\ldots,g_{n}),x,t$ with $q(g_{1},\ldots,g_{n})\leq n$, where
$q$ is the function defined in Definition~\ref{defn:secondfiltration}.

We define the map $\tEOG{\oO}\to \tWG$ as follows.  We send
$S^{0}=\tEOG{\oO}_{0}$ by the identity into $S^{0}=\tWG_{0}$, and we
send the element of $\tEOG{\oO}$ represented by 
\[
a,(g_{1},\ldots,g_{n}),(t_{1},\ldots,t_{n}),
\]
to the basepoint if $\sum t_{i}\geq 1$, to the non-basepoint element
of $S^{0}$ if $t_{i}=0$ for all $i$, and otherwise to the element
represented by
\[
(g_{1},\ldots,g_{n}),x,t
\]
where $t=\sum t_{i}$ and $x$ has barycentric coordinates
\[
t_{1}/t, t_{2}/t,\ldots, t_{n}/t.
\]
It is clear from the gluing relations on $\tEOG{\oO}$ that this is
well-defined, continuous, and equivariant.  Moreover, it is clear from
the construction of the filtration that it is filtered for the
pseudocellular filtration.   The following proposition completes the
work we need on the pseudocellular filtration.

\begin{prop}
The map $\tEOG{\oO}^{pc}_{n}\to \tWG_{n}$ is an equivariant weak
equivalence for all $n$.
\end{prop}

\begin{proof}
For every non-trivial $H<G$, the induced map on $H$-fixed points is
the identity map $S^{0}\to S^{0}$, so it suffices to show that it is a
non-equivariant weak equivalence.  For both $\tEOG{\oO}$ and $\tWG$,
the $n$th piece of the filtration is built from the $(n-1)$st piece of
the filtration by attaching cells of a certain form along boundaries.  For
$\tEOG{\oO}$, these cells are of the form
\[
\ooO(n)\times G^{n}\times I^{n+m}
\]
with boundary
\[
(\ooO(n)\times sG^{n}\times I^{n+m})\cup (\ooO(n)\times G^{n}\times \partial I^{n+m})\subset \ooO(n)\times G^{n}\times I^{n+m}
\]
for $m\geq 0$, where $sG^{n}$ denotes the subspace where at least one
coordinate is the same as the following coordinate.  The cells for $\tWG$
are in one-to-one correspondence but of the form
\[
G^{n}\times \Delta [n-1+m]\times I
\]
with boundary
\[
(sG^{n}\times \Delta [n-1+m]\times I)\cup (G^{n}\times \partial (\Delta[n-1+m]\times I))\subset G^{n}\times \Delta [n-1+m]\times I.
\]
The map $\tEOG{\oO}^{pc}_{n}\to \tWG_{n}$ sends each cell of
$\tEOG{\oO}^{pc}_{n}$ to the corresponding cell of $\tWG_{n}$ by a homotopy
equivalence, which is a homotopy equivalence on the boundary. 
\end{proof}

\section{Constructing the filtered model: The negative filtration}
\label{sec:NegFilt}

This section continues the work on constructing the filtered functor
$T^M$ by constructing the negative part of the Hesselholt-Madsen
multiplicative filtration.  The negative filtration arises from the
$\bT$-cellular filtration of $E\bT$ and the diagonal map $E\bT\to
E\bT\times E\bT$ (partly) induces the multiplication.  Of course, the
diagonal is not compatible with the filtration; we must modify
the diagonal to make it cellular.  Our approach is to parametrize
different diagonal maps using a variant of the little $1$-cubes
operad, the ``overlapping little $1$-cubes'' operad.

\begin{defn}\label{defn:ol1co}
The \term{overlapping little $1$-cubes} $\oCO$ has $n$th
space the subspace of elements $((x_{1},y_{1}),\ldots,(x_{n},y_{n}))$ of $(I^{2})^{n}$
satisfying $x_{i}<y_{i}$, with $\Sigma_{n}$ acting in the usual way on
$(-)^{n}$.  Composition is just like in the little $1$-cubes operad:
\begin{multline*}
((x_{1},y_{1}),\ldots,(x_{n}y_{n}))\circ_{i}((x'_{1},y'_{1}),\ldots,(x'_{n'},y'_{n'}))
\\=
((x_{1},y_{1}),\ldots,(x_{i-1},y_{i-1}),
(x_{i}+(y_{i}-x_{i})x'_{1},x_{i}+(y_{i}-x_{i})y'_{1}),\\\ldots,
(x_{i}+(y_{i}-x_{i})x'_{n},x_{i}+(y_{i}-x_{i})y'_{n}),
(x_{i+1},y_{i+1}),\ldots,(x_{n},y_{n}))
\end{multline*}
\end{defn}

Identifying $I^{2}$ with the increasing linear endomorphisms of $I$ by
$(x,y) \leftrightarrow f_{(x,y)}$ where
$f_{(x,y)}(t)=x+(y-x)t$, the composition formula can be written more
conceptually as 
\begin{multline*}
(f_{(x_{1},y_{1})},\ldots,f_{(x_{n},y_{n})})\circ_{i}
(f_{(x'_{1},y'_{1})},\ldots,f_{(x'_{n'},y'_{n'})})\\=
(f_{(x_{1},y_{1})},\ldots,f_{(x_{i-1},y_{i-1})},f_{(x_{i},y_{i})}\circ
f_{(x'_{1},y'_{1})},\\\ldots,
f_{(x_{i},y_{i})}\circ f_{(x'_{n'},y'_{n'})},f_{(x_{i},y_{i})},\ldots,f_{(x_{n},y_{n})}).
\end{multline*}
Note that the distinguished point $((0,1),\ldots,(0,1))$ in
$\oCO(n)$ (corresponding to the identity map $I\to I$)
induces a map of operads from the commutative operad into
$\oCO$; this map is a spacewise equivariant homotopy
equivalence.  We have a canonical map from the little $1$-cubes operad
$\oC_{1}$ to $\oCO$ as the spacewise inclusion of the subsets
where the open intervals $(x_{1},y_{1}),\ldots,(x_{n},y_{n})\subset I$
do not overlap.

The purpose of the overlapping little $1$-cubes operad $\oCO$
is that it has a natural coaction on the geometric realization of a
simplicial space, generalizing the diagonal map.  We recall that for a
space $X$, the space of continuous maps
$\End^{\op}_{X}(n)=Map(X,X^{n})$ naturally forms an operad, and a
\term{coaction} of an operad $\oO$ on the space $X$ is a map of
operads $\oO\to \End^{\op}_{X}$.  For example, the set of diagonal maps $X\to
X^{n}$ give a coaction of $\Com$ on any space.  The theorem is the following.

\begin{thm}\label{thm:coaction}
The operad $\oCO$ has a natural coaction on the geometric
realization of a simplicial space $|Z\subdot|$ such that the composite coaction of
$\Com$ is the set of diagonal maps and the composite coaction of
$\oC_{1}$,
\[
|Z\subdot|\times \oC_{1}(n)\to |Z\subdot|^{n}
\]
is filtered for the simplicial filtration on $|Z\subdot|$.
\end{thm}

We emphasize that the target uses the cartesian product filtration for
the filtration on $|Z\subdot|^{n}$ rather than the simplicial filtration
on $|Z\subdot^{n}|$.  We apply this to the standard model of $EG$
formed from the bar construction.  In the case when $G$ is finite, the simplicial
filtration is just the cellular filtration.  In the case when $G=\bT$,
it is a renumbering of the cellular filtration which has
one free $\bT$-cell in each even dimension. (The cartesian product
filtration on $E\bT^{n}$ for $n>1$ is not closely related to a
cellular filtration since $\bT$ is positive dimensional.)  Recall that
we
define $EG_{n}$ to be the $n$th geometric filtration level when $G$ is finite
and $EG_{2n+1}=EG_{2n}$ to be the $n$th geometric filtration level when
$G=\bT$.  Then for any partition
$n=n_{1}+\cdots +n_{m}$, the map 
\[
EG\times  \oC_{1}(m)\to EG\times \cdots \times EG
\]
takes the subspace $EG_{n-1}\times \oC_{1}(m)$ to the subspace
\[
(EG^{m})_{n-1}\subset 
(EG_{n_{1}-1}\times EG\times \cdots \times EG)\cup \cdots 
\cup (EG\times\cdots \times EG\times  EG_{n_{m}-1}).
\]
We therefore get a map of based $G$-spaces
\[
EG/EG_{n-1} \sma \oC_{1}(m)_{+}\to EG/EG_{n_{1}-1}\sma \cdots \sma EG/EG_{n_{m}-1}.
\]
Then for any orthogonal $G$-spectra
$X_{1},\ldots,X_{m}$, the cotensor adjunction induces a map
of orthogonal $G$-spectra
\begin{multline*}
\oC_{1}(m)_{+}\sma F(EG/EG_{n_{1}-1},X_{1})\sma \cdots \sma
F(EG/EG_{n_{m}-1},X_{m})\\
\to F(EG/EG_{n-1},X_{1}\sma \cdots \sma X_{m}).
\end{multline*}

In the next section, we convert this structure into a monoidal
structure using the usual Moore trick to convert $\oC_{1}$ structures to
strictly associative structures.

We conclude the section with the proof of Theorem~\ref{thm:coaction}.

\begin{proof}[Proof of Theorem~\ref{thm:coaction}]
Recall the Milnor coordinates $u_{1}\leq u_{2}\leq \cdots \leq u_{n}$
on the standard $n$-simplex $\Delta[n]$ give a homeomorphism of
$\Delta[n]$ with a subspace of $I^{n}$ and relate to
the barycentric coordinates $t_{0},\ldots,t_{n}$ by the formulas
\[
u_{j}=\sum_{i=0}^{j-1} t_{i}\qquad t_{j}=u_{j+1}-u_{j}
\]
(where in the righthand formula $u_{0}=0$ and $u_{n+1}=1$).
Then any element of the geometric realization of a simplicial space
$Z\subdot$ is specified (non-uniquely) by an element $z\in Z_{n}$
and Milnor coordinates $u_{1}\leq \cdots \leq u_{n}$, and an element
of $\oCO(m)$ is specified by a sequence of subintervals of
$I$, $(x_{1},y_{1}),\ldots,(x_{m},y_{m})$.  On this element, 
the coaction map has the form
\begin{multline*}
(z,u_{1}\leq \cdots \leq u_{n}), ((x_{1},y_{1}),\ldots,(x_{m},y_{m}))
\in (Z_{n}\times \Delta[n])\times \oCO(m)
\\
\mapsto (z,v^{1}_{1}\leq \cdots \leq v^{1}_{n}),\ldots,
(z,v^{m}_{1}\leq \cdots \leq v^{m}_{n})\in (Z_{n}\times \Delta[n])^{m}.
\end{multline*}
To specify the $v^{i}_{j}$'s, we note that the linear function
$f_{(x,y)}$ has a unique weakly increasing left inverse function
$g_{(x,y)}\colon I\to I$ given by the formula
\[
g_{(x,y)}(t)=\begin{cases}
0&t<x\\
(t-x)/(y-x)&x\leq t\leq y\\
1&t>y.
\end{cases}
\]
We let $v^{i}_{j}=g_{(x_{i},y_{i})}(u_{j})$.

This formula describes a continuous map 
\[
(Z_{n}\times\Delta[n])\times \oCO(m)\to 
(Z_{n}\times\Delta[n])^{m}\to |Z\subdot|^{m}.
\]
This map is compatible with the simplicial gluing instructions and so
constructs a map 
\[
|Z\subdot|\otimes \oCO(m)\to |Z\subdot|^{m}
\]
or equivalently a map $\oC(m)\to \End^{\op}_{|Z\subdot|}(m)$.
A straightforward check of the explicit formulas shows that this is a
map of operads.

For the composite coaction of $\Com$, we are looking at the case 
\[
(x_{1},y_{1}),\ldots, (x_{m},y_{m})=(0,1),\ldots,(0,1)
\]
and $g_{(0,1)}$ is the identity endomorphism of $I$ and so
$v^{i}_{j}=u_{j}$.  The coaction is therefore the set of diagonal maps.

For the composite coaction of $\oC_{1}$, we are looking at the case
when the open intervals $(x_{1},y_{1}),\ldots, (x_{m},y_{m})$ are
non-overlapping.  In this case, for each fixed $j$, at most one
$v^{i}_{j}$ can lie in $(0,1)$ (i.e., be different from $0$ or $1$).
It follows that each $v^{i}_{1}\leq \cdots \leq v^{i}_{n}$ lies in a
$n_{i}$-face of $\Delta[n]$ for some $0\leq n_{i}\leq n$ that can be
chosen so that $n_{1}+\cdots+n_{m}\leq n$.  Thus, the image of
$(Z_{n}\times \Delta[n])\times \oC_{1}(m)$ in $(Z_{n}\times
\Delta[n])^{m}$ lands in the $n$th filtration level.
\end{proof}

\section[Verifying the hypotheses
of Section~\ref{sec:main}]%
{Constructing the filtered model and verifying the hypotheses
of Section~\ref{sec:main}}\label{sec:Moore}\label{sec:end}

In this section, we construct the functor $T^{M}$ postulated in
Section~\ref{sec:main} and establish the properties stated there.
Specifically, we construct $T^{M}$ as a lax monoidal functor
(property~\eqref{e:monoidal}), filtered (property~\eqref{e:filt}) so
that the filtration on $T$ is lax monoidal (property~\eqref{e:fm}).
Moreover, we show that $T$ is naturally isomorphic to $(-)^{tG}$ (as
functors to the stable category) via an isomorphism that takes the
filtration on $T^{M}$ to the Hesselholt-Madsen Tate filtration 
and satisfies Hypothesis~\ref{hyp:gr0}.

As in Section~\ref{sec:TSS}, we construct the integer graded filtration
out of a functor from the poset $(\bN,\oleq)\times (\bN,\ogeq)$ to orthogonal
spectra. 

\begin{cons}\label{cons:lastlabel}
For $i,j\in \bN\times \bN$, let 
\[
\bar T_{i,j}X
=(R_{G}(F(EG/EG_{j-1},R_{G}(X))\sma \tEOG{\oC_{1}}_{i}))^{G}
\]
where $R_{G}$ is a lax symmetric monoidal fibrant approximation functor,
$EG$ is the standard bar construction model, and $\oC_{1}$ is the
Boardman-Vogt little $1$-cubes operad with $\tEOG{\oC_{1}}$ as in
Construction~\ref{cons:tEOG}.  The filtration on $EG$ is the standard
simplicial filtration when $G$ is finite and twice that when $G=\bT$, just
as we used in Sections~\ref{sec:TSS}
and~\ref{sec:NegFilt}, and the filtration on $\tEOG{\oC_{1}}$ is the renumbered
pseudocellular filtration as defined in Definition~\ref{defn:secondfiltration}.
\end{cons}

Using the work of the previous two sections, $\bar T_{*,*}X$ comes
with canonical maps
\begin{equation}\label{eq:fass}
\ooCo(n)_{+}\sma 
(\bar T_{i_{1},j_{1}}X_{1} \sma \cdots \sma \bar T_{i_{n},j_{n}}X_{n})
\to \bar T_{i,j}(X_{1}\sma\cdots\sma X_{n})
\end{equation}
for $i_{k},j_{k}\in \bN\times \bN$ and $i=\sum i_{k}$, $j=\sum j_{k}$,
using the maps
\begin{gather*}
EG/EG_{j-1}\sma \ooCo(n)_{+}\to 
EG/EG_{j_{1}-1}\sma \cdots \sma EG/EG_{j_{n}-1}\hbox to 0pt{,%
\qquad and\hss}\\
\ooCo(n)_{+}\sma 
(\tEOG{\oC_{1}}_{i_{1}}\sma \cdots \sma \tEOG{\oC_{1}}_{i_{n}})
\to \tEOG{\oC_{1}}_{i}
\end{gather*}
of the previous two sections.  Moreover, these maps are consistent
with the operadic multiplication in the obvious way.  For the
construction postulated in Section~\ref{sec:main}, we need to rectify
the $\ooCo$ in the formulas above to $\oAss$; the standard trick for
doing this is the Moore construction.  

\begin{cons}\label{cons:mu}
Let $\bar T^{\bar M}_{*,*}X=(\bar T_{*,*}X)\sma \bR^{>0}_{+}$ where
$\bR^{>0}$ denotes the set of positive real numbers.  Let
$\mu_{n}\colon (\bR^{>0})^{n}\to \ooCo(n)$ be the map that takes
$\ell_{1},\ldots,\ell_{n}\in (\bR^{>0})^{n}$ to the element of
$\ooCo(n)$ consisting of the subintervals
\[
[0,\frac{\ell_{1}}{\ell}], [\frac{\ell_{1}}{\ell},\frac{\ell_{1}+\ell_{2}}{\ell}],\ldots, [\frac{\ell_{1}+\cdots +\ell_{n-1}}{\ell},1]
\]
where $\ell=\ell_{1}+\cdots+\ell_{n}$.  Let 
\[
\mu \colon \bar T^{\bar M}_{i,j}X \sma T^{\bar M}_{i',j'}Y\to \bar
T^{\bar M}_{i+i',j+j'}(X\sma Y)
\]
be the map 
\[
\bar T_{i,j}X\sma \bR^{>0}_{+}\sma 
\bar T_{i',j'}Y\sma \bR^{>0}_{+}\to
\bar T_{i+i',j+j'}(X\sma Y) \sma \bR^{>0}_{+}
\]
obtained using the canonical map $\bar T_{i,j}X\sma \bar T_{i',j'}Y\to
\bar T_{i+i',j+j'}(X\sma Y)$ above for $\mu_{2}\colon \bR^{>0}\times
\bR^{>0}\to \ooCo(2)$ and the map $+\colon \bR^{>0}\times \bR^{>0}\to
\bR^{>0}$. 
\end{cons}

For example, for $\ell_{1},\ell_{2}\in \bR^{>0}$, the restriction of
the map above to
\[
\bar T_{i,j}X\sma \{\ell_{1}\}_{+}\sma 
\bar T_{i',j'}Y\sma \{\ell_{2}\}_{+}
\]
uses $\mu_{2}(\ell_{1},\ell_{2})$ for the map
\[
\bar T_{i,j}X\sma 
\bar T_{i',j'}Y\to \bar T_{i+i',j+j'}(X\sma Y)
\]
and restricts to
\[
\bar T_{i,j}X\sma \{\ell_{1}\}_{+}\sma 
\bar T_{i',j'}Y\sma \{\ell_{2}\}_{+}\to
\bar T_{i+i',j+j'}(X\sma Y) \sma \{\ell_{1}+\ell_{2}\}_{+}.
\]
Since 
\[
\mu_{2}(\ell_{1},\mu_{2}(\ell_{2},\ell_{3}))=\mu_{3}(\ell_{1},\ell_{2},\ell_{3})=\mu_{2}(\mu_{2}(\ell_{1},\ell_{2}),\ell_{3})
\] 
and $+\colon \bR^{>0}\times \bR^{>0}\to \bR^{>0}$ is
associative, the diagram
\[
\xymatrix{%
\bar T^{\bar M}_{i,j}X\sma 
\bar T^{\bar M}_{i',j'}Y\sma 
\bar T^{\bar M}_{i'',j''}Z
\ar[r]^-{\id\sma \mu}
\ar[d]_-{\mu\sma\id}
&\bar T^{\bar M}_{i,j}X\sma 
(\bar T^{\bar M}_{i'+i'',j'+j''}(Y\sma Z)\ar[d]^-{\mu}\\
\bar T^{\bar M}_{i+i',j+j'}(X\sma Y)\sma 
\bar T^{\bar M}_{i'',j''}Z
\ar[r]_-{\mu}
&\bar T^{\bar M}_{i+i'+i'',j+j'+j''}(X\sma Y\sma Z)
}
\]
commutes for all orthogonal $G$-spectra $X,Y,Z$, and all $i,i',i'',j,j',j''\geq 0$.

\begin{cons}
Define the filtered spectrum $T^{\bar M}_{*}X$ as 
\[
T^{\bar M}_{n}X=\lhocolim_{i-j\leq n}T_{i,j}^{\bar M}X
\]
using the bar construction model for the homotopy colimit as in
Section~\ref{sec:TSS}.  Let $T^{\bar M}X=\colim T^{\bar M}_{n}X$.
\end{cons}

As in Proposition~\ref{prop:hcof}, $T^{\bar M}_{*}X$ is reasonably filtered.
The maps 
\[
\bar T^{\bar M}_{i,j}X\to (R_{G}(F(EG_{+},R_{G}(X))\sma \tEOG{\oC_{1}}))^{G}=T^{\oC_{1}}_{G}X
\]
therefore induce a weak equivalence $T^{\bar M}X\to T_{G}X$.  (See
Construction~\ref{cons:PSTate} and
Notation~\ref{notn:TwO} for the definition of $T_{G}X$ and
$T^{\oC_{1}}_{G}$, respectively.)

By construction $T^{\bar M}_{*}$ inherits an associative pairing 
\[
\mu\colon T^{\bar M}_{*}X\sma T^{\bar M}_{*}Y\to T^{\bar M}_{*}(X\sma Y)
\]
in the category of filtered spectra.  To make this pairing unital as
well, we need a slight modification that takes as input a unital
spectrum.

\begin{cons}
For $X$ an orthogonal $G$-spectrum under $\bS$, we have a canonical
map $\bS\to \bar T_{0,0}X$, whence a map $\bS\to \bar T_{i,0}X$ for
$i\geq 0$.  Let 
\[
\bar T^{M}_{i,j}X=\bar T^{\bar M}_{i,j}X=\bar T_{i,j}X\sma \bR^{>0}_{+}
\]
for $j>0$ and define $\bar T^{M}_{i,0}X$ as the pushout
\[
\xymatrix{%
\bS\sma \bR^{>0}_{+}\ar[d]\ar[r]&\bar T_{i,0}X\sma \bR^{>0}_{+}\ar[d]\\
\bS\sma \bR^{\mathstrut\geq 0}_{\mathstrut+}\ar[r]&\bar T^{M}_{i,0}X
}
\]
where $\bR^{\geq 0}$ denotes the space of non-negative real number.
Define $T^{M}_{*}X$ by $T^{M}_{n}X=\hocolim_{i-j\leq n}\bar
T^{M}_{i,j}X$ and let $T^{M}X=\colim T^{M}_{n}X$.
\end{cons}

The lax associative pairing on $\bar T^{\bar M}_{*,*}$ extends to $\bar
T^{M}_{*,*}$ where it is now lax monoidal using the canonical map 
\[
\bS\iso \bS\sma\{0\}_{+}\to \bar T^{\bar M}_{0,0}(\bS)
\]
as unit.
We then get a lax symmetric monoidal structure on the functor
$T^{M}_{*}$.  This gives the first part of the following theorem,
which shows that $T^{M}_{*}$ satisfies the hypotheses
required for the argument of Section~\ref{sec:main}.

\begin{thm}
$T^{M}_{*}$ is a lax monoidal functor from orthogonal $G$-spectra
under $\bS$ to reasonably filtered spectra, is naturally
isomorphic to the Tate functor in the homotopy category of filtered
spectra, and satisfies Hypothesis~\ref{hyp:gr0}.
\end{thm}

\begin{proof}
We constructed the lax symmetric monoidal structure above.  For the
comparison we note that the inclusion of $\bS$ in $T_{0,0}X$ and the
collapse map $\bR\to *$ induce a natural transformation
\[
\bar T^{M}_{i,j}X\to \bar T_{i,j}X
\]
of functors from 
orthogonal $G$-spectra to $(\bN,\oleq)\times
(\bN,\ogeq)$-diagrams in orthogonal spectra.  This map 
is a weak equivalence for all $X$; see~\cite[6.2]{Mandell-Smash} (or
Proposition~\ref{prop:UooCo} below).   

To verify Hypothesis~\ref{hyp:gr0}, we let
\[
TT(X)=\bar T^{M}_{0,0}X\cup_{\bar T^{M}_{0,1}X}(\bar T^{M}_{0,1}X\sma I),
\]
the standard model of the homotopy cofiber (where we use $1$ as the
basepoint of $I$).  This functor has a monoidal structure coming from
the unit of $T^{M}_{0,0}$ and the pairing on $T^{M}_{*,*}$ and the
pairing $\max$ on $I$.  This is canonically isomorphic as a monoidal
functor (reversing the direction of the interval) to the quotient of
the categorical bar construction homotopy colimit for the diagram 
\[
\bar T^{M}_{0,1}X\to \bar T^{M}_{0,0}X
\]
by the inclusion of $T^{M}_{0,1}X$.
We then have a natural transformation of monoidal functors $TT\to
(T^{M}_{0}/T^{M}_{-1})$ induced by the inclusion in the diagram of
$\bar T^{M}_{0,0}$ and $\bar T^{M}_{0,1}$.  

We define 
\[
RR(X)=(R_{G}(F((EG_{0})_{+},R_{G}X)))^{G}.
\]
We then have the monoidal natural transformation 
\[
i^{*}X\to i^{*}R_{G}R_{G}X\to (R_{G}(F((EG_{0})_{+},R_{G}X)))^{G}= RR(X).
\]
To construct the map $TT(X)\to
RR(X)$, we use the map 
\begin{multline*}
\bar T_{0,0}^{M}X\to \bar T_{0,0}X=
(R_{G}(F(EG_{+},R_{G}X)\sma \bS))^{G}\\\to
(R_{G}(F((EG_{0})_{+},R_{G}X)\sma \bS))^{G}\iso RR(X)
\end{multline*}
Since 
\[
F(EG/EG_{0},R_{G}X)\to 
F(EG_{+},R_{G}X)\to
F(EG_{0+},R_{G}X)
\]
is a point-set fiber sequence, we see that the map $TT(X)\to
RR(X)$ is a weak equivalence.  The composite map $\bar
T^{M}_{0,1}X\to RR(X)$ factors through 
\[
(R_{G}(F(EG/EG_{0},R_{G}X)\sma \bS))^{G},
\]
and therefore is the point-set trivial map.  As a consequence, the
natural transformation $TT\to
RR$ factors through the monoidal natural transformation
$TT\to \bar T^{M}_{0,0}/\bar T^{M}_{0,1}$.  Thus, it remains to show
that the transformation $T^{M}_{0,0}\to RR$ is monoidal,
but this is clear because the coaction of $\oCO$ on the zeroth filtration
level of $EG$ is the diagonal.
\end{proof}

\section{The \texorpdfstring{$E^1$}{E1}-term of the Hesselholt-Madsen
\texorpdfstring{$\bT$}{T}-Tate spectral sequence}\label{sec:CCS}

In this section we prove Lemma~\ref{lem:HMTcc} and
Theorems~\ref{thm:CCS} and~\ref{thm:CCT}.  All three results require a
careful study of the filtration, and this is where we begin.   In this
section, $E^{1}_{*,*}(X)$ will always refer to the $E^{1}$-term of the
Hesselholt-Madsen Tate spectral sequence for an orthogonal
$\bT$-spectrum $X$.

Our calculation of the $E^{1}$-term will use a filtration on
$X^{E\bT}\sma \tET$ corresponding to the Hesselholt-Madsen Tate
filtration before taking $\bT$-fixed points.  Let
\[
\bar UX_{i,j}=F(E\bT/E\bT_{j-1},R_{\bT}X)\sma \tET_{i}
\]
and let
\[
\bar U_{n}X = \lhocolim_{i-j\leq n}UX_{i,j}, \qquad \bar{\bar U}_{n}X=\lhocolim_{i-j\leq n}R_{\bT}UX_{i,j}.
\]
Because the point-set fixed point functor commutes with the
categorical bar construction of homotopy colimits and the derived
fixed point functor commutes with homotopy colimits, we have the
following proposition.  (See~\eqref{eq:Thocolim} for the notation
$\bar T_{n}X$.)

\begin{prop}
The canonical map $(\bar{\bar U}_{n}X)^{\bT}\to \bar
T_{n}X$ is an isomorphism and the canonical map $(\bar{\bar
U}_{n}X)^{\bT}\to (R_{\bT}(\bar U_{n}X))^{\bT}$ is a weak equivalence.
\end{prop}

This proposition then gives a canonical identification of the
$E^{1}$-term of the Hesselholt-Madsen Tate spectral sequence for $X$
as
\[
E^{1}_{s,t}(X)=\pi^{\bT}_{s+t}(\bar U_{s} X/ \bar U_{s-1}X),
\]
with pairing induced by the pairing on $\bar U_{*}(-)$ (itself induced by the
pairing on $\bar U(-)_{*,*}$).  We can now dispense with $\bar
T_{*}X$, $T_{\bT}X_{i,j}$, and
$\bar{\bar U}_{*}X$ and work exclusively with the equivariant
$\bT$-spectra $\bar U_{*}(-)$ and $\bar U(-)_{*,*}$.

To study the filtration quotients $\bar U_{i} X/ \bar U_{i-1}X$,
we need another homotopy colimit. Let 
\[
\bar UX^{\oleq,\ogeq}_{m,n}=\lhocolim_{i\leq m,j\geq n}\bar UX_{i,j}.
\]
Since $(m,n)$ is the final object in the category in the homotopy colimit,
the inclusion $\bar UX_{m,n}\to \bar UX^{\oleq,\ogeq}_{m,n}$ is a 
homotopy equivalence of orthogonal $\bT$-spectra.  It is easy to see that
for $n\geq 0$
\[
\bar U_{n}X=(\bar UX^{\oleq,\ogeq}_{n,0})\cup_{(\bar UX^{\oleq,\ogeq}_{n,1})}
(\bar UX^{\oleq,\ogeq}_{n+1,1})
\cup_{(\bar UX^{\oleq,\ogeq}_{n+1,2})}
(\bar UX^{\oleq,\ogeq}_{n+2,2})\cup
\cdots 
\]
and for $n\leq 0$
\[
\bar U_{n}X=(\bar UX^{\oleq,\ogeq}_{0,-n})\cup_{(\bar UX^{\oleq,\ogeq}_{0,-n+1})}
(\bar UX^{\oleq,\ogeq}_{1,-n+1})
\cup_{(\bar UX^{\oleq,\ogeq}_{1,-n+2})}
\bar (UX^{\oleq,\ogeq}_{2,-n+2})\cup
\cdots. 
\]
This gives us what we need to prove Lemma~\ref{lem:HMTcc}, conditional
convergence of the Hesselholt-Madsen Tate spectral sequence.

\begin{proof}[Proof of Lemma~\ref{lem:HMTcc}]
Using the weak equivalence 
\begin{multline*}
\bar UX_{i,j}=F(E\bT/E\bT_{j-1},R_{\bT}X)\sma \tET_{i}\\
\simeq (X^{E\bT}\smal S^{-j\bC(1)}\smal S^{i\bC(1)})\simeq (X^{E\bT}\smal S^{(i-j)\bC(1)}),
\end{multline*}
we see in particular that $\bar U_{n}X$ is weakly
equivalent to the cofiber of a map
\[
\bigvee_{\ell\geq 0} X^{E\bT}\smal S^{(n-1)\bC(1)}\to 
\bigvee_{\ell\geq 0} X^{E\bT}\smal S^{n\bC(1)}
\]
where the map on the $\ell$th summand is the Greenlees Tate inclusion into the $\ell$th
summand minus the Greenlees Tate inclusion into the $(\ell+1)$st summand.
Conditional convergence in the sense
of~\cite[5.10]{Boardman-SpectralSequences} is equivalent to showing
that $\holim_{n}(R_{\bT}\bar U_{n}X)^{\bT}\simeq *$.  Because of the
identification 
of $\bar U_{n}X$ as a homotopy cofiber above, it suffices to show 
\[
\holim_{n}\big(\bigvee (X^{E\bT}\smal S^{-n\bC(1)})^{\bT}\big)\simeq *
\]
or, equivalently,
\[
\holim_{n}\big(\bigvee ((X^{E\bT})^{E\bT/E\bT_{2n-1}})^{\bT}\big)\simeq *
\]
This is clear when $\pi_{*}X^{E\bT}$ is bounded above and follows in
the general by considering the Postnikov tower for $X^{E\bT}$.
\end{proof}

Returning to the discussion of the quotients $\bar U_{s} X/ \bar
U_{s-1}X$, the work above allows us to identify these in terms of the
double homotopy cofibers of maps for the 
$\bar UX_{i,j}$.  Writing $C(B,A)$ for the homotopy cofiber of an
understood map $A\to B$, we have that for $s>0$, $\bar U_{s} X/ \bar U_{s-1}X$ is weakly
equivalent to the cofiber of the inclusion
\begin{multline*}
(\bar UX^{\oleq,\ogeq}_{s-1,0})\cup_{(\bar UX^{\oleq,\ogeq}_{s-1,1})}
(\bar UX^{\oleq,\ogeq}_{s,1})
\cup_{(\bar UX^{\oleq,\ogeq}_{s,2})}
(\bar UX^{\oleq,\ogeq}_{s+1,2})\cup
\cdots\\
\to
\bar U_{s}X=(\bar UX^{\oleq,\ogeq}_{s,0})\cup_{(\bar UX^{\oleq,\ogeq}_{s,1})}
(\bar UX^{\oleq,\ogeq}_{s+1,1})
\cup_{(\bar UX^{\oleq,\ogeq}_{s+1,2})}
(\bar UX^{\oleq,\ogeq}_{s+2,2})\cup
\cdots 
\end{multline*}
which is easily seen to be the wedge of double homotopy cofibers
\[
\bigvee_{n\geq 0} 
  C(C(\bar UX_{s+n,n},\bar UX_{s+n-1,n}),
    C(\bar UX_{s+n,n+1},\bar UX_{s+n-1,n+1})).
\]
Similarly, for $s\leq 0$, $\bar U_{s} X/ \bar U_{s-1}X$ is weakly
equivalent to the wedge
\begin{multline*}
C(\bar UX_{0,-s},\bar UX_{0,-s+1})\ \vee\\
\bigvee_{n>0} 
  C(C(\bar UX_{n,n-s},\bar UX_{n-1,n-s}),
    C(\bar UX_{n,n-s+1},\bar UX_{n-1,n-s+1})).
\end{multline*}
(The difference in formulas arises because for $s>0$, $\bar UX_{s,0}$
admits maps from $\bar UX_{s-1,0}$ and $\bar UX_{s,1}$, whereas for
$s\leq 0$, $\bar UX_{0,-s}$ only admits a map from $\bar
UX_{0,-s+1}$.)

For convenience of later reference, we will use the following
abbreviations
\begin{equation}\label{eq:abbrev}
\begin{aligned}
C\bar UX_{n}&=C(\bar UX_{0,n},\bar UX_{0,n+1})\\
CC\bar UX_{m,n}&=C(C(\bar UX_{m,n},\bar UX_{m-1,n}),C(\bar UX_{m,n+1},\bar UX_{m-1,n+1})).
\end{aligned}
\end{equation}
We note that since the pairing on $\bar U_{*}(-)$ is induced by the
pairing on $\bar U(-)_{*,*}$, the pairings on the associated graded is
the wedge of pairings on the pieces
\begin{equation}\label{eq:newpair}
\begin{aligned}
CC\bar UX_{m,n}\sma CC\bar UY_{m',n'}&\to CC\bar U(X\sma Y)_{m+m',n+n'}\\
C\bar UX_{n}\sma CC\bar UY_{m',n'}&\to CC\bar U(X\sma Y)_{m',n+n'}\\
CC\bar UX_{m,n}\sma C\bar UY_{n'}&\to CC\bar U(X\sma Y)_{m,n+n'}\\
C\bar UX_{n}\sma C\bar UY_{n'}&\to C\bar U(X\sma Y)_{n+n'}.
\end{aligned}
\end{equation}
We proceed by studying these pairings; results translate
back to $E^{1}$-terms using the identification
\begin{equation}\label{eq:E1}
\bigoplus_{s,t}E^{1}_{s,t}(X)=\bigoplus_{n\geq 0} \pi^{\bT}_{q}(C\bar
UX_{n})\oplus \bigoplus_{m>0,n\geq 0}\pi^{\bT}_{q}(CC\bar UX_{m,n})
\end{equation}
where the $n,q$ summand on the right lies in the $s=n$, $t=q+n$
summand on the left and the $m,n,q$ summand on the right lies in the
$s=m-n$, $t=q-m+n$ summand on the left.

We have canonical identifications of the homotopy cofibers
\begin{align*}
C(\bar UX_{m,n},\bar UX_{m-1,n})
  &\simeq F(E\bT/E\bT_{\lfloor (n-1)/2\rfloor},R_{\bT}X)\sma (S^{\bC(1)^{\lfloor m/2\rfloor}}/S^{\bC(1)^{\lfloor (m-1)/2\rfloor}})\\
C(\bar UX_{m,n},\bar UX_{m,n+1})
  &\simeq F(E\bT_{\lfloor n/2 \rfloor}/E\bT_{\lfloor (n-1)/2\rfloor},R_{\bT}X)\sma S^{\bC(1)^{\lfloor m/2\rfloor}};
\end{align*}
the first cofiber is trivial if $m$ is odd and the second is trivial
if $n$ is odd.  It follows that the double homotopy cofibers are
trivial if either $m$ or $n$ is odd, and when both are even, we have
\begin{multline*}
CC\bar UX_{2m,2n}=
C(C(\bar UX_{2m,2n},\bar UX_{2m-1,2n}),C(\bar UX_{2m,2n+1},\bar UX_{2m-1,2n+1}))\\
\simeq F(E\bT_{2n}/E\bT_{2n-1},R_{\bT}X)\sma (S^{\bC(1)^{m}}/S^{\bC(1)^{m-1}}),
\end{multline*}
Since $E\bT_{2n}/E\bT_{2n-1}$ and $S^{\bC(1)^{m}}/S^{\bC(1)^{m-1}}$
are all of the form $\bT_{+}\sma Z$ for some based $\bT$-spaces $Z$
(depending on $m$ or $n$),
the associated graded pieces $\bar U_{s} X/ \bar U_{s-1}X$ are all
(non-canonically) of the form $\bT_{+}\sma Z$ for some spectra $Z$.
The next step is to review the relationship between the homotopy
groups $\pi_{*}$ and $\pi_{*}^{\bT}$ 
of such spectra.

Writing $\piS(\bT)$ for the stable homotopy groups of $\bT_{+}$, the
inclusion of the unit $S^{0}\to \bT_{+}$ and the collapse map
$\bT_{+}\to S^{0}$ induce a canonical splitting 
\[
\piS(\bT)\iso \pi_{*}\bS \oplus \pi_{*}(\Sigma \bS).
\]
We denote the canonical generator in $\piS[0](\bT)$ as $1$ and the
generator in $\piS[1](\bT)$ (determined by the orientation on $\bT$ as the
unit complex numbers) as $\tact$; the multiplication on $\bT$ makes $\piS(\bT)$
a graded ring, and $1$ is the unit in this structure.  It is
well-known (and essentially the definition of the Hopf invariant) that
the element $\tact$ satisfies $\tact^{2}=\eta \tact$ where $\eta$
denotes the non-zero element of $\pi_{1}\bS\iso \bZ/2$.  For any orthogonal
$\bT$-spectrum $X$, we then have a canonical action of $\piS(\bT)$ on
$\pi_{*}X$, or equivalently, a degree $1$ operator $\tact$ satisfying
$\tact^{2}x=\eta \tact x$.  

\begin{prop}\label{prop:pifixed}
Let $X$ be an orthogonal $\bT$-spectrum which is isomorphic in the
equivariant stable category to $\bT_{+}\sma Z$ for some orthogonal
$\bT$-spectrum $Z$.  Then the map $\pi_{*}^{\bT}X\to \pi_{*}X$ is
injective and is surjective onto the kernel of multiplication by
$\tact$, which is equal to the image of multiplication by $\tact +\eta$.
\end{prop}

\begin{proof}
The image of $\pi_{*}^{\bT}X$ in $\pi_{*}X$ is clearly contained in
the kernel of multiplication by $\tact$.  The kernel of
multiplication by $\tact$ clearly contains the image of multiplication
by $\tact+\eta$; an easy calculation shows that they are equal under
the hypotheses.  It suffices
to consider the case when $Z$ comes from a non-equivariant orthogonal
spectrum (using the isomorphism between the diagonal action and the
action on just $\bT$).  In the case $Z=\bS$, the transfer gives a weak
equivalence $\pi_{*}\Sigma \bS\to (\bT_{+}\sma \bS)^{\bT}$ and the
composite 
\[
\pi_{*}\Sigma \bS \to \pi_{*}(\bT_{+}\sma \bS)^{\bT} \to
\pi_{*}(\bT_{+}\sma \bS)\iso \piS(\bT)
\]
takes the fundamental class of $\pi_{1}\Sigma \bS$ to the element $\tact
+\eta 1$; the theorem follows in this case.  For
arbitrary $Z$, naturality of the transfer map implies that
\[
\pi_{*}\Sigma Z\to \pi_{*}(\bT_{+}\sma Z)^{\bT} \to
\pi_{*}(\bT_{+}\sma Z)\iso \piS(\bT)\otimes_{\pi_{*}\bS}\pi_{*}Z
\]
takes an element $x\in \pi_{n}Z\iso \pi_{n+1}\Sigma Z$ to
$(\tact+\eta 1)\otimes x=(\tact +\eta)(1\otimes x)$.
\end{proof}

Now we are ready to start identifying elements in $\pi_{*}(C\bar U
\bS_{n})$ and $\pi_{*}(CC\bar U\bS_{m,n})$. We can then apply
Proposition~\ref{prop:pifixed} (and~\eqref{eq:E1}) to compute $E^{1}_{*,*}(\bS)$ and with some more work, $E^{1}_{*,*}(X)$ for general $X$.  Using the
canonical orientation on $S^{\bC(1)^{m}}$, we get canonical generators
\[
a_{m}\in \pi_{2m-1}\Sigma^{\infty}(S^{\bC(1)^{m}}/S^{\bC(1)^{m-1}}),\qquad 
b_{m}\in \pi_{2m}\Sigma^{\infty}(S^{\bC(1)^{m}}/S^{\bC(1)^{m-1}})
\]
where $b_{m}$ is the image of the fundamental class of
$S^{\bC(1)^{m}}$ and $a_{m}$ maps to the fundamental class in
$S^{\bC(1)^{m-1}}$;
$\pi_{*}\Sigma^{\infty}(S^{\bC(1)^{m}}/S^{\bC(1)^{m-1}})$ is then the
free $\pi_{*}\bS$-module generated by $a_{m}$ and $b_{m}$.  In terms
of the action of $\piS(\bT)$, it is easy to see geometrically that
$\tact a_{1}=b_{1}$.  Using the cofiber sequence 
\[
\bT_{+}\to S^{0}\to S^{\bC(1)}\to S^{\bC(1)}/S^{0}\iso \Sigma\bT_{+}
\]
we see that $\tact b_{1}=\eta b_{1}$; the identification
\[
S^{\bC(1)^{m}}/S^{\bC(1)^{m-1}}\iso (S^{\bC(1)})^{(m-1)}\sma (S^{\bC(1)}/S^{0})
\]
shows that in general
\begin{equation}\label{eq:actab}
\begin{aligned}
\tact a_{m}&=b_{m}+(m-1)\eta a_{m}\\
\tact b_{m}&=m\eta b_{m}.
\end{aligned}
\end{equation}
To produce corresponding formulas for
$F(E\bT_{2n}/E\bT_{2n-1},R_{\bT}\bS)$, we use the following lemma proved
in the next section.

\begin{lem}\label{lem:homeo}\label{LEM:HOMEO}
There exists a unique system of isomorphisms in the Borel stable category
$E\bT/E\bT_{2n-1}\overto{\simeq} E\bT_{+}\sma S^{\bC(1)^{n}}$ such
that:
\begin{enumerate}
\item The diagram
\[
\xymatrix@-1pc{%
E\bT/E\bT_{2n-1}\ar[d]_{\simeq}\ar[r]
&E\bT/E\bT_{2i-1}\sma E\bT/E\bT_{2j-1}\ar[d]^{\simeq}\\
E\bT_{+}\sma S^{\bC(1)^{n}}\ar[r]
&(E\bT_{+}\sma S^{\bC(1)^{i}})\sma(E\bT_{+}\sma S^{\bC(1)^{j}})
}
\]
commutes for all $i+j=n$, where the top horizontal map is induced
by the filtered approximation to the diagonal on $E\bT$ constructed in
Section~\ref{sec:PosFilt} and the bottom horizontal map is induced by
the diagonal map on $E\bT$ and the homeomorphism 
\[
S^{\bC(1)^{n}}=
S^{\bC(1)^{i+j}}\iso 
S^{\bC(1)^{i}\oplus \bC(1)^{j}}\iso 
S^{\bC(1)^{i}}\sma S^{\bC(1)^{j}}.
\]
\item The diagram
\[
\xymatrix@-1pc{%
E\bT/E\bT_{2m-1}\ar[d]_{\simeq}\ar[r]&E\bT/E\bT_{2m}\ar[d]^{\simeq}\\
E\bT_{+}\sma S^{\bC(1)^{m}}\ar[r]&E\bT_{+} \sma S^{\bC(1)^{m+1}}
}
\]
commutes for $m=0$, where the bottom map is induced by the
inclusion of $\bC(1)^{m}$ in $\bC(1)^{m+1}$ as the first $m$
coordinates.  
\end{enumerate}
Moreover, the diagram in (ii) commutes for all $m\geq 0$.
\end{lem}

The lemma gives us a 
canonical weak equivalence 
\[
F(E\bT/E\bT_{2m-1},R_{\bT}X)\simeq
X^{E\bT}\smal S^{-m\bC(1)},
\]
and we then define generators
$a'_{m}, b'_{m}$ for $\pi_{*}F(E\bT_{2m}/E\bT_{2m-1},R_{\bT}\bS)$ as
above.  Namely, $b'_{m}$ is the generator of $\pi_{-2m}$ in the image
of the fundamental class of 
\[
\pi_{-2m}F(E\bT/E\bT_{2m-1},R_{\bT}\bS)\iso
\pi_{-2m}(\bS^{E\bT}\smal S^{-m\bC(1)})\iso \pi_{-2m}S^{-m\bC(1)}
\]
(for the standard orientation)
and  $a'_{m}$ is the generator of $\pi_{-2m-1}$ that maps to the
fundamental class of 
\[
\pi_{-2m-2}F(E\bT/E\bT_{2m}R_{\bT}\bS)=
\pi_{-2m-2}F(E\bT/E\bT_{2m+1}R_{\bT}\bS)
\iso \pi_{-2m-2}S^{-(m+1)\bC(1)}.
\]
As above, we have
\begin{equation}\label{eq:actabp}
\begin{aligned}
\tact a'_{m}&=b'_{m}+(-m-1)\eta a'_{m}=b'_{m}+(m-1)\eta a'_{m}\\
\tact b'_{m}&=-m\eta b'_{m}=m\eta b'_{m}.
\end{aligned}
\end{equation}

Using the elements introduced above, we can identify the homotopy
groups of $CC\bar UX_{2m,2n}$ as 
\[
\pi_{*}(CC\bar UX_{2m,2n})\iso
\pi_{*}\bS\langle a'_{n},b'_{n}\rangle\otimes_{\pi_{*}\bS} 
\pi_{*}\bS\langle a_{m},b_{m}\rangle 
\otimes_{\pi_{*}\bS} \pi_{*}X
\]
using the map
\begin{multline*}
\pi_{*}(F(E\bT_{2n}/E\bT_{2n-1},R_{\bT}\bS))\otimes_{\pi_{*}\bS}
\pi_{*}\Sigma^{\infty}(S^{\bC(1)^{m}}/S^{\bC(1)^{m-1}})
\otimes_{\pi_{*}\bS} \pi_{*}X\\\to
\pi_{*}(F(E\bT_{2n}/E\bT_{2n-1},R_{\bT}\bS)\sma
(S^{\bC(1)^{m}}/S^{\bC(1)^{m-1}})\sma X)
\\\to
\pi_{*}(F(E\bT_{2n}/E\bT_{2n-1},R_{\bT}X)\sma
(S^{\bC(1)^{m}}/S^{\bC(1)^{m-1}}))=\pi_{*}(CC\bar UX_{2m,2n}),
\end{multline*}
where $\pi_{*}\bS\langle-\rangle$ denotes the free $\pi_{*}\bS$-module
on the given generators (in their natural degrees).  We likewise have
\[
\pi_{*}(C\bar UX_{2n})\iso
\pi_{*}\bS\langle a'_{n},b'_{n}\rangle\otimes_{\pi_{*}\bS} 
\pi_{*}X.
\]

We next describe the pairing for $X=\bS$.  The induced pairing on
homotopy groups of the pairing in~\eqref{eq:newpair} in this case
takes the form 
\begin{multline*}
(\pi_{*}\bS\langle a'_{n},b'_{n}\rangle \otimes_{\pi_{*}\bS} 
\pi_{*}\bS\langle a_{m},b_{m}\rangle)
\otimes_{\pi_{*}\bS}
(\pi_{*}\bS\langle a'_{n'},b'_{n'}\rangle \otimes_{\pi_{*}\bS} 
\pi_{*}\bS\langle a_{m'},b_{m'}\rangle)
\\\to 
\pi_{*}\bS\langle a'_{n+n'},b'_{n+n'}\rangle \otimes_{\pi_{*}\bS} 
\pi_{*}\bS\langle a_{m+m'},b_{m+m'}\rangle.
\end{multline*}
the analogous formulas for the other three pairings, dropping either
$\langle a_{m},b_{m}\rangle$, $\langle a_{m'},b_{m'}\rangle$, or both
tensor factors also hold. 

Looking at the pairing on $\bar U\bS_{m,n}$, it is clear that the
pairing on the\break 
$F(E\bT/E\bT_{2n-1},R_{\bT}\bS)$-part and on the $S^{m\bC(1)}$-part
are independent of each other; Lemma~\ref{lem:homeo} describes what
happens on the former, and on the latter it is given by the usual
isomorphism
\[
S^{m\bC(1)}\sma S^{m'\bC(1)}\simeq S^{(m+m')\bC(1)}.
\]
This implies that for the pairings above
we have $b'_{n}\otimes b'_{n'}\mapsto b'_{n+n'}$ and $b_{m}\otimes
b_{m'}\mapsto b_{m+m'}$.  In addition, we see that $a'_{n}\otimes
a'_{n'}\mapsto 0$ and $a_{m}\otimes a_{m'}\mapsto 0$ for dimension reasons.
Since the pairings on the spectrum level are $\bT$-equivariant, the
algebraic maps above commute with the action of $\tact$. Then
\[
\tact (a_{m}\otimes a_{m'})= (\tact a_{m})\otimes
a_{m'}-a_{m}\otimes (\tact a_{m'}) 
= b_{m}\otimes a_{m'}-a_{m}\otimes b_{m'}+(m-m')\eta a_{m}\otimes a_{m'}
\]
goes to $0$ since $a_{m}\otimes a_{m'}\mapsto 0$, and we see that
$b_{m}\otimes a_{m'}$ and $a_{m}\otimes b_{m'}$ go to the same element
of $\pi_{*}\bS\langle a_{m+m'},b_{m+m'}\rangle$.  
For dimension reasons, these elements go to an integral multiple of
$a_{m+m'}$; since 
\begin{align*}
\tact (b_{m}\otimes a_{m'})
&=(\tact b_{m})\otimes a_{m'}+b_{m}\otimes (\tact a_{m'})\\
&=m\eta b_{m}\otimes a_{m'}+b_{m}\otimes b_{m'}+(m'-1)\eta b_{m}\otimes a_{m'}
\\
&=b_{m}\otimes b_{m'}+(m+m'-1)\eta b_{m}\otimes a_{m'},
\end{align*}
we see $b_{m}\otimes a_{m'}\mapsto a_{m+m'}$.  The same
observations apply to $b'_{n}\otimes a'_{n'}$ and $a'_{n}\otimes
b'_{n'}$ to show that these both map to $a'_{n+n'}$.  We summarize
this in the following proposition.

\begin{prop}\label{prop:mult}
The pairings of homotopy groups $\pi_{*}$ from~\eqref{eq:newpair} are
induced by the tensor products of the maps
\begin{align*}
b'_{n}\otimes b'_{n'}&\mapsto b'_{n+n'}&
b_{m}\otimes b_{m'}&\mapsto b_{m+m'}\\
b'_{n}\otimes a'_{n'}&\mapsto a'_{n+n'}&
b_{m}\otimes a_{m'}&\mapsto a_{m+m'}\\
a'_{n}\otimes b'_{n'}&\mapsto a'_{n+n'}&
a_{m}\otimes b_{m'}&\mapsto a_{m+m'}\\
a'_{n}\otimes a'_{n'}&\mapsto 0&
a_{m}\otimes a_{m'}&\mapsto 0.
\end{align*}
\end{prop}

We see from the formulas above that $b_{0}\in \pi_{0}(C\bar U\bS_{0})$
acts by the identity and satisfies $\tact b_{0}=0$, and so gives the
identity element of $E^{1}_{0,0}(\bS)$.  We now identify
some other key elements of $\pi_{*}(\bar U\bS_{s}/\bar U\bS_{s-1})$.

\begin{notn}\label{notn:xyz}
Let 
\begin{align*}
\bar x&=b'_{0}\otimes b_{1}+\eta b'_{0}\otimes a_{1},
  &\bar x&\in \pi_{2}(CC\bar U\bS_{2,0})\\
\bar y&=a'_{0}\otimes b_{1}-b'_{0}\otimes a_{1}
  &\bar y&\in \pi_{1}(CC\bar U\bS_{2,0})\\
\bar z&=b'_{1}+\eta a'_{1},
  &\bar z&\in \pi_{-2}(C\bar U\bS_{2}).
\end{align*}
\end{notn}

Applying the formulas~\eqref{eq:actab} and~\eqref{eq:actabp} for the
action of $\tact$
and using Proposition~\ref{prop:pifixed} to identify
$\pi^{\bT}_{*}$ as a subset of $\pi_{*}$ for $CC\bar U\bS_{2,0}$ and
$C\bar U\bS_{2}$, we define elements $x,y,z$ in $E^{1}_{*,*}(\bS)$ as
follows. 

\begin{prop}
The elements $\bar x$, $\bar y$, and $\bar z$ lift to (unique) elements 
\begin{align*}
&x\in \pi^{\bT}_{2}(CC\bar U\bS_{2,0})\subset E^{1}_{2,0}(\bS)\\
&y \in \pi^{\bT}_{1}(CC\bar U\bS_{2,0})\subset E^{1}_{2,-1}(\bS)\\
&z \in \pi^{\bT}_{-2}(C\bar U\bS_{2})\subset E^{1}_{-2,0}(\bS).
\end{align*}
\end{prop}

We can now prove Theorems~\ref{thm:CCS} and~\ref{thm:CCT} 

\begin{proof}[Proof of Theorem~\ref{thm:CCS}]
Eliding notation for the distinction between an element of $\pi^{\bT}_{*}$ and its image in
$\pi_{*}$, the multiplication formulas of Proposition~\ref{prop:mult}
give us
\begin{align*}
x^{m}&=b'_{0}\otimes b_{m}+m\eta b'_{0}\otimes a_{m}\\
y^{2}&=0\\
z^{n}&=b'_{n}+n\eta a'_{n}.
\end{align*}
From here we see 
\begin{align*}
x^{m}z^{n}&=b'_{n}\otimes b_{m}+n\eta a'_{n}\otimes b_{m}+m\eta b'_{n}\otimes a_{m}+mn\eta^{2}a'_{n}\otimes a_{m}\\
x^{m}y&=a'_{0}\otimes b_{m+1}-b'_{0}\otimes a_{m+1}+m\eta a'_{0}\otimes a_{m+1}\\
x^{m}yz^{n}&=a'_{n}\otimes b_{m+1}-b'_{n}\otimes a_{m+1}+(n+m)\eta a'_{n}\otimes a_{m+1}\\
yz^{n}&=a'_{n}\otimes b_{1}-b'_{n}\otimes a_{1}+n\eta a'_{n}\otimes a_{1}
\end{align*}
for $m>0, n\geq 0$.  Since $\{a'_{n}\otimes a_{m},a'_{n}\otimes
b_{m},b'_{n}\otimes a_{m},b'_{n}\otimes b_{m}\}$ is a set of
generators for the free $\pi_{*}\bS$-module $\pi_{*}(CC\bar
U\bS_{2m,2n})$ and $\{a'_{n},b'_{n}\}$ is a set of generators for the
free $\pi_{*}\bS$-module $\pi_{*}(C\bar U\bS)_{2n}$, we see from the
formulas above that the map from the free (tri)graded commutative
$\pi_{*}\bS$-algebra on $\bar x,\bar y,\bar z$ (in the appropriate
degrees; q.v.~\ref{notn:xyz}) modulo $\bar y^{2}=0$ injects
into 
\[
\bigoplus_{n\geq 0} \pi_{*}(C\bar U\bS_{2n}) \oplus \bigoplus_{m>0,n\geq 0} \pi_{*}(CC\bar U\bS_{2m,2n})
\]
and it follows that it injects into 
\[
\bigoplus_{n\geq 0} \pi^{\bT}_{*}(C\bar U\bS_{2n}) \oplus \bigoplus_{m>0,n\geq 0} \pi^{\bT}_{*}(CC\bar U\bS_{2m,2n}).
\]
To see it is surjective, we note that by Proposition~\ref{prop:mult},
the elements 
\begin{align*}
(\tact +\eta)(b'_{n}\otimes b_{m})&=
(m+n+1)\eta b'_{n}\otimes b_{m}
&(\tact +\eta)b'_{n}&=(n+1)\eta b'_{n}\\
(\tact +\eta)(b'_{n}\otimes a_{m})&=
b'_{n}\otimes b_{m}+(n+m)\eta b'_{n}\otimes a_{m}
&(\tact +\eta)a'_{n}&=b'_{n}+n\eta a'_{n}\\
(\tact +\eta)(a'_{n}\otimes b_{m})&=
b'_{n}\otimes b_{m}+(n+m)\eta a'_{n}\otimes b_{m}\\
(\tact +\eta)(a'_{n}\otimes a_{m})&=
b'_{n}\otimes a_{m}-a'_{n}\otimes b_{m}
  +(n+m+1)\eta a'_{n}\otimes a_{m}\hskip -4em
\end{align*}
generate $\pi^{\bT}_{*}(CC\bar U\bS_{2m,2n})$ (in the left column) and
$\pi^{\bT}(C\bar U\bS_{2n})$ (in the right column) as
$\pi_{*}\bS$-modules.  From this we see that 
\begin{itemize}
\item $x^{m}z^{n},x^{m-1}yz^{n}$ generate $\pi_{*}^{\bT}(CC\bar
U\bS)_{2m,2n}$ for $m\geq 1$, $n\geq 0$
\item $z^{n}$ generates $\pi^{\bT}_{*}(C\bar U\bS_{2n})$ for $n\geq 0$.
\end{itemize}
The theorem now follows from~\eqref{eq:E1}.
\end{proof}

\begin{proof}[Proof of Theorem~\ref{thm:CCT}]
The cofiber $C\bar UX_{0}$ of $\bar UX_{0,1}\to \bar UX_{0,0}$ has a
ca\-nonical weak equivalence to
$R_{\bT}F(E\bT_{0}/ET_{-1},X)=R_{\bT}F(\bT_{+},R_{\bT}X)$.  The map 
\[
i^{*}X\to R_{\bT}F(\bT_{+},R_{\bT}X)^{\bT}\to
i^{*}R_{\bT}F(\bT_{+},R_{\bT}X)\simeq \dR F(\bT_{+},i^{*}X)
\]
in the non-equivariant stable category is adjoint to the $\bT$-action map 
\[
\bT_{+}\sma i^{*}X\to i^{*}X,
\]
so the map
\[
\piS\bT\otimes_{\pi_{*}\bS} \pi_{*}X\to \piS\bT\otimes_{\pi_{*}\bS} \pi_{*}(C\bar
UX_{0})\to \pi_{*}X
\]
sends $1\otimes x$ to $x$ and $\tact \otimes x$ to $\tact x$, where
the second map is induced by the counit (evaluation) 
\[
\bT_{+}\sma \dR F(\bT_{+},i^{*}X)\to i^{*}X.
\]
We note that the evaluation map is $\bT$-equivariant and so in the
case of $X=\bS$ satisfies
\begin{align*}
1\otimes b'_{0}&\mapsto 1&1 \otimes a'_{0}&\mapsto 0\\
\tact \otimes b'_{0}&\mapsto 0&\tact \otimes a'_{0}&\mapsto -1
\end{align*}
(the first formula coming from Lemma~\ref{lem:homeo}, the remainder
following for equivariance and dimension reasons). In terms of our
identification 
\[
\pi_{*}(C\bar UX_{0})\iso\pi_{*}\bS\langle a'_{0},b'_{0}\rangle\otimes_{\pi_{*}\bS}\pi_{*}X,
\]
we therefore have that the map $\pi_{*}X\to \pi_{*}(C\bar UX_{0})$ satisfies the
formula
\begin{equation}\label{eq:piX}
v\mapsto b'_{0}\otimes v-a'_{0}\otimes \tact v.
\end{equation}

By Proposition~\ref{prop:mult}, we see that the map in the statement
\[
\CC_{*,*}\otimes \pi_{*}X\to E^{1}_{*,*}(X)
\]
satisfies the following formulas
\begin{equation}\label{eq:nonmultform}
\begin{aligned}
z^{n}\otimes v&\mapsto
  (b_{n}+n\eta a'_{n})\otimes v-a'_{n}\otimes \tact v\\
x^{m}z^{n}\otimes v&\mapsto
  (b'_{n}\otimes b_{m}+n\eta a'_{n}\otimes b_{m}
    +m\eta b'_{n}\otimes a_{m}+mn\eta^{2}a'_{n}\otimes a_{m})\otimes v\\
  &\qquad\qquad -(a'_{n}\otimes b_{m} + m\eta a'_{n}\otimes a_{m})\otimes \tact v\\
x^{m}yz^{n}\otimes v&\mapsto
  (a'_{n}\otimes b_{m+1}-b'_{n}\otimes a_{m+1}+(n+m)\eta a'_{n}\otimes a_{m+1})\otimes v\\
  &\qquad\qquad +a'_{n}\otimes a_{m+1}\otimes \tact v\\
yz^{n}\otimes v&\mapsto
  (a'_{n}\otimes b_{1}-b'_{n}\otimes a_{1}+n\eta a'_{n}\otimes a_{1})\otimes v
  + a'_{n}\otimes a_{1}\otimes \tact v
\end{aligned}
\end{equation}
for all $m>0,n\geq 0$, from which it is clear that the map is
injective.  Surjectivity follows from looking at the image of
multiplication by $\tact+\eta $ just as in the proof of
Theorem~\ref{thm:CCS} above.

The bimodule statement follows from the monoidality statement, which
we now check.  For $v\in \pi_{*}X$, $w\in \pi_{*}Y$, and writing $v\sma
w$ for the smash product element in $\pi_{*}(X\sma Y)$, the map
\[
(\CC_{*}\otimes \pi_{*}X)\otimes (\CC_{*}\otimes \pi_{*}Y)\to
E^{1}_{*,*}(X)\otimes E^{1}_{*,*}(Y)\to E^{1}_{*,*}(X\sma Y)
\]
takes $(1\otimes v)\otimes (1\otimes w)$ to 
\begin{multline*}
(b'_{0}\otimes v-a'_{0}\otimes \tact v)
  (b'_{0}\otimes w-a'_{0}\otimes \tact w)\\
=b'_{0}\otimes (v\sma w)
   -a'_{0}\otimes ((\tact v)\sma w)
-(-1)^{|v|}a'_{0}\otimes (v\sma (\tact w))\\
=b'_{0}\otimes (v\sma w)-a'_{0}\otimes \tact (v\sma w).
\end{multline*}
This agrees with the image of $v\otimes w$ under the map
\[
(\CC_{*,*}\otimes \pi_{*}X)\otimes (\CC_{*}\otimes \pi_{*}Y)\to
\CC_{*,*}\otimes \pi_{*}(X\sma Y)\to E^{1}_{*,*}(X\sma Y)
\]
induced by the multiplication on $\CC_{*}$. More generally, for any
$\alpha \in \CC_{*}$, both maps send $(1\otimes x)\otimes
(\alpha\otimes y)$ to the same element of $E^{1}_{*,*}(X\sma Y)$. The
monoidality statement follows.
\end{proof}

\section[Comparison of the Spectral Sequences]{Comparison of the
Hesselholt-Madsen and Greenlees \texorpdfstring{$\bT$}{T}-Tate Spectral Sequences}\label{sec:HMvsG}

We continue the work of the previous section to prove
Theorems~\ref{thm:HMvsG} and~\ref{thm:HMd2}, which are not needed in
the rest of the memoir.  As we use Theorem~\ref{thm:HMd2} to prove
Theorem~\ref{thm:HMvsG}, we begin there.

\begin{proof}[Proof of Theorem~\ref{thm:HMd2}]
Since the multiplication comes from a pairing of filtered objects, all
differentials satisfy the Leibniz rule.   
By definition, the connecting map $\partial$ for the cofiber sequence $S^{0}\to
S^{\bC(1)}\to S^{\bC(1)}/S^{0}$ takes $a_{1}\in
\pi_{1}(S^{\bC(1)}/S^{0})$ to the generator $1\in \pi_{0}(S^{0})$ and
takes the generator $b_{1}\in \pi_{2}(S^{\bC(1)}/S^{0})$ to $0$.  Likewise, in the
cofiber sequences for the maps $S^{-(m+1)\bC(1)}\to S^{-m\bC(1)}$
that arise from filtration on $E\bT$, 
the connecting maps satisfy $\partial a'_{m}=b'_{m+1}$ and $\partial
b'_{m}=0$.  We can then read off the formulas for $d_{2}x$, $d_{2}y$,
and $d_{2}z$ in the spectral sequence for $X=\bS$ from the definition of
these classes in Notation~\ref{notn:xyz}:
\begin{align*}
d_{2}x&=\eta b'_{0}=\eta \\
d_{2}y&=b'_{1}\otimes b_{1}-b'_{0}=b'_{1}\otimes b_{1}-1\\
d_{2}z&=\eta b'_{2}=z^{2}.
\end{align*}
To finish the identification of $d_{2}y$, we note that 
\begin{multline*}
xz+\eta yz
=(b'_{1}\otimes b_{1}+\eta a'_{1}\otimes b_{1} 
   + \eta b'_{1}\otimes a_{1}+\eta^{2}a'_{1}\otimes a_{1}) 
\\
+\eta 
(a'_{1}\otimes b_{1}-b'_{1}\otimes a_{1}+\eta a'_{1}\otimes a_{1})
=b'_{1}\otimes b_{1}.
\end{multline*}

For $v\in \pi_{*}X$, the tensor product notation in the statement
conflicts with the tensor product notation we have been using above.
For clarity, we continue to write elements in $\pi_{*}(C\bar U
X_{*})$ using the symbol $\otimes$, and switch to writing elements of
$HM_{*,*}\otimes X$ using the symbol $\boxtimes$.
For $v\in \pi_{*}X$, by~\eqref{eq:piX}, the element $1\boxtimes
v$ in $HM_{*,*}\otimes \pi_{*}X$ is the element $b'_{0}\otimes
v-a'_{0}\otimes \tact v$ in $\pi_{*}(C\bar U X_{*})$, giving the formula
\[
d_{2}(1\boxtimes v)=-b'_{1}\otimes \tact v\in \pi_{*}(C\bar U X_{0}).
\]
Since $1\boxtimes\tact v\in HM_{*,*}\otimes \pi_{*}X$ is the
element $b'_{0}\otimes \tact v-\eta a'_{0}\otimes \tact v\in
\pi_{*}(C\bar U X_{0})$, we see that
$z\boxtimes \tact v\in HM_{*,*}\otimes \pi_{*}X$ is the element
\[
(b'_{1}+\eta a'_{1})(b'_{0}\otimes \tact v-\eta a'_{0}\otimes \tact v)=
b'_{1}\otimes \tact v \in \pi_{*}(C\bar U X_{2}).
\]
This then gives the formula $d_{2}(1\boxtimes v)=-z\boxtimes \tact v$
in $HM_{*,*}\otimes \pi_{*}X$.
\end{proof}

\begin{proof}[Proof of Theorem~\ref{thm:HMvsG}]
It is clear from the description in the statement that the specified
map is a map of filtered spectra and induces a map of spectral
sequences. For the purposes of this proof, we write $GE^{2}_{*,*}$ for
the $E^{2}=E^{1}$ term of the Greenlees $\bT$-Tate spectral sequence
as a complex with its $d_{2}$ differential, and analogously
$HME^{2}_{*,*}$ for the $E^{2}=E^{1}$ term of the Hesselholt-Madsen
$\bT$-Tate spectral sequence.  The Greenlees spectral sequence arises
from $\pi^{\bT}_{*}$ of a filtered spectrum; write $GC_{s}$
for the level $n$ cofiber so that (by definition)
\[
GE^{2}_{s,t}=\pi^{\bT}_{s+t}GC_{s}.
\]
We have $GC_{n}\simeq *$ for $n$ odd and
\[
GC_{2n}\simeq C(S^{n\bC(1)},S^{(n-1)\bC(1)})\sma X
\]
for all $n\in \bZ$.
The filtered map from the Greenlees filtration to the
Hesselholt-Madsen filtration then sends $GC_{n}$ isomorphically (in
the stable category) to $C\bar UX_{|n|}$ for $n\leq 0$ or to $C(\bar
UX_{n,0},\bar U_{n-1,0})$ for $n>0$.  Let
\[
b''_{n}\in \pi_{2n}C(S^{n\bC(1)},S^{(n-1)\bC(1)})
\]
be the element that is the image
of the fundamental class of $S^{n\bC(1)}$ (for the standard
orientation)  and let
\[
a''_{n}\in \pi_{2n-1}C(S^{n\bC(1)},S^{(n-1)\bC(1)})
\]
be the element that the connecting map in the cofiber sequence takes
to the fundamental class of $S^{(n-1)\bC(1)}$.  Then 
\[
\pi_{*}GC_{2n}=\pi_{*}\bS\langle a''_{n},b''_{n}\rangle \otimes_{\pi_{*}\bS}\pi_{*}X
\]
and $GE^{2}_{*,*}$ is the kernel of the action of $\tact\in
\pi^{\bS}_{1}\bT$ (Proposition~\ref{prop:pifixed}) on the above, which can
be computed by the Leibniz rule from its action on each tensor factor.
In this notation, the map $GE^{2}_{*,*}\to HME^{2}_{*,*}$ is induced
by the map 
\begin{align*}
b''_{n}&\mapsto 
\begin{cases}
b'_{|n|}&n\leq 0\\
b'_{0}\otimes b_{n}&n>0
\end{cases}\\[1ex]
a''_{n}&\mapsto 
\begin{cases}
a'_{|n|}&n\leq 0\\
b'_{0}\otimes a_{n}&n>0
\end{cases}
\end{align*}
in the notation of the previous section. 

We introduce the notation $HMC_{*}$ for the graded spectrum analogous
to $GC_{*}$ that we studied (in terms of its further decomposition) in
the previous section,
\[
HMC_{n}=\bar U_{n}X/\bar U_{n-1}X\simeq \begin{cases}
C\bar UX_{|n|}\vee \bigvee\limits_{m>0} CC\bar UX_{m,m+|n|}&n\leq 0\\[.5ex]
\bigvee\limits_{m\geq n} CC\bar UX_{m,m-n}&n>0.\\
\end{cases}
\]
Since the filtration cofibers $GC_{*}$ and $HMC_{*}$ are trivial
in odd degrees, the $d_{2}$ differential makes
$\pi_{*}GC_{*}$ and $\pi_{*}HMC_{*}$ into differential graded
$\pi_{*}^{\bS}\bT$-modules.  To prove the theorem, it suffices to show
that the map $\pi_{*}GC_{*}\to \pi_{*}HMC_{*}$ is a chain homotopy
equivalence of differential graded $\pi_{*}^{\bS}\bT$-modules.  Since
the map for general $X$ is $(-)\otimes_{\pi_{*}\bS}\pi_{*}X$
applied to the map in the case $X=\bS$, it suffices to treat the case
$X=\bS$.

Write $i$ for the map $\pi_{*}GC_{*}\to \pi_{*}HMC_{*}$ and
define the map $r\colon \pi_{*}HMC_{*}\to \pi_{*}GC_{*}$ by 
\begin{align*}
b'_{n}\otimes b_{m}&\mapsto b''_{m-n}&b'_{n}&\mapsto b''_{n}\\
a'_{n}\otimes b_{m}&\mapsto a''_{m-n}&a'_{n}&\mapsto a''_{n}\\
b'_{n}\otimes a_{m}&\mapsto a''_{m-n}\\
a'_{n}\otimes a_{m}&\mapsto 0.
\end{align*}
As the differential takes $a_{n}\ (n>1)$, $a'_{n}\ (n\geq 0)$,
$a''_{n}\ (n\in \bZ)$ to $b_{n-1}$, $b'_{n+1}$, $b''_{n-1}$
(respectively), takes $b_{n}\ 
(n>0)$, $b'_{n}\ (n\geq 0)$, $b''_{n}\ (n\in \bZ)$ to $0$, and in
the exceptional cases satisfies
\begin{align*}
d(a'_{n}\otimes a_{1})&=b'_{n+1}\otimes a_{1}-a'_{n}\\
d(b'_{n}\otimes a_{1})&=b'_{n},
\end{align*}
we see that $r$ preserves the
differential.  Using the formulas for the action of $\tact$ in the
previous section (and the corresponding formulas in $\pi_{*}GC_{*}$),
we see that $r$ is a map of $\pi_{*}^{\bS}\bT$-modules.  Moreover, the
composite $r\circ i$ is the identity on $\pi_{*}GC_{*}$. 

To complete the argument, we construct a chain homotopy $B$ from
$i\circ r$ to the identity on $\pi_{*}HM_{*}$ that satisfies $\tact
B=-B\tact$.  Since $1-i\circ r=0$ on the image of $i$, we can regard
$1-i\circ r$ as a homomorphism $(\pi_{*}HM_{*})/i\to \pi_{*}HM_{*}$
and construct $B$ as $(1-i\circ r)\circ A$ for $A$ a nullhomotopy of
$(\pi_{*}HM_{*})/i$ satisfying $\tact A=-A\tact$.  As a
$\pi_{*}\bS$-module, $(\pi_{*}HM_{*})/i$ is free on elements
\begin{align*}
a'_{n}\otimes a_{m}\quad (m>0,n\geq 0),
&&a'_{n}\otimes b_{m}\quad (m>0,n\geq 0),\\
b'_{n}\otimes a_{m}\quad (m>0,n>0),&&
b'_{n}\otimes b_{m}\quad (m>0,n>0).
\end{align*}
We extend the notation of these elements to all $m\in \bZ,n\geq 0$ by
understanding the elements as zero when $m,n$ is outside the indicated
ranges above.  Then for all $m\in \bZ, n\geq 0$, the following formulas
are valid for the differential on $(\pi_{*}HM_{*})/i$:
\begin{align*}
d(a'_{n}\otimes a_{m})&=b'_{n+1}\otimes a_{m}-a'_{n}\otimes b_{m-1}\\
d(a'_{n}\otimes b_{m})&=b'_{n+1}\otimes b_{m}\\
d(b'_{n}\otimes a_{m})&=b'_{n}\otimes b_{m}\\
d(b'_{n}\otimes b_{m})&=0,
\end{align*}
and the following formulas are valid for
the action of $\tact$ on $(\pi_{*}HM_{*})/i$: 
\begin{align*}
\tact(a'_{n}\otimes a_{m})
  &=b'_{n}\otimes a_{m}-a'_{n}\otimes b_{m}+(m+n)\eta a'_{n}\otimes a_{m}\\
\tact(a'_{n}\otimes b_{m})&=b'_{n}\otimes b_{m}+(m+n-1)\eta a'_{n}\otimes b_{m}\\
\tact(b'_{n}\otimes a_{m})&=b'_{n}\otimes b_{m}+(m+n-1)\eta b'_{n}\otimes a_{m}\\
\tact(b'_{n}\otimes b_{m})&=(m+n)\eta b'_{n}\otimes b_{m}.
\end{align*}
For $m\geq 1$, let $A(a'_{0}\otimes a_{m})=0$, $A(a'_{0}\otimes
b_{m})=0$, and for $m\geq 1,n\geq 1$, let
\begin{align*}
A(a'_{n}\otimes a_{m})&=0\\
A(a'_{n}\otimes b_{m})&=a'_{n-1}\otimes a_{m}+\dotsb+a'_{0}\otimes a_{m-(n-1)}\\
A(b'_{n}\otimes a_{m})&=a'_{n-1}\otimes a_{m}+\dotsb+a'_{0}\otimes a_{m-(n-1)}\\
A(b'_{n}\otimes b_{m})
&=a'_{n-1}\otimes b_{m}+\dotsb+a'_{0}\otimes b_{m-(n-1)}\\
&\qquad \qquad 
-(b'_{n-1}\otimes a_{m}+\dotsb+b'_{0}\otimes a_{m-(n-1)})
\end{align*}
We note that $A(a'_{n}\otimes b_{m})=A(b'_{n}\otimes a_{m})$ and 
\[
A(b'_{n}\otimes b_{m})=-(\tact+(m-n-1)\eta)A(b'_{n}\otimes a_{m});
\]
from this it follows that $\tact A=-A\tact$.  For $m\geq 1,n=0$, we
have 
\begin{align*}
(Ad+dA)(a'_{0}\otimes a_{m})&=A(b'_{1}\otimes a_{m}-a'_{0}\otimes b_{m-1})
=a'_{0}\otimes a_{m}\\
(Ad+dA)(a'_{0}\otimes b_{m})&=A(b'_{1}\otimes b_{m})=a'_{0}\otimes b_{m}-b'_{0}\otimes a_{m}=a'_{0}\otimes b_{m}
\end{align*}
and in the remaining cases $m\geq 1,n\geq 1$, we have
\begin{multline*}
(Ad+dA)(a'_{n}\otimes a_{m})=A(b'_{n+1}\otimes a_{m}-a'_{n}\otimes b_{m-1})\\
=a'_{n}\otimes a_{m}+\dotsb+a'_{0}\otimes a_{m-n}
  -(a'_{n-1}\otimes a_{m-1}+\dotsb+a'_{0}\otimes_{m-n})=a'_{n}\otimes a_{m}
\end{multline*}
\begin{multline*}
(Ad+dA)(a'_{n}\otimes b_{m})
=A(b'_{n+1}\otimes b_{m})+d(a'_{n-1}\otimes a_{m}+\dotsb+a'_{0}\otimes a_{m-(n-1)})\\
=a'_{n}\otimes b_{m}+\dotsb+a'_{0}\otimes b_{m-n}-(b'_{n}\otimes a_{m}+\dotsb+b'_{0}\otimes a_{m-n})\\
\qquad\qquad +b'_{n}\otimes a_{m}-a'_{n-1}\otimes b_{m-1}+\dotsb+b'_{1}\otimes
a_{m-(n-1)}-a'_{0}\otimes b_{m-n}\\
=a'_{n}\otimes b_{m}-b'_{0}\otimes a_{m-n}=a'_{n}\otimes b_{m}
\end{multline*}
\begin{multline*}
(Ad+dA)(b'_{n}\otimes a_{m})
=A(b'_{n}\otimes b_{m-1})+d(a'_{n-1}\otimes a_{m}+\dotsb+a'_{0}\otimes a_{m-(n-1)})\\
=a'_{n-1}\otimes b_{m-1}+\dotsb+a'_{0}\otimes b_{m-n}-(b'_{n-1}\otimes a_{m-1}+\dotsb+b'_{0}\otimes a_{m-n})\\
\qquad\qquad +b'_{n}\otimes a_{m}-a'_{n-1}\otimes b_{m-1}+\dotsb+b'_{1}\otimes
a_{m-(n-1)}-a'_{0}\otimes b_{m-n}\\
=b'_{n}\otimes a_{m}-b'_{0}\otimes a_{m-n}=b'_{n}\otimes a_{m}
\end{multline*}
\begin{multline*}
(Ad+dA)(b'_{n}\otimes b_{m})
=d(a'_{n-1}\otimes b_{m}+\dotsb+a'_{0}\otimes b_{m-(n-1)}\\
\shoveright{-(b'_{n-1}\otimes a_{m}+\dotsb+b'_{0}\otimes a_{m-(n-1)}))}\\
=b'_{n}\otimes b_{m}+\dotsb+b'_{1}\otimes b_{m-(n-1)}
-(b'_{n-1}\otimes b_{m-1}+\dotsb+b'_{0}\otimes b_{m-n})\\
=b'_{n}\otimes b_{m}-b'_{0}\otimes b_{m-n}=b'_{n}\otimes b_{m}.
\end{multline*}
This completes the proof.
\end{proof}

The end of Section~\ref{sec:TSS} specifies a formula for the map from
the Greenlees $\bT$-Tate spectral sequence to the Hesselholt-Madsen
$\bT$-Tate spectral sequence in terms of the usual identification of
the $E^{2}$-term of the former and our notation for the $E^{2}$-term
of the latter.  Given that the inclusion of $\pi_{*}X$ in homological
degree zero agree for these spectral sequences, it can be seen by
naturality from the case $X=\Sigma^{\infty}\bT_{+}$, that the formula
there is the unique one consistent with the differentials of the
spectral sequences.  Alternatively, it can be calculated directly
as in~\eqref{eq:piX} that for an element $v\in \pi_{*}X$, under the
usual identification of the $E^{2}$-term of the Greenlees spectral
sequence, up to sign, the element $t^{n}\otimes v\in \bZ[t,t^{-1}]\otimes
\pi_{*}X$ goes to the element 
\[
\begin{cases}
b'_{n}\otimes v+n\eta a'_{n}\otimes v-a'_{n}\otimes \tact v&n\geq 0\\
b'_{0}\otimes b_{|n|}\otimes v+n\eta b'_{0}\otimes a_{|n|}\otimes v-b'_{0}\otimes a_{|n|}\otimes \tact v&n\leq 0.
\end{cases}
\]
The formulas at the end of Section~\ref{sec:TSS} can then be checked
directly and the sign determined to be $(-1)^{n}$.

\section[Proof of Lemma~{\ref{lem:homeo}}]{Coherence of the equivalences \texorpdfstring{$E\bT/E\bT_{2n-1}\simeq
E\bT_{+}\sma S^{\bC(1)^{n}}$}{}\sbreak (Proof of Lemma~\ref{lem:homeo})}

In this section we prove Lemma~\ref{lem:homeo}.  We first
observe uniqueness.  We note that 
\begin{multline*}
\StabGB[\bT](\Sigma^{\infty} E\bT_{+}\sma
S^{\bC(1)^{n}},\Sigma^{\infty} E\bT_{+}\sma S^{\bC(1)^{n}})
\\\iso
\StabG[\bT](\Sigma^{\infty} E\bT_{+}\sma S^{\bC(1)^{n}},\Sigma^{\infty} E\bT_{+}\sma S^{\bC(1)^{n}})
\\\iso \StabG[\bT](\Sigma^{\infty}E\bT_{+},\Sigma^{\infty}E\bT_{+})\iso
\pi^{h\bT}_{0}(\bS),
\end{multline*}
which is isomorphic to $\bZ$
by~\cite[{$\mathrm{B}^{\bT}$}]{Greenlees-SegalToric}.  Thus, for the
moment assume there exists a Borel equivalence
$\Sigma^{\infty}E\bT/E\bT_{2n-1}\simeq \Sigma^{\infty}E\bT_{+}\sma
S^{\bC(1)^{n}}$ for all $n\geq 0$.  Then there are exactly two choices
for each $n$, which induce opposite sign isomorphisms on $H_{2n}\iso
\bZ$, and diagrams involving only spectra Borel equivalent to
$\Sigma^{\infty}E\bT/E\bT_{2n-1}$ (for fixed $n$) commute
exactly when they commute on $H_{2n}$.  It follows that the diagram
in~(i) determines a unique choice in the case $n=0$ and in the other
cases a unique choice once the choice for $n=1$ is specified.  The
diagram in~(ii) determines a unique choice in the case $n=1$. 

The proof of existence is a construction, which we do on the space
level.  First, it is convenient to change models.  Recall that $S(\bC(1)^{n})$
denotes the unit sphere in $\bC(1)^{n}$, where $\bC(1)$ denotes the
complex numbers $\bC$ with the standard action of $\bT$ as the unit
complex numbers.  Let $S(\bC(1)^{\infty})$ be the union of
$S(\bC(1)^{n})$, where we include $\bC(1)^{n}$ in $\bC(1)^{n+1}$ as
the vectors with $0$ in the last coordinate.  Then
$S(\bC(1)^{\infty})$ is a model of $E\bT$ and so we have a unique
$\bT$-homotopy class of $\bT$-homotopy equivalence from
$S(\bC(1)^{\infty})$ to our standard model of $E\bT$.  Looking at the
induced map on $\bT$-quotients, the restriction 
\[
\bC P^{n}=S(\bC(1)^{n})/\bT\to E\bT/\bT\iso B\bT
\]
factors through $B\bT_{2n-2}$ (with our doubled filtration index
convention) by obstruction theory, and so the map $S(\bC(1)^{n})\to
E\bT$ factors up to $\bT$-homotopy through $E\bT_{2n-2}$ by principal
bundle theory.
This constructs a $\bT$-homotopy commutative diagram of $\bT$-spaces
\[
\xymatrix@-1pc{%
S(\bC(1))\ar[r]\ar[d]_{\simeq}&S(\bC(1)^{2})\ar[r]\ar[d]_{\simeq}
&S(\bC(1)^{3})\ar[r]\ar[d]_{\simeq}
&\cdots \ar[r]&S(\bC(1)^{\infty})\ar[d]^{\simeq}\\
E\bT_{0}\ar[r]&E\bT_{2}\ar[r]&E\bT_{4}\ar[r]&\cdots \ar[r]&E\bT.
}
\]
Choosing homotopies, it suffices to construct and study the maps using the model
$S(\bC(1)^{\infty})/S(\bC(1)^{n})$ in place of
$E\bT/E\bT_{2n-1}=E\bT/E\bT_{2n-2}$. 

\begin{rem}
A concrete construction of the maps $S(\bC(1)^{n})\to E\bT_{2n-2}$ is
as follows.  For the model $W\bT$ of Section~\ref{sec:PosFilt}, we can
use the map that sends an element 
\[
(z_{1},\dotsc,z_{n})\in S(\bC(1)^{n})\subset \bC(1)^{n}
\]
to the element represented by 
\[
\left(\frac{z_{i_{0}}}{||z_{i_{0}}||},\dotsc,\frac{z_{i_{m}}}{||z_{i_{m}}||}\right),
(||z_{i_{0}}||^{2},\dotsc,||z_{i_{m}}||^{2})\in \bT^{m+1}\times\Delta[m]
\]
where $z_{i_{0}},\dotsc,z_{i_{m}}$ drops the coordinates $z_{j}$ that
are $0$, and the element of $\Delta[m]$ is specified in barycentric
coordinates.  The corresponding map to the model $E\bT$ sends the
element to
\[
\left(\frac{z_{i_{0}}}{||z_{i_{0}}||},\frac{z_{i_{0}}^{-1}}{||z_{i_{0}}||}\frac{z_{i_{1}}}{||z_{i_{1}}||},
\dotsc,\frac{z_{i_{m-1}}^{-1}}{||z_{i_{m-1}}||}\frac{z_{i_{m}}}{||z_{i_{m}}||}\right),
(||z_{i_{0}}||^{2},\dotsc,||z_{i_{m}}||^{2})\in \bT^{m+1}\times\Delta[m].
\]
In the case $n=1$, this map is clearly the canonical isomorphism
$\bT\iso \bT\times \Delta[0]$.  More generally, the map is surjective and generically one-to-one
on the top dimensional cell, and so is a homology isomorphism on the
underlying non-equivariant homotopy spheres.  Since the map is also
$\bT$-equivariant and both sides are free $\bT$-cell complexes, the map is
a $\bT$-homotopy equivalence.
\end{rem}

\begin{cons}
We construct a $\bT$-homeomorphism 
\[
h_{n,q}\colon S(\bC(1)^{n+q})/S(\bC(1)^{n})\to
S^{\bC(1)^{n}}\sma S(\bC(1)^{q})_{+}
\]
by constructing a $\bT$-equivariant map 
\[
\bar h_{n,q}\colon S(\bC(1)^{n+q})\to
S^{\bC(1)^{n}}\sma S(\bC(1)^{q})_{+}
\]
that sends all of $S(\bC(1)^{n})$ to the basepoint.  For
$\vec x=(w_{1},\ldots,w_{n},z_{1},\ldots,z_{q})\in
\bC(1)^{n+q}$ with $||\vec x||=1$, write $\vec
w=(w_{1},\ldots,w_{n})\in \bC(1)^{n}$ and $\vec
z=(z_{1},\ldots,z_{q})\in \bC(1)^{q}$.  We have $\vec w\in
S(\bC(1)^{n})\subset S(\bC(1)^{n+q})$ exactly when $\vec z=0$, in
which case we define $\bar h_{n,q}(\vec x)$ to be the basepoint as required; otherwise
we define 
\[
\bar h_{n,q}(\vec x)=\frac{\vec w}{||\vec z||}\sma \frac{\vec z}{||\vec z||}
\in S^{\bC(1)^{n}}\sma S(\bC(1)^{q})_{+}.
\]
This is continuous since $\vec w/||\vec z||$ goes to the basepoint as
$\vec z$ goes to $0$.  The factorization $h_{n,q}$ is a
$\bT$-equivariant continuous bijection of compact spaces and hence a
$\bT$-homeomorphism. 
\end{cons}

The maps $h_{n,q}$ are consistent for varying $q$ to define a
$\bT$-homeomorphism
\[
h_{n,\infty}\colon S(\bC(1)^{\infty})/S(\bC(1)^{n})\overto{\iso}
S^{\bC(1)^{n}}\sma S(\bC(1)^{\infty})_{+}.
\]
Collapsing $S(\bC(1)^{\infty})$ to a point, we then get a composite map 
\[
h_{n}\colon S(\bC(1)^{\infty})/S(\bC(1)^{n})\to S^{\bC(1)^{n}}
\]
that is a Borel equivalence.

\begin{prop}\label{prop:di}
The diagram
\[
\xymatrix{%
S(\bC(1)^{\infty})/S(\bC(1)^{m})\ar[d]_{h_{m}}\ar[r]
&S(\bC(1)^{\infty})/S(\bC(1)^{m+1})\ar[d]^{h_{m+1}}\\
S^{\bC(1)^{m}}\ar[r]
&S^{\bC(1)^{m+1}}
}
\]
commutes up to $\bT$-homotopy where the top horizontal map
is the quotient map and the bottom horizontal map is induced by the
inclusion of $\bC(1)^{m}$ into $\bC(1)^{m+1}$.
\end{prop}

\begin{proof}
Let $\bar i_{m}$ be the self-map of $S(\bC(1)^{\infty})$
induced by the map $\bC(1)^{\infty}\to \bC(1)^{\infty}$ that is the
identity on $\bC(1)^{m}$ but sends the standard basis vector $e_{n}$
to $e_{n+1}$ for $n>m$; let $i_{m}$ be the induced self-map of
$S(\bC(1)^{\infty})/S(\bC(1)^{m})$.  Then $\bar i_{m}$ is a $\bT$-equivariantly
homotopic to the identity through a homotopy that preserves
$S(\bC(1)^{m})$ pointwise, and therefore $i_{m}$ is also
$\bT$-equivariantly homotopic to the identity.  Since the diagram
\[
\xymatrix{%
S(\bC(1)^{\infty})/S(\bC(1)^{m})\ar[dr]_{h_{m}}\ar[r]^{i_{m}}
&S(\bC(1)^{\infty})/S(\bC(1)^{m})\ar[r]
&S(\bC(1)^{\infty})/S(\bC(1)^{m+1})\ar[d]^{h_{m+1}}\\
&S^{\bC(1)^{m}}\ar[r]
&S^{\bC(1)^{m+1}}
}
\]
commutes, the diagram in the statement commutes up to $\bT$-homotopy.
\end{proof}

If we choose $i,j\geq 0$, letting $n=i+j$, we define a $\bT$-equivariant map 
\[
\delta_{i,j}\colon S(\bC(1)^{\infty})/S(\bC(1)^{n})\to
S(\bC(1)^{\infty})/S(\bC(1)^{i})\sma
S(\bC(1)^{\infty})/S(\bC(1)^{j})
\]
as follows.  Writing an element of $S(\bC(1)^{\infty})$ as 
\[
\vec x=(v_{1},\ldots,v_{i},w_{1},\ldots,w_{j},z_{1},z_{2},\ldots),
\]
taking $\vec v=(v_{1},\ldots,v_{i},0,0,\ldots)\in \bC(1)^{\infty}$,
$\vec w=(w_{1},\ldots,w_{j},0,0,\ldots)\in \bC(1)^{\infty}$, and 
\[
\vec z=(\underbrace{0,\ldots,0}_{n},z_{1},z_{2},\ldots)\in \bC(1)^{\infty},
\]
we define 
\[
\delta_{i,j}(\vec x)=\frac{\vec v+\vec z}{||\vec v+\vec z||}\sma
\frac{\vec w+\vec z}{||\vec w+\vec z||} \in 
S(\bC(1)^{\infty})/S(\bC(1)^{i})\sma S(\bC(1)^{\infty})/S(\bC(1)^{j})
\]
where we understand the point as the basepoint when $\vec z=0$.  We
use this map in the following proposition.

\begin{prop}\label{prop:dii}
For any $i,j\geq 0$ and $n=i+j$, The diagram
\[
\xymatrix{%
S(\bC(1)^{\infty})/S(\bC(1)^{n})\ar[d]_{h_{n}}\ar[r]^-{\delta_{i,j}}
&S(\bC(1)^{\infty})/S(\bC(1)^{i})\sma
S(\bC(1)^{\infty})/S(\bC(1)^{j})\ar[d]^{h_{i}\sma h_{j}}\\
S^{\bC(1)^{n}}\ar[r]_-{\iso}
&S^{\bC(1)^{i}}\sma S^{\bC(1)^{j}}
}
\]
commutes up to $\bT$-homotopy where the bottom horizontal map is the
canonical isomorphism 
\[
S^{\bC(1)^{n}}=S^{\bC(1)^{i+j}}\iso S^{\bC(1)^{i}\oplus
\bC(1)^{j}}\iso 
S^{\bC(1)^{i}}\sma S^{\bC(1)^{j}}.
\]
\end{prop}

\begin{proof}
In the $\vec x;\vec v,\vec w,\vec z$ notation above the down-then-right map
sends $\vec x$ to 
\[
\frac{\vec v}{||\vec z||}\sma \frac{\vec w}{||\vec z||}
\]
while the right-then-down map sends $\vec x$ to 
\[
\frac{\vec v}{||\vec v+\vec z||\, ||\vec z||}\sma \frac{\vec w}{||\vec w+\vec z||\, ||\vec z||}.
\]
The homotopy 
\[
\frac{\vec v}{||\vec v+\vec z||^{t}\, ||\vec z||}\sma \frac{\vec w}{||\vec w+\vec z||^{t}\, ||\vec z||}.
\]
is $\bT$-equivariant.
\end{proof}

We are now ready to prove Lemma~\ref{lem:homeo}.

\begin{proof}
As discussed above, uniqueness follows immediately from existence.
We use the map in the Borel stable category given by the zigzag
\[
E\bT/E\bT_{2n-1}\overfrom{\simeq}S(\bC(1)^{\infty})/S(\bC(1)^{n})\overto{h_{n,\infty}}S(\bC(1)^{\infty})_{+}\sma
S^{\bC(1)^{n}}\overto{\simeq}E\bT_{+}\sma S^{\bC(1)^{n}}.
\]
Proposition~\ref{prop:dii} shows that the diagram in~(ii) commutes for
all $m$. To deduce that the diagram in~(i) commutes,
it suffices to show that the diagram
\[
\xymatrix{%
S(\bC(1)^{\infty})/S(\bC(1)^{n})\ar[r]^-{\delta_{i,j}}\ar[d]
&S(\bC(1)^{\infty})/S(\bC(1)^{i})\sma S(\bC(1)^{\infty})/S(\bC(1)^{j})\ar[d]\\
E\bT/E\bT_{2n-1}\ar[r]&E\bT/E\bT_{2i-1}\sma E\bT/E\bT_{2j-1}
}
\]
commutes after applying $H_{2n}$.  This is a straightforward homology
calculation, keeping track of signs.
\end{proof}

\section{The strong K\"unneth theorem for \texorpdfstring{$THH$}{THH}}\label{sec:THH}

The following well-known result does not appear to have a good
reference in the literature.

\begin{thm}\label{thm:thhkun}
Let $R$ be a commutative ring orthogonal spectrum.  Let $X$ and $Y$ be
$R$-algebras.  The symmetric monoidal structure map
\[
THH(X)\smal_{THH(R)}THH(Y)\to THH(X\smal_{R} Y)
\]
is an isomorphism in the Borel derived category of $THH(R)$-modules.
\end{thm}

We can assume without loss of generality that $R$ is cofibrant as a
commutative ring orthogonal spectrum and $X$ and $Y$ are cofibrant as
$R$-algebras.  Then the cyclic bar constructions $N^{\cy}(X)$,
$N^{\cy}(Y)$, and $N^{\cy}(R)$ model $THH(X)$, $THH(Y)$, and $THH(R)$,
respectively.
The map in question is induced by the map
\[
N^{\cy}X\sma N^{\cy}Y\to N^{\cy}(X\sma Y)\to N^{\cy}(X\sma_{R}Y)
\]
which coequalizes the maps
\[
\xymatrix@C-1pc{%
N^{\cy}X\sma N^{\cy}R\sma N^{\cy}Y\ar[r]<-.5ex>\ar[r]<.5ex>
&N^{\cy}X\sma N^{\cy}Y
}
\]
where the top arrow is the $N^{\cy}R$-action on $N^{\cy}X$ and the bottom
arrow is the $N^{\cy}R$-action on $N^{\cy}Y$. It suffices to prove the
following stronger point-set theorem.

\begin{thm}\label{thm:thhpskun}
Let $R$ be a commutative ring orthogonal spectrum.  Let $X$ and $Y$ be
$R$-algebras.  The symmetric monoidal structure map
\[
N^{\cy}(X)\sma_{N^{\cy}(R)}N^{\cy}(Y)\to N^{\cy}(X\sma_{R} Y)
\]
is an isomorphism in the point-set category of equivariant $N^{\cy}(R)$-modules.
\end{thm}

\begin{proof}
Since coequalizers
commute with geometric realization, looking at simplicial level $n$,
it suffices to show that the map
\[
(X\sma\cdots \sma X)\sma_{R\sma\cdots \sma R}(Y\sma\cdots \sma Y)\to
(X\sma_{R}Y)\sma \cdots\sma (X\sma_{R}Y)
\]
is an isomorphism and this is clear by inspection.
\end{proof}

\begin{rem}
Since the map in Theorem~\ref{thm:thhkun} is natural in the category
$\Ho^{M}(\Spcat_{R})$ obtained from the category of $R$-spectral
categories by formally inverting the Morita equivalences, the result
generalizes to the case when $X$ and $Y$ are $R$-spectral categories.
\end{rem}

\begin{rem}
We have consistently worked in the Borel derived category throughout
rather than the equivariant derived category.  It is clear from
Theorem~\ref{thm:thhpskun} that the strong K\"unneth theorem holds for
$THH$ in any category where $THH$ is modeled by a functor of
$N^{\cy}$ on cofibrant objects of the given types.  In particular, it
applies to the norm model of $THH$ constructed as
$N_{e}^{\bT}=I_{\bR^{\infty}}^{U}N^{\cy}$ as in \cite{ABGHLM}.
\end{rem}

\section[Proof of Theorem~{\ref{main:THHfg}}]{\texorpdfstring{$THH$}{THH} of smooth and proper algebras (Proof of Theorem~\ref{main:THHfg})}\label{sec:spthhfg}

\bgroup
\def\sma_#1{\mathbin{\wedge_{#1}}}

The purpose of this section is to prove Theorem~\ref{main:THHfg} of
the introduction.  To\"en-Vaqui\'e~\cite[2.6]{ToenVaquie} show that
any smooth and proper $k$-linear dg category $\aX$ is Morita
equivalent to a smooth and proper dg $k$-algebra, and so by Morita
invariance of $THH$~\cite[5.12]{BM-tc}, it suffices to prove the
following theorem.

\begin{thm}\label{thm:THHsp}\label{thm:small}
Let $R$ be a commutative ring orthogonal spectrum and $A$ a
$R$-algebra which is smooth and proper when viewed as an $R$-spectral category.  Then $THH(A)$ is a small
$THH(R)$-module.
\end{thm}

We emphasize that the statement here and the work in this section is
non-equivariant.

We recall that an $R$-module $M$ is small means that maps in the
derived category out of it $\aD_{R}(M,-)$ commutes with arbitrary
coproducts; this is equivalent to $M$ being weakly equivalent to the
homotopy retract of a finite cell $R$-module.  For the hypothesis of
the theorem, $A$ is a proper as an $R$-spectral category precisely
when $A$ is small as an $R$-module and $A$ is a smooth as an
$R$-spectral category precisely when $A$ is small as an
$A\smal_{R}A^{\op}$-module (q.v.~\cite[\S3.2]{BGT}), or more
precisely, when for some (hence any) cofibrant replacement $A'\to A$
in $R$-algebras, $A$ is small as an $A'\sma_{R}A^{\prime\op}$-module.

Without loss of generality, we model $R$ as a cofibrant commutative
ring orthogonal spectrum and we 
model $A$ as a cofibrant $R$-algebra.  These hypotheses will be
enough to ensure that all smash products we use below in the proof of
Theorem~\ref{thm:small} represent derived smash products.

The proof of Theorem~\ref{thm:small} involves a comparison of
different models representing $THH(A)$ as a $THH(R)$-module, some of
which involve elaborate simplicial constructions that play off the
smash product of orthogonal spectra against the smash products
$\sma_{R}$ for $R$-modules.  To avoid confusion, we will
always write $\sma_{\bS}$ in this section for the smash product of
orthogonal spectra to contrast with $\sma_{R}$.  Likewise, we will
refer to orthogonal spectra with no extra structure as $\bS$-modules
to contrast with the extra structure of $R$-modules or $A$-modules.  

The basic building blocks of the constructions of this section are the
two-sided and cyclic bar constructions relative to $\bS$ and $R$. We
begin by establishing conventions and notation for these constructions
in this section.  In this section only, 
we use $N\subdot^{\bS}$ and $N^{\cy;\bS}\subdot$ to denote the
two-sided bar and cyclic bar constructions using $\sma_{\bS}$, and $B^{\bS}$ and
$B^{\cy;\bS}$ for the geometric realizations.  We also use
corresponding notation replacing $\bS$ with $R$.  To be specific about
the face maps: in 
\[
N^{\bS}_{q}(X,A,Y)=X\sma_{\bS}\underbrace{A\sma_{\bS}\cdots \sma_{\bS}A}_{q}\sma_{\bS}Y,
\]
$d_{0}$ uses the left action of $A$ on $Y$ and $d_{q}$ uses the right
action of $A$ on $X$. We call $X$ the \term{leftside module} and $Y$
the \term{rightside module}; the leftside module is a right module and
the rightside module is a left module.
If the leftside module is an
$A$-bimodule, then the left action of $A$ on $X$ induces a left action
of $A$ on $N^{\bS}(X,A,Y)$; if $Y$ is also an $A$-bimodule, then
$N^{\bS}(X,A,Y)$ inherits an $A$-bimodule structure.  We have similar
observations using $R$ in place of $A$ or in place of $\bS$.

Recall that an $A$-bimodule (in the category of $\bS$-modules)
consists of an $\bS$-module $M$ together with commuting left and right
$A$-module action maps $A\sma_{\bS}M\to M$ and $M\sma_{\bS}A\to M$;
this is equivalent data to the structure of either a left or right
$A\sma_{\bS}A^{\op}$-module on $M$.  Likewise, an $A$-bimodule (in the
category of $R$-modules) consists of an $R$-module $M$ and commuting
left and right $A$-module action maps $A\sma_{R}M\to M$ and
$M\sma_{R}A\to M$, with equivalent data a (left or right)
$A\sma_{R}A^{\op}$-module structure on $M$.

Given an $A$-bimodule $M$ (in the category of $\bS$-modules), when forming
the cyclic bar construction
\[
N^{\cy;\bS}_{q}(A;M)=M\sma_{\bS}\underbrace{A\sma_{\bS}\cdots \sma_{\bS}A}_{q},
\]
our convention is to have $d_{0}$ use the left action of $A$ on $M$ and $d_{q}$ the right
action of $A$ on $M$. We then have a canonical isomorphism 
\[
N^{\cy;\bS}\subdot(A;M)\iso M\sma_{A\sma_{\bS} A^{\op}}N^{\bS}\subdot(A,A,A)
\]
where we use the (previously unused) left $A$-module structure on the
leftside $A$ and right $A$-module structure on the rightside $A$.
Writing things this way, it is natural to regard $M$ as a right
$A\sma_{\bS}A^{\op}$-module and $N^{\bS}\subdot(A,A,A)$ as a
simplicial left $A\sma_{\bS} A^{\op}$-module.

In the context of commutative ring orthogonal spectra, the two-sided bar
construction $B^{\bS}(R,R,R)$ and cyclic bar construction $B^{\cy;\bS}(R)$
are special cases of tensors.  In the following, we write $\otimes$
for the tensor of a commutative ring orthogonal spectrum with an unbased space
(in the category of commutative ring orthogonal spectra).  Then
$B^{\bS}(R,R,R)$ is canonically isomorphic to the tensor $R\otimes
\Delta[1]$ and $B^{\cy;\bS}(R)$ is canonically isomorphic to the
tensor $R\otimes(\Delta[1]/\partial\Delta[1])$.

We are now ready to begin the constructions of the various models of
$THH(A)$.  We use the following shorthand notation for some of the
constructions described above:

\begin{notn}
Let $RI$ denote $R\otimes \Delta[1]$.  Let $R_{0}=R\otimes \{0\}$ and
$R_{1}=R\otimes \{1\}$; then $R_{0}$ and $R_{1}$ are canonically
isomorphic to $R$ as commutative ring orthogonal spectra, but the subscripts
keep track of the maps of commutative ring orthogonal spectra $R\to RI$.  (In
fact, we have analogous maps $R_{t}$ for $t\in (0,1)$, although we do
not use these.)
\end{notn}

\begin{notn}
Let $ARI=A\sma_{R_{0}}RI$, the extension of scalars $RI$-algebra.  We
regard $ARI$ as an $R$ algebra via $R_{1}\to RI$.
\end{notn}

Recalling that
$d_{0}$ on the tautological element of $\Delta[1]_{1}$ corresponds to
the vertex $1$ and $d_{1}$ to the vertex $0$, we see that $ARI$ is the
geometric realization of the simplicial associative ring orthogonal spectrum
\[
ARI\subdot=N^{\bS}\subdot(A,R,R),
\]
and its $R$-algebra structure comes from the rightside $R$.

The inclusion of $A$ in $ARI\subdot$ is map of simplicial
associative ring orthogonal spectra with $A$ constant but not a map of simplicial
$R$-algebras.  The collapse map $ARI\subdot\to A$ is a map
of $R$-algebras and a simplicial homotopy equivalence of simplicial
associative ring orthogonal spectra.

\begin{notn}
Let $O\subdot$ be the simplicial model of the circle obtained by
gluing 2 copies of $\Delta[1]\subdot$ along $0$ and along $1$. To be
definite later, we label one copy (a) and the other (b).  Let
$RO\subdot$ be the simplicial commutative ring orthogonal spectrum $R\otimes
O\subdot$ (a $\sma_{\bS}$-power of $R$ in each simplicial degree).  We
write $RO$ for the geometric realization, which is canonically
isomorphic to the commutative ring orthogonal spectrum
$RI\sma_{R_{0}\sma_{\bS}R_{1}}RI$.  We have canonical maps of
commutative ring orthogonal spectra $R_{0},R_{1}\to RO$.  
\end{notn}

\begin{notn}
We let $ARO$ be the extension of scalars $RO$-algebra
$RO\sma_{R_{0}}A$.  When we regard this as an $R$-algebra it will be
via $R_{1}\to RO$.
\end{notn}

We can also describe $ARO$ as the geometric
realization of a simplicial associative ring orthogonal spectrum (or simplicial $R$-algebra).
We have 2 maps of $R$-algebras $ARI\to ARO$ corresponding to the 2 maps
of $R$-algebras $RI\to RO$ (corresponding to the 2 maps of
$\Delta[1]\subdot$ into $O\subdot$). We use the ``(a)'' map to define
a left $ARI$-module structure on $ARO$ and the ``(b)'' map to define a
right $ARI$-module structure on $ARO$; together these give a
$ARI$-bimodule structure in $R$-modules, a $ARI\sma_{R}ARI^{\op}$-module
structure. This bimodule structure commutes with the $RO$-module
structure.

\begin{notn}
Let $R\partial\Delta[2]=R\otimes \partial \Delta[2]$ and
$R\partial\Delta[2]\subdot=R\otimes \partial \Delta [2]\subdot$.  
\end{notn}

We have a map of simplicial sets $\alpha \colon \partial \Delta[2]\subdot \to O\subdot$
sending the 1-simplex $\{0,1\}$ to the point $\{0\}$ (i.e., collapsing
$d^{2}\Delta[1]\subdot$), sending $\{0,2\}$ to (a) and $\{1,2\}$ to (b).
This induces a weak equivalence of commutative ring orthogonal spectra
$R\partial \Delta[2]\to RO$. Indeed, this map is homotopic in the
category of commutative ring orthogonal spectra to an isomorphism, just as the
geometric realization of $\partial \Delta[2]\subdot \to O\subdot$ is
homotopic in the category of spaces to a homeomorphism.  For
consistency with this map, when we refer to $R\partial \Delta[2]$ as
an $R$-module it will be via $R\otimes \{2\}$.  (The vertex $2$ goes
to the vertex $1$ in the simplicial map $\partial \Delta[2]\subdot \to
O\subdot$.)

The inclusion $d^{2}\colon \Delta [1]\subdot \to \partial
\Delta[2]\subdot$ induces a map of commutative ring orthogonal spectra $RI\to
R\partial \Delta[2]$.  

\begin{notn}
Let $D$ be the $R\partial
\Delta [2]$-module $R\partial \Delta
[2]\sma_{d^{2},RI}B^{\bS}(A,A,A)$.  
\end{notn}

The maps $d^{0}, d^{1}\colon \Delta[1]\subdot\to \Delta [2]\subdot$
induce a pair of maps $RI\to R\partial \Delta[2]$, giving $D$
a pair of commuting $RI$-module structures -- indeed an
$RI\sma_{R_{1}}RI$-module structure; using the left and right
action of $A$ on $B^{\bS}(A,A,A)$, these extend to a $ARI$-bimodule
structure on $D$ (with the left action of $ARI$ corresponding to the
1-simplex $\{0,2\}$ of $\Delta[2]\subdot$ and the right action of $ARI$
corresponding to the 1-simplex $\{1,2\}$).  In fact, the $ARI$-bimodule
structure is a $ARI\sma_{R}ARI^{\op}$-module structure that commutes with the
$R\partial \Delta[2]$-module structure.  

The $A$-bimodule (in $\bS$-modules) weak equivalence $B^{\bS}(A,A,A)\to A$ induces
a weak equivalence $D\to ARO$, which is a map of $ARI$-bimodules in $R$-modules and
of $R\partial \Delta[2]$-modules (using the map $R\partial
\Delta[2]\to RO$ for the $R\partial \Delta[2]$-module structure on
$ARO$).
Applying $B^{\cy;R}(ARI;-)$, we get a weak equivalence of $R\partial
\Delta[2]$-modules 
\begin{equation}\label{eq:first}
B^{\cy;R}(ARI;D)\to B^{\cy;R}(ARI;ARO).
\end{equation}
We will see below that these are models of $THH(A)$; the weak
equivalence~\eqref{eq:first} is the first key ingredient to the proof of
Theorem~\ref{thm:small} below.

The last model arises by applying the Dennis-Waldhausen Morita
trick~\cite[\S6]{BM-tc} to
$B^{\cy;R}(ARI;D)$.  To do so, we need to identify $B^{\cy;R}(ARI;D)$ as
the geometric realization of the simplicial $\bS$-module
\begin{equation}\tag{\relax*}\label{eq:D}
N^{\cy;R}\subdot(ARI;N^{\cy;\bS}\subdot(A;ARI\sma_{R}ARI^{\op})).
\end{equation}
In \eqref{eq:D}, we use the above map of associative ring orthogonal spectra $A\to ARI$ to get a
map of associative ring orthogonal spectra 
\[
A\sma_{\bS}A^{\op}\to ARI\sma_{\bS}ARI^{\op}\to ARI\sma_{R}ARI^{\op},
\]
and this gives commuting right $A\sma_{\bS} A^{\op}$-module and left
$ARI\sma_{R}ARI^{\op}$-module structures on $ARI\sma_{R}ARI^{\op}$;
the canonical isomorphism (symmetry isomorphism on smash factors) 
\[
(ARI\sma_{R} ARI^{\op})^{\op}\iso ARI\sma_{R}ARI^{\op}
\]
makes this commuting right $A\sma_{\bS} A^{\op}$ and
$ARI\sma_{R}ARI^{\op}$-module structures, allowing us to do
construction~\eqref{eq:D}.  To show that $B^{\cy;R}(ARI;D)$ is 
the geometric realization of the simplicial $\bS$-module~\eqref{eq:D},
we just need to produce an isomorphism of left
$ARI\sma_{R}ARI^{\op}$-modules between $D$ and $B^{\cy;\bS}(A;ARI\sma_{R}ARI)$.

By construction $D$ is isomorphic to the
geometric realization of
\[
R\otimes \partial \Delta[2]\subdot \sma_{d^{2}} N^{\bS}\subdot(A,A,A).
\]  
We can construct $\partial \Delta[2]\subdot$ as the diagonal of a
bisimplicial set obtained by gluing a copy of $\Lambda^{2}_{2}$ (the faces $d^{0}$
and $d^{1}$ of $\partial \Delta[2]\subdot$) in a first simplicial
direction along vertices $0$ and $1$ to a copy of $\Delta[1]$ (the
face $d^{2}$) in a second simplicial direction.  Taking the
geometric realization of the first simplicial direction, we see that
$D$ is isomorphic to the geometric realization of 
\[
R\otimes \Lambda^{2}_{2}\sma_{R_{0}\sma_{\bS}R_{1}}N^{\bS}\subdot(A,A,A)
\iso
(R\otimes \Lambda^{2}_{2}\sma_{R_{0}\sma_{\bS}R_{1}}(A\sma_{\bS}
A^{\op}))\sma_{A\sma_{\bS} A^{\op}}N^{\bS}\subdot(A,A,A)
\]
(to be clear, using the $R$-action corresponding to the vertex $0$ of
$\Delta[2]\subdot$ to attach $A$ and corresponding to the vertex $1$
to attach $A^{\op}$).
Breaking $\Lambda^{2}_{2}$ as $\Delta[1]\cup_{\{1\}}\Delta[1]$, this
is easily recognized as
\[
(ARI\sma_{R}ARI^{\op})\sma_{A\sma_{\bS} A^{\op}}N^{\bS}\subdot(A,A,A)
\iso N^{\cy;\bS}(A;ARI\sma_{R}ARI^{\op}).
\]
The isomorphism constructed preserves the $ARI$-bimodule (in
$R$-module) structure, using the standard left action of $ARI$ on $ARI$
and right action on $ARI^{\op}$.

The previous paragraph constructs an isomorphism from $B^{\cy;R}(ARI;D)$
to the geometric realization of the simplicial $\bS$-module
\[
N^{R}\subdot(ARI,ARI,ARI)\sma_{ARI\sma_{R}ARI^{\op}}
(ARI\sma_{R}ARI^{\op})
\sma_{A\sma_{\bS} A^{\op}}N^{\bS}\subdot(A,A,A),
\]
Using the symmetry isomorphism of $\sma_{\bS}$ to switch the sides we
put $N^{\bS}\subdot(A,A,A)$ and $N^{R}\subdot(ARI,ARI,ARI)$ on, 
we get an isomorphism with 
\[
N^{\cy;\bS}\subdot(A;N^{R}\subdot(ARI,ARI,ARI))
\]
and so an isomorphism of $\bS$-modules
\begin{equation}\label{eq:second}
B^{\cy;R}(ARI;D)\iso B^{\cy;\bS}(A;B^{R}(ARI,ARI,ARI)).
\end{equation}
(This is the Dennis-Waldhausen Morita trick.)  We transport the
$R\partial \Delta[2]$-module structure on $B^{\cy;R}(ARI;D)$
constructed above to $B^{\cy;\bS}(A;B^{R}(ARI,ARI,ARI))$ along this
isomorphism.  We emphasize that this constructs a well-defined $R\partial
\Delta[2]$-module structure on $B^{\cy;\bS}(A;B^{R}(ARI,ARI,ARI))$,
but we can describe it more concretely as follows.  We
write $B^{\cy;\bS}(A;B^{R}(ARI,ARI,ARI))$ as the geometric realization of
the multisimplicial $\bS$-module
\[
N^{\cy;\bS}_{i}(A;N^{R}_{n}(ARI_{j},ARI_{m},ARI_{k})).
\]
For fixed $m,n$, grouping together the $i$,$j$,$k$-terms, we are
looking at
\begin{multline}
\tag{\relax**}\label{eq:action}
A^{\sma_{\bS}(i)}\sma_{\bS}(ARI_{j}\sma_{R}ARI_{k}\sma_{R}ARI_{m}^{\sma_{R}(n)})\\
=
A^{\sma_{\bS}(i)}\sma_{\bS}\left((A\sma_{\bS}R^{\sma_{\bS}(j)}\sma_{\bS}R)\sma_{R}(A\sma_{\bS}R^{\sma_{\bS}(k)}\sma_{\bS}R)\sma_{R}ARI_{m}^{\sma_{R}(n)}\right)
\end{multline}
where the $R$-module structure on $ARI\subdot$ for $\sma_{R}$ is the last factor of
$R$ in the expansion.  Viewing $\partial \Delta[2]\subdot$ as the
diagonal of a trisimplicial set, $R\partial \Delta[2]$ is the
geometric realization of the trisimplicial set 
\[
R_{0}\sma_{\bS}R^{\sma_{\bS}(i)}\sma_{\bS}
R_{1}\sma_{\bS}R^{\sma_{\bS}(j)}\sma_{\bS}
R_{2}\sma_{\bS}R^{\sma_{\bS}(k)},
\]
where $R_{0}$, $R_{1}$, and $R_{2}$ are all copies of $R$ that we have
marked by numbers to keep track of them in the formulas that follow.
(The numbers correspond to the vertices of $\partial \Delta[2]$.)
We can rewrite the above in terms of $\sma_{R_{2}}$ as
\[
R_{0}\sma_{\bS}R^{\sma_{\bS}(i)}\sma_{\bS}
\left((R_{1}\sma_{\bS}R^{\sma_{\bS}(j)}\sma_{\bS}R_{2})
\sma_{R_{2}}(R^{\sma_{\bS}(k)}\sma_{\bS}R_{2})\right),
\]
and we can move $R_{0}$ inside the last factor
\[
R^{\sma_{\bS}(i)}\sma_{\bS}
\left((R_{1}\sma_{\bS}R^{\sma_{\bS}(j)}\sma_{\bS}R_{2})
\sma_{R_{2}}(R_{0}\sma_{\bS}R^{\sma_{\bS}(k)}\sma_{\bS}R_{2})\right).
\]
In terms of the formula~\eqref{eq:action}, we have $R^{\sma_{\bS}(i)}$
acting on $A^{\sma_{\bS}}(i)$, we have $R^{\sma_{\bS}(j)}$ and
$R^{\sma_{\bS}(k)}$ acting on the respective factors in
\eqref{eq:action}, we have $R_{2}$ acting on the righthand $R$-factors
(used in $\sma_{R}$), and finally $R_{0}$ acts on the $A$ in the
$k$-factor while $R_{1}$ acts on the $A$ in the $j$-factor.  

The collapse maps 
\[
B^{R}(ARI,ARI,ARI)\to ARI\to A
\]
induce a weak equivalence of $\bS$-modules 
\begin{equation}\label{eq:third}
B^{\cy;R}(ARI;D)\to B^{\cy;\bS}(A;B^{R}(ARI,ARI,ARI))\to B^{\cy;\bS}(A).
\end{equation}
This is a map of $R\partial \Delta[2]$-modules when we give
$B^{\cy;\bS}(A)$ the module structure induced by the map
\[
\beta \colon \partial \Delta[2]\subdot\to \partial
\Delta[2]\subdot/\Lambda^{2}_{2}\iso 
\Delta[1]\subdot/\partial \Delta[1]
\]
and the usual $B^{\cy;\bS}(R)\iso R\otimes \Delta[1]/\partial
\Delta[1]$ module structure on $B^{\cy}(A)$.

In the end, we have constructed in~\eqref{eq:first},
\eqref{eq:second}, and~\eqref{eq:third}  a zigzag of weak equivalences of
$R\partial \Delta[2]$-modules relating $B^{\cy;R}(ARI;ARO)$ and
$B^{cy}(A)$.  We are now ready to prove Theorem~\ref{thm:small}

\begin{proof}[Proof of Theorem~\ref{thm:small}]
The hypothesis that $A$ is proper over $R$ means that $A$ is weakly
equivalent to a homotopy retract of a finite cell $R$-module.  Since
the derived functor $RO\sma_{R_{0}}(-)$ from $R$-modules to
$RO$-modules preserves cofiber sequences, we see that $ARO$ is weakly
equivalent to a homotopy retract of a finite cell $RO$-module.  Since
the map $R\otimes \alpha \colon R\Delta[2]\to RO$ is a weak
equivalence, we see that $ARO$ is weakly equivalent to a homotopy
retract of a finite cell $R\partial \Delta [2]$-module.  The
hypothesis that $A$ is smooth over $R$ is that $A$ is weakly
equivalent to a homotopy retract of a finite cell $A\sma_{R}A$-module.
Using the weak equivalence of $R$-algebras $ARI\to A$, we see that $ARI$
and $B^{R}(ARI,ARI,ARI)$ are weakly equivalent to homotopy retracts of
finite cell right $ARI\sma_{R}ARI$-modules. Since 
\[
B^{\cy;R}(ARI;ARO)\iso B^{R}(ARI,ARI,ARI)\sma_{ARI\sma_{R}ARI^{\op}}ARO,
\]
we see that $B^{\cy;R}(ARI;ARO)$ is weakly equivalent to 
the homotopy retract of a finite cell $R\partial \Delta [2]$-module.

Since the map $R\otimes \beta \colon R\otimes \partial \Delta [2]\to
B^{\cy;\bS}(R)$ is a weak equivalence, the forgetful functor from
$B^{\cy;\bS}(R)$-modules to $R\otimes \partial \Delta [2]$-modules is
the right adjoint of a Quillen equivalence.  We conclude that
$B^{\cy;\bS}(A)$ is weakly equivalent to a homotopy retract of a
finite cell $B^{\cy;\bS}(R)$-module.
\end{proof}

\begin{rem}\label{rem:TA}
The zigzag of weak equivalences between $B^{\cy;R}(ARI;ARO)$ and
$B^{cy}(A)$ can be interpreted as an isomorphism in the stable category 
\[
THH(A)\simeq THH^{R}(A;THH(R;A)).
\]
We note that although $THH(R;A)$ is weakly
equivalent as an $A$-module or $A^{\op}$-module to
$THH(R)\smal_{R}A$, it is generally
not for the
$A\sma_{R}A^{\op}$ 
structure in the weak equivalence: if it were, we would then have a weak
equivalence $THH(A)\simeq THH^{R}(A)\smal_{R}THH(R)$.  This does not
hold in the example of $R=H\bZ$ and $A=H\bZ[i]$ as shown by the
calculation in~\cite{Lindenstrauss-GaussianIntegers}.
\end{rem}

\egroup

\section[Proof of Theorem~{\ref{main:finite}}]{The finiteness theorem for \texorpdfstring{$TP$}{TP} (Proof of Theorem~\ref{main:finite})}\label{sec:tpfg}

The purpose of this section is to prove Theorem~\ref{main:finite} and
its analogue for $G=C_{p}$.  For convenience of reference, we state
the combined theorem here.

\begin{thm}\label{thm:combfinite}
Let $k$ be a perfect field of characteristic $p>0$ and let $G=\bT$ or
$G=C_{p}$.  If $\aX$ is a smooth and proper $k$-linear dg category,
then $T_{G}(THH(\aX))$ is a finite $T_{G}(THH(k))$-module.
\end{thm}

We have left off the second statement, because we use it to deduce the first:
the main work of the section is to prove the following theorem.

\begin{thm}\label{thm:fg}
Let $k$ be a perfect field of characteristic $p>0$, and let $G$ be a
closed subgroup of $\bT$.  If $B$ is an
$Hk$-algebra with $THH_{*}(B)$ finitely generated over $THH_{*}(k)$,
then $\pi^{tG}_{*} THH(B)$ is finitely generated over $\pi^{tG}_{*} THH(k)$. 
\end{thm}

To deduce Theorem~\ref{thm:combfinite} from Theorem~\ref{thm:fg}, we
note that by Morita invariance of $THH$~\cite[5.12]{BM-tc}, it
suffices to consider dg $k$-algebras that are smooth and proper as
$k$-linear dg categories in Theorem~\ref{thm:combfinite}.
Theorem~\ref{main:THHfg} then implies that the $Hk$-algebras
corresponding to smooth and proper dg $k$-algebras satisfy the
hypotheses of Theorem~\ref{thm:fg}.  From here
Theorem~\ref{thm:combfinite} is a consequence of the following observation.

\begin{prop}\label{prop:finite}
Let $G=\bT$ or $G=C_{p}$ and let $X$ be a $T_{G}(THH(k))$-module.  Then
$X$ is a finite $T_{G}(THH(k))$-module if and only if $\pi_{*}X$ is
finitely generated over $\pi_{*}(T_{G}(THH(k)))$.
\end{prop}

\begin{proof}
When $G=\bT$, $\pi_{*}(T_{\bT}(THH(k)))\iso \Wk[v,v^{-1}]$
where $\Wk$ denotes the $p$-typical Witt vectors on $k$ and $v$ is an
element of $TP_{-2}(k)=\pi_{-2}(T_{\bT}(THH(k)))$, a particular choice
of which Hesselholt constructed in~\cite[4.2]{Hesselholt-Periodic}.
As a graded ring $\pi_{*}(T_{\bT}(THH(k)))$ is a graded PID;
specifically, all graded ideals are of the form $(p^{n})$ for some
$n$.  As a consequence every finitely generated
$T_{\bT}(THH(k))$-module is a wedge of 
copies of suspensions $T_{\bT}(THH(k))$ and suspensions of the cofiber
of multiplication by $p^{n}$; a module is finite over
$T_{\bT}(THH(k))$ if and only if its homotopy groups are finitely
generated over $\pi_{*}(T_{\bT}(THH(k)))$.  In the case $G=C_{p}$,
$\pi_{*}(T_{C_{p}}(THH(k)))\iso k[v,v^{-1}]$, a graded field in the
sense that every graded module over it is free.  Every
$T_{C_{p}}(THH(k))$-module is a wedge of suspensions of
$T_{C_{p}}(THH(k))$ and a $T_{C_{p}}THH(k)$-module is finite if
and only if its homotopy groups are finitely generated over
$\pi_{*}T_{C_{p}}THH(k)$.
\end{proof}

For other closed subgroups $C_{r}<\bT$, writing $r=p^{n}m$ with
$(p,m)=1$, we have a weak equivalence $T_{C_{r}}THH(A)\simeq
T_{C_{p^{n}}}THH(A)$ induced by the transfer.  For $n>1$,
$\pi^{tC_{p^{n}}}THH(k)\iso \Wk_{n}[v,v^{-1}]$, where $Wk_{n}$ denotes
the Witt vectors of length $n$.  This graded ring has infinite
global (projective) dimension and there exist
$T_{C_{p^{n}}}THH(k)$-modules whose homotopy groups are finitely 
generated but which are not themselves small. For example,
$T_{C_{p}}THH(k)$ is a $T_{C_{p^{n}}}THH(k)$-module that is not small for
$n>1$ but has homotopy groups $(\pi_{*}T_{C_{p^{n}}}THH(k))/p$.

For the proof of Theorem~\ref{thm:fg}, we prove the following slightly
more general theorem.  

\begin{thm}\label{thm:genfinite}
Let $k$ be a perfect field of characteristic $p>0$, let $G$ be a closed subgroup of $\bT$, and let $X$ be a $G$-equivariant
$N^{\cy}(Hk)$-module with 
$\pi_{*}X$ finitely generated 
over $THH_{*}(k)$.
Then $\pi^{tG}_{*}X$ is finitely generated over $\pi^{tG}_{*}THH(k)$.
\end{thm}

\begin{proof}
We can assume without loss of generality that $G=\bT$ or $G=C_{p^{r}}$
for some $r>0$.  In the case when $G=\bT$, we let $r=\infty$ and we
understand $p^{\infty}=0$ and $\Wk_{\infty}=\Wk$.

First we note that the conditionally convergent Greenlees Tate spectral
sequence is strongly convergent: 
The hypothesis that $\pi_{*}(X)$ is finitely generated over
$THH_{*}(k)$ implies that $E^{2}_{i,j}=0$ for $j$ small enough
(below the minimum degree of a generator) and each
$E^{2}_{i,j}$ is a finite dimensional vector space over $k$. 

The Greenlees Tate spectral sequence for $X$ is a module over the
Greenlees Tate spectral sequence for $THH(k)$.  In the latter, the
$E^{2}$-term is isomorphic as a graded ring to $k[\bar v,\bar
v^{-1},t]$ in the case $G=\bT$ and $k[\bar v,\bar v^{-1},t,b]/b^{2}$
in the case when $G$ is finite, where $\bar v$ is in bidegree $(-2,0)$
(the image of $v\in \pi_{-2}^{tG}THH(k)$ in the spectral sequence),
$t$ is in bidegree $(0,2)$ (the image of a generator of
$\pi_{2}THH(k)$), and $b$ is in bidegree $(1,0)$. The elements $\bar
v$ and $t$ are infinite cycles while $d^{2r+1}b=\bar v^{r}t^{r}$,
q.v.~\cite[6.2]{Hesselholt-Periodic}.  When $r=1$, $t$ becomes zero on
$E^{4}$, but otherwise, we can choose $t$ so that it represents
$pv^{-1}$ in the spectral sequence.

The $E^{2}$-term for $X$ is naturally isomorphic as a graded $k[\bar
v,\bar v^{-1},t]$-module to $THH_{*}(X)\otimes_{THH_{*}(k)} k[\bar
v,\bar v^{-1},t]$ in the case when $G=\bT$ and
$THH_{*}(X)\otimes_{THH_{*}(k)} k[\bar v,\bar v^{-1},t,b]/b^{2}$ in
the case when $G$ is finite.  In either case, the hypothesis that
$THH_{*}(X)$ is finitely generated over $THH_{*}(k)$ implies that the
$E^{2}$-term for $X$ is finitely generated over $k[\bar v,\bar
v^{-1},t]$.  Since $k[\bar v,\bar v^{-1},t]$ is Noetherian, the
$E^{\infty}$-term is also finitely generated over $k[\bar v,\bar
v^{-1},t]$.   The theorem now follows from a standard spectral
sequence comparison argument; we give the full proof in the current
context. 

Choose elements $\bar x_{1},\ldots,\bar x_{n}$ in $E^{\infty}_{*,*}$
that generate $E^{\infty}$ as a $k[\bar v,\bar v^{-1},t]$-module.
Choose representatives $x_{1},\ldots,x_{n}$; we need to show that
$x_{1},\ldots,x_{n}$ generate $\pi_{*}^{tG}X$ over
$\Wk_{r}[v,v^{-1}]$.  Let $y\neq 0$ be any element of
$\pi_{*}^{tG}X$ and let $\bar y$ denote the element in
$E^{\infty}_{i,j}$ representing $y$, where $y$ is in filtration level
$i$ (but not $i-1$) and total degree $i+j$.  Define
$i_{\ell},j_{\ell}\in \bZ$ by $\bar x_{\ell}\in
E^{\infty}_{i_{\ell},j_{\ell}}$. It will also be convenient to write
$d=i+j$ for the total degree of $y$ and $d_{\ell}=i_{\ell}+j_{\ell}$
for the total degree of $x_{\ell}$. Since $y$ is arbitrary, it
suffices to show that $y$ is in the submodule of $\pi^{tG}_{*}X$
generated by $x_{1},\ldots,x_{n}$ over $\pi^{tG}_{*}THH(k)$.

Since $E^{\infty}_{*,*}$ is generated by $\bar x_{1},\ldots,\bar
x_{n}$ over $k[\bar v,\bar v^{-1},t]$, we can write
\[
\bar y=\bar a^{0}_{1}\bar v^{\half(i_{1}-i)}t^{\half(j-j_{1})}\bar x_{1}+\cdots
   + \bar a^{0}_{n}\bar v^{\half(i_{n}-i)}t^{\half(j-j_{n})}\bar x_{n}
\]
for some $\bar a^{0}_{1},\ldots,\bar a^{0}_{n}\in k$, where we must have
$\bar a^{0}_{\ell}=0$ if $i_{\ell}-i$ is odd, $j-j_{\ell}$ is 
odd, or $j-j_{\ell}<0$.  Let $a^{0}_{\ell}=\omega(\bar
a^{0}_{\ell})\in \Wk$ using the Teichm\"uller character and let 
\begin{align*}
y_{0}&=a^{0}_{1}v^{\half(i_{1}-i)}(pv^{-1})^{\half(j-j_{1})}x_{1}+\cdots 
   + a^{0}_{n}v^{\half(i_{n}-i)}(pv^{-1})^{\half(j-j_{n})}x_{n}\\
&=a^{0}_{n}p^{\half(j-j_{1})}v^{\half(d_{1}-d)}x_{1}+\cdots 
   + a^{0}_{n}p^{\half(j-j_{n})}v^{\half(d_{n}-d)}x_{n}.
\end{align*}
We again have $a^{0}_{\ell}=0$ when $d_{\ell}-d$ is odd, $j-j_{\ell}$
is odd, or $j-j_{\ell}<0$, so this formula makes sense.  Let
$z_{1}=y-y_{0}$.  Since $y_{0}$ also represents $\bar y$ in
$E^{\infty}_{i,j}$, we must have that $z_{1}$ is in filtration degree
$i-1$ or lower, so represents an element $\bar z_{1}$ in
$E^{\infty}_{i-s_{1},j+s_{1}}$ for some $s_{1}>0$.  Writing $\bar
z_{1}$ in terms of the generators $\bar x_{1},\ldots,\bar x_{n}$, we
must have
\[
\bar z_{1}=\bar a^{0}_{1}\bar v^{\half(i_{1}-i+s_{1})}t^{\half(j-j_{1}+s_{1})}\bar x_{1}+\cdots
   + \bar a^{0}_{n}\bar v^{\half(i_{n}-i+s_{1})}t^{\half(j-j_{n}+s_{1})}\bar x_{n}
\]
for some $\bar a^{1}_{1},\ldots,\bar a^{1}_{n}\in k$. Let
$a^{1}_{\ell}=\omega(\bar a^{1}_{\ell})$ and let
\begin{align*}
y_{1}&=y_{0}+a^{1}_{1}p^{\half(j-j_{1}+s_{1})}v^{\half(d_{1}-d)}x_{1}+\cdots
   + a^{1}_{n}p^{\half(j-j_{n}+s_{1})}v^{\half(d_{n}-d)}x_{n}\\
&=\bigl(a^{0}_{1}p^{\half(j-j_{1})}+a^{1}_{1}p^{\half(j-j_{1}+s_{1})}\bigr)v^{\half(d_{1}-d)}x_{1}
   + \cdots \\&\qquad\qquad\qquad\qquad\quad +  
     \bigl(a^{0}_{n}p^{\half(j-j_{n})}+a^{1}_{n}p^{\half(j-j_{n}+s_{1})}\bigr)v^{\half(d_{n}-d)}x_{n}.
\end{align*}
Let $z_{2}=y-y_{2}$; then $z_{2}$ must be in filtration level
$i-1-s_{1}$, so represents an element $\bar z_{2}$ in
$E^{\infty}_{i-s_{2},j+s_{2}}$ for some $s_{2}>s_{1}$.  Inductively
construct as above $y_{m},z_{m}$, with $\bar z_{m}\in
E^{\infty}_{i-s_{m},j+s_{m}}$ for a strictly increasing sequence of
positive integers $s_{m}$ and with $y_{m}$ of the form
\begin{multline*}
y_{m}=\bigl(a^{0}_{1}p^{\half(j-j_{1})}+\cdots +a^{m}_{1}p^{\half(j-j_{1}+s_{m})}\bigr)v^{\half(d_{1}-d)}x_{1}
   + \cdots \\+
      \bigl(a^{0}_{n}p^{\half(j-j_{n})}+\cdots +a^{m}_{n}p^{\half(j-j_{n}+s_{m})}\bigr)v^{\half(d_{n}-d)}x_{n},
\end{multline*}
where $a^{m}_{\ell}=0$ if $j-j_{\ell}+s_{m}$ is odd or negative or if
$d_{\ell}-d$ is odd.  Because $s_{m}\to \infty$, the coefficients 
\[
a^{0}_{\ell}p^{\half(j-j_{\ell})}+\cdots +a^{m}_{\ell}p^{\half(j-j_{\ell}+s_{m})}
\]
converge to define an element $y_{\infty}$ in the submodule of
$\pi^{tG}_{*}X$ generated by $x_{1},\ldots,x_{n}$ over
$\pi^{tG}_{*}THH(k)$.  The difference $y-y_{\infty}$ is in filtration
level $s_{m}$ for all $m$, and so $y=y_{\infty}$ since the filtration on
homotopy groups is complete.  
\end{proof}

\section{Comparing monoidal models}\label{sec:compare}

In Section~\ref{sec:Moore} we defined a lax monoidal model $T^{M}$ for
the Tate fixed point functor for any finite group $G$ and in Section~\ref{sec:lax} we defined a
lax symmetric monoidal model $JT^{\oL}_{G}$.  In this section we
argue that these define the same lax monoidal structure on the Tate
fixed points viewed as a functor to the stable category.  
This comparison is in
itself is not enough to compare the map
\[
T^{M}X\smal_{T^{M}\!A}T^{M}Y\to T^{M}(X\smal_{A}Y)
\]
we used in Section~\ref{sec:main} to the map 
\[
JT^{\oL}_{G}X\smal_{JT^{\oL}_{G}A}JT^{\oL}_{G}Y\to JT^{\oL}_{G}(X\smal_{A}Y)
\]
implicitly in the statement of the main theorem; we also make that
comparison here.

Both arguments use an elaboration of the operadic structure on
$T^{\oO}_{G}$ described in~\eqref{eq:operadassoc}, which for $n=2$
takes the form
\[
\oO(2)_{+}\sma T^{\oO}_{G}(X)\sma T^{\oO}_{G}(Y)\to
T^{\oO}_{G}(X\sma Y).
\]
This structure uses the diagonal map on $EG$. Any operad $\oO$ admits a
canonical map $\oO\to \Com\to \oCO$ and we view the structure
of~\eqref{eq:operadassoc} as corresponding to this map.  We can use
any map of operads $\oO\to \oCO$ and then the structure
in~\eqref{eq:operadassoc} generalizes to this context, where we use the 
coaction of $\oCO$ on $EG$ in place of the diagonal map.  This
structure is natural in maps of ($E_{\infty}$ or $A_{\infty}$) operads
over $\oCO$.  

Now consider the maps of operads over $\oCO$
\[
\oC_{1}
\from
\oL\times \oC_{1}
\to
\oL\times \oCO
\from
\oL
\]
where the lefthand map is projection, the middle map is induced by the
inclusion of $\oC_{1}$ in $\oCO$ and the righthand map is induced by
the identity on $\oL$ and the map $\oL\to\Com\to \oCO$.  The backward
maps are weak equivalence, and looking at the structure above, we get
a natural commuting diagram
\[
\xymatrix@R-1pc{%
\ooCo(2)_{+}\sma T^{\oC_{1}}_{G}X\sma T^{\oC_{1}}_{G}Y\ar[r]
&T^{\oC_{1}}_{G}(X\sma Y)\\
(\oL(2)\times \ooCo(2))_{+}\sma T^{\oL\times \oC_{1}}_{G}X\sma T^{\oL\times \oC_{1}}_{G}Y
\ar[u]^-{\simeq}\ar[r]\ar[d]_-{\simeq}
&T^{\oL\times \oC_{1}}_{G}(X\sma Y)\ar[d]^-{\simeq}\ar[u]_-{\simeq}\\
(\oL(2)\times \oCO(2))_{+}\sma T^{\oL\times \oCO}_{G}X\sma T^{\oL\times \oCO}_{G}Y\ar[r]
&T^{\oL\times \oCO}_{G}(X\sma Y)\\
\oL(2)_{+}\sma T^{\oL}_{G}X\sma T^{\oL}_{G}Y\ar[r]\ar[u]^-{\simeq}
&T^{\oL}_{G}(X\sma Y)\ar[u]_-{\simeq}
}
\]
(where we have used the identity permutation subspace $\ooCo(2)$ of
$\oC_{1}(2)$ for the entries involving $\oC_{1}$).  Precomposing with
the universal map from the derived smash product to the smash product
and restricting to the case when $X$ and $Y$ are cofibrant, we get an
analogous diagram with the smash product replaced by the derived smash
product.  Together with the canonical map $\bS\to T^{\oO}_{G}\bS$,
the horizontal maps give structure maps for lax monoidal structures on
the Tate fixed point functor and the vertical maps imply that these
structures coincide.  We can now prove the following theorem.

\begin{thm}
The canonical isomorphism in the stable category between
$JT_{G}$ and $T^{M}$ preserves the lax
monoidal structure.
\end{thm}

\begin{proof}
The comparison of unit maps is clear.
The canonical isomorphism in the stable category $JT^{\oL}_{G}\to
T^{\oL}_{G}$ is symmetric monoidal
and the canonical map in the stable category 
\[
\oL(2)_{+}\sma T^{\oL}_{G}X \smal T^{\oL}_{G}Y\to
T^{\oL}_{G}X\smal_{\oL}T^{\oL}_{G}Y
\]
is an isomorphism, so the work above reduces to comparing the
associativity map for $T^{M}$ to the associativity map on
$T^{\oC_{1}}_{G}$.  For this it is enough to
compare the associativity map for $\bar T^{\bar M}_{*,*}$ and
$\bar T_{*,*}$, where $\bar T_{*,*}$ is defined in
Construction~\ref{cons:lastlabel}. 
Compatibility of these maps is easily seen from the commuting diagram
\[
\xymatrix{%
\bar T^{\bar M}_{i_{1},j_{1}}X \sma \bar T^{\bar M}_{i_{2},j_{2}}Y\ar[r]\ar[d]
&\bar T^{\bar M}_{i,j}(X\sma Y)\ar[d]
\\
\ooCo(2)_{+}\sma
\bar T_{i_{1},j_{1}}X \sma \bar T_{i_{2},j_{2}}Y\ar[r]
&\bar T_{i,j}(X\sma Y)
}
\] 
where the vertical map on the left is induced by
\begin{align*}
\bar T^{\bar M}_{i_{1},j_{1}}X \sma \bar T^{\bar M}_{i_{2},j_{2}}Y
&=
(\bar T_{i_{1},j_{1}}X\sma \bR_{+}^{>0})\sma 
(\bar T_{i_{2},j_{2}}Y\sma \bR_{+}^{> 0})\\
&\iso 
(\bR^{>0}\times \bR^{>0})_{+}\sma 
\bar T_{i_{1},j_{1}}X\sma \bar T_{i_{2},j_{2}}Y\\
&\hspace{-1ex}\to \ooCo(2)_{+}\sma \bar T_{i_{1},j_{1}}X\sma \bar T_{i_{2},j_{2}}Y
\end{align*}
(with the last map induced by the map $\mu_{2}$ of
Construction~\ref{cons:mu}), and the vertical map on the 
right is induced by the collapse map $\bR^{>0}\to *$.
\end{proof}

Next we move on to compare the maps
\[
T^{M}X\smal_{T^{M}\!A}T^{M}Y\to T^{M}(X\smal_{A}Y)
\]
and
\[
JT^{\oL}_{G}X\smal_{JT^{\oL}_{G}A}JT^{\oL}_{G}Y\to JT^{\oL}_{G}(X\smal_{A}Y).
\]
First we need to compare the objects, and we follow roughly same
strategy as above, transporting the constructions across the maps of
operads.  Following the ideas of~\cite{Mandell-Smash}, we construct the balanced
smash product of operadic modules.

The diagram preceding the previous theorem compares the $n=2$ case of
the structure maps 
\begin{equation}\label{eq:genopassoc}
\ooO(n)_{+}\sma T^{\oO}_{G}X_{1}\sma \cdots \sma  T^{\oO}_{G}X_{n}
\to  T^{\oO}_{G}(X_{1}\sma \cdots \sma X_{n})
\end{equation}
for $\oO=\ooCo$, $\oL\times \ooCo$, $\oL\times \oCO$, or $\oL$, where
$\ooO(n)$ denotes the corresponding $n$th space (for the $E_{\infty}$
operads) or identity permutation subspace (for the $A_{\infty}$
operads).  We also write $\ooO$ for the non-$\Sigma$ $A_{\infty}$
operad corresponding to $\oO$: For one of the $E_{\infty}$ operads
$\ooO=\oO$ with the permutations forgotten, and for one of the
$A_{\infty}$ operads $\oO$ is the identity permutation component in
each arity.  It is clear from the general case of the structure map
and comparison diagram that for an associative ring orthogonal
$G$-spectrum $A$, $T^{\oO}_{G}A$ inherits the structure of a
(non-$\Sigma$) $\ooO$-algebra and the maps of operads $\oO\to \oO'$
over $\oCO$ induce maps of $\ooO$-algebras $T^{\oO}_{G}A\to
T^{\oO'}_{G}A$.  In addition, for $Y$ a left $A$-module,
$T^{\oO}_{G}Y$ inherits the structure of a left
$T^{\oO}_{G}A$-module over $\ooO$: It has structure maps of the form
\[
\ooO(n+1)_{+}\sma (T^{\oO}_{G}A)^{(n)}\sma T^{\oO}_{G}Y \to T^{\oO}_{G}Y
\]
(for $n\geq 0$) 
satisfying the usual unit and associativity conditions with respect to
the operadic multiplication on $A$.  The maps of operads $\oO\to \oO'$
above induce maps of left $T^{\oO}_{G}A$-modules over $\ooO$.
Analogous observations apply to right modules.

For an $\ooO$-algebra $B$,
the category of left $B$-modules over $\ooO$ is equivalent to the
category of left modules for an associative ring symmetric spectrum
$U^{\ooO}_{L}B$, called the \term{left enveloping algebra}. Concretely
$U^{\ooO}_{L}B$ can be constructed as the coequalizer
\[
\xymatrix@-1pc{%
\bigvee\limits_{n,m_{1},\ldots m_{n}}
(\ooO(n+1)\times(\ooO(m_{1})\times \cdots \times \ooO(m_{n})))_{+}\sma B^{(m)}
\ar@<-.5ex>[r]\ar@<.5ex>[r]
&\bigvee\limits_{n}\ooO(n+1)_{+}\sma B^{(n)}
}
\]
for $n\geq 0$, $m_{1},\ldots,m_{n}\geq 0$, and
$m=m_{1}+\cdots+m_{n}$, where one map is induced by the operadic
multiplication
\begin{multline*}
\ooO(n+1)\times(\ooO(m_{1})\times \cdots \times \ooO(m_{n}))\\
\iso \ooO(n+1)\times(\ooO(m_{1})\times \cdots \times \ooO(m_{n})\times \{1\})
\to \ooO(m+1)
\end{multline*}
and the other by the $\ooO$-action
\[
\ooO(m_{i})_{+}\sma B^{(m_{i})}\to B.
\]
The unit is induced by the map 
\[
\bS\iso \{1\}_{+}\sma \bS = \{1\}_{+}\sma B^{(0)}\to \ooO(1)\sma B^{(0)}
\]
and the multiplication is induced by the map
\begin{multline*}
(\ooO(n+1)_{+}\sma B^{(n)})\sma (\ooO(n'+1)_{+}\sma B^{(n')})\\
\iso (\ooO(n+1)\times \ooO(n'+1))_{+}\sma B^{(n+n')}
\to \ooO(n+n'+1)_{+}\sma B^{(n+n')}
\end{multline*}
induced by the operadic multiplication $\circ_{n+1}$.  Analogously, the
category of right $B$-modules over $\ooO$ is the category of right
modules for the \term{right enveloping algebra} $U^{\ooO}_{R}B$, which
may be constructed as an analogous coequalizer with identical formulas
except that it uses the map
\begin{multline*}
\ooO(n+1)\times(\ooO(m_{1})\times \cdots \times \ooO(m_{n}))\\
\iso \ooO(n+1)\times(\{1\}\times \ooO(m_{1})\times \cdots \times \ooO(m_{n}))
\to \ooO(m+1)
\end{multline*}
in the construction and the map 
\begin{multline*}
(\ooO(n+1)_{+}\sma B^{(n)})\sma (\ooO(n'+1)_{+}\sma B^{(n')})\\
\iso (\ooO(n'+1)\times \ooO(n+1))_{+}\sma B^{(n+n')}
\to \ooO(n+n'+1)_{+}\sma B^{(n+n')}
\end{multline*}
induced by $\circ_{1}$ (leaving the factors of $B$ in the same order)
in the multiplication. 

\begin{cons}
We construct a right $(U^{\ooO}_{R}B)^{\op}\sma (U^{\ooO}_{L}B)$-module
$\Bal^{\ooO}(B)$ as follows.  The underlying orthogonal spectrum is the
coequalizer 
\[
\xymatrix@-1pc{%
\bigvee\limits_{n,m_{1},\ldots m_{n}}
(\ooO(n+2)\times(\ooO(m_{1})\times \cdots \times \ooO(m_{n})))_{+}\sma B^{(m)}
\ar@<-.5ex>[r]\ar@<.5ex>[r]
&\bigvee\limits_{n}\ooO(n+2)_{+}\sma B^{(n)}
}
\]
where one map is induced by the operadic multiplication
\begin{multline*}
\ooO(n+2)\times(\ooO(m_{1})\times \cdots \times \ooO(m_{n}))\\
\iso 
\ooO(n+2)\times(\{1\}\times \ooO(m_{1})\times \cdots \times \ooO(m_{n})\times \{1\})
\to \ooO(m+2)
\end{multline*}
and the other by the $\ooO$-action
\[
\ooO(m_{i})_{+}\sma B^{(m_{i})}\to B.
\]
The left $U^{\ooO}_{R}B$-action is induced by the map
\begin{multline*}
(\ooO(n+1)_{+}\sma B^{(n)})\sma (\ooO(n'+2)_{+}\sma B^{(n')})\\
\iso
(\ooO(n'+2)\times \ooO(n+1))_{+}\sma B^{(n)}\sma B^{(n')}
\to \ooO(n'+n+2)_{+}\sma B^{(n+n')})
\end{multline*}
induced by $\circ_{1}$ (where the factors of $B$ remain in the same order).
The right $U^{\ooO}_{L}B$-action is induced by the map
\begin{multline*}
(\ooO(n'+2)_{+}\sma B^{(n')})\sma (\ooO(n+1)_{+}\sma B^{(n)}) \\
\iso
(\ooO(n'+2)\times \ooO(n+1))_{+}\sma B^{(n')}\sma B^{(n)}
\to \ooO(n'+n+2)_{+}\sma B^{(n'+n)})
\end{multline*}
induced by $\circ_{n'+2}$.
\end{cons}

\begin{defn}\label{defn:balsma}
Let $\ooO$ be a non-$\Sigma$ $A_{\infty}$ operad, let $B$ be an
$\ooO$-algebra, let $M$ be a right $B$-module over $\ooO$, and let $N$
be a left $B$-module over $\ooO$.  Define the \term{point-set balanced smash
product} of $M$ and $N$ over $B$ and $\ooO$ to be the orthogonal
spectrum
\[
M\sma_{B}^{\ooO}N:=\Bal^{\ooO}(B)\sma_{(U^{\ooO}_{R}B)^{\op}\sma (U^{\ooO}_{L}B)}
(M\sma N),
\]
an enriched bifunctor 
\[
\Mod_{U^{\ooO}_{R}B}^{r}\sma 
\Mod_{U^{\ooO}_{L}B}^{\ell}\to \Sp.
\]
\end{defn}

This definition relates to the example of $T^{\oO}_{G}$ as
follows. When $A$ is an associative 
ring orthogonal $G$-spectrum and $X$ and $Y$ are right and left
$A$-modules, the structure maps~\eqref{eq:genopassoc}
\[
\ooO(n+2)_{+}\sma 
T^{\oO}_{G}X\sma (T^{\oO}_{G}A)^{(n)}\sma T^{\oO}_{G}Y
\to T^{\oO}_{G}(X\sma A^{(n)}\sma Y)
\to T^{\oO}_{G}(X\sma_{A}Y)
\]
fit together to define a map
\begin{equation}\label{eq:bsmap}
T^{\oO}_{G}X\sma^{\ooO}_{T^{\oO}_{G}A}T^{\oO}_{G}Y\to T^{\oO}_{G}(X\sma_{A}Y),
\end{equation}
naturally in $A$, $X$, $Y$, and $\oO$.  

We have a derived version of the balanced smash product that is
easiest to define if we restrict to the non-$\Sigma$ $A_{\infty}$
operads $\ooCo$, $\oL\times \ooCo$, $\oL\times \oCO$, and $\oL$
involved in our comparison.  These operads have the following special
property that we prove in the next section.

\begin{thm}\label{thm:UB}
Let $\ooO=\ooCo$, $\oL\times \ooCo$, $\oL\times \oCO$, or $\oL$, and
let $B$ be any $\ooO$-algebra.
The maps 
\begin{align*}
\ooO(2)_{+}\sma B&\to U^{\ooO}_{L}B\\
\ooO(2)_{+}\sma B&\to U^{\ooO}_{R}B\\
\ooO(3)_{+}\sma B&\to \Bal^{\ooO}B
\end{align*}
in the defining colimits are weak equivalences.
\end{thm}

\begin{rem}
For more general non-$\Sigma$ $A_{\infty}$ operads, we do not expect
the enveloping algebras to always have the correct homotopy type for
arbitrary $\ooO$-algebras $B$; they will, however, have the correct
homotopy types for cofibrant $\ooO$-algebras, and so cofibrant
replacement of the algebra $B$ is necessary to obtain the correct
derived categories of $B$-modules in this setting.  This presents no
real practical difficulties except to complicate definitions and
statements of theorems.  To avoid these complications, we define
derived functors only for $\ooO=\ooCo$, $\oL\times \ooCo$, $\oL\times
\oCO$, or $\oL$.
\end{rem}

\begin{defn}
For $\ooO=\ooCo$, $\oL\times \ooCo$, $\oL\times \oCO$, or $\oL$, 
define the \term{derived balanced smash product} $\Tor^{B,\ooO}(M,N)$ as the
left derived enriched bifunctor \cite[5.3]{LewisMandell2} of the
point-set balanced product functor. 
\end{defn}

As a special case of~\cite[8.2]{LewisMandell2}, the derived balanced
smash product may be constructed by using a cofibrant left
$(U^{\ooO}_{R}B)^{\op}\sma (U^{\ooO}_{L}B)$-module approximation of
$M\sma N$; one good way to choose such an approximation is to smash a
cofibrant right $U^{\ooO}_{R}B$-module approximation of $M$ with a
cofibrant left $U^{\ooO}_{L}B$-module approximation of $N$.

We now need to compare the balanced smash product of
Definition~\ref{defn:balsma} with the Blumberg-Hill EKMM smash product in
$\oL(1)$-spectra in orthogonal spectra and the balanced smash product
of modules over an associative ring orthogonal spectrum.  The first of
these is the following theorem, the proof of which is given in 
the next section.  Recall that we use $J$ to denote the functor
$(-)\sma_{\oL}\bS$ from $\oL(1)$-spectra in orthogonal spectra to
EKMM $\bS$-modules in orthogonal spectra.

\begin{thm}\label{thm:BHbs}
Let $B$ be an associative ring EKMM $\bS$-module in orthogonal
spectra, $M$ a right $B$-module, and $N$ a left $B$-module (for the
EKMM smash product).  Then the Blumberg-Hill EKMM smash product
$M\sma_{B}N$ is canonically isomorphic to $J(M\sma_{B}^{\oL}N)$ for
the balanced smash product constructed in Definition~\ref{defn:balsma}.
\end{thm}

The following is now an immediate consequence of
Theorems~\ref{thm:UB} and~\ref{thm:BHbs}. 

\begin{cor}
Let $A$ be an associative ring orthogonal $G$-spectrum, let $X$ be a
right $A$-module and let $Y$ be a left $A$-module.  Then the natural
transformations of balanced smash products
\begin{multline*}
T^{\oC_{1}}_{G}X\sma_{T^{\oC_{1}}_{G}A}^{\ooCo}T^{\oC_{1}}_{G}Y\from
T^{\oL\times \oC_{1}}_{G}X\sma_{(T^{\oL\times \oC_{1}}_{G}A)}^{\oL\times \ooCo}T^{\oL\times \oC_{1}}_{G}Y\to\\
T^{\oL\times \oCO}_{G}X\sma_{(T^{\oL\times \oCO}_{G}A)}^{\oL\times
\oCO}T^{\oL\times \oCO}_{G}Y\from
T^{\oL}_{G}X\sma_{T^{\oL}_{G}A}^{\oL}T^{\oL}_{G}Y\\
\from
J(JT^{\oL}_{G}X\sma_{JT^{\oL}_{G}A}^{\oL}JT^{\oL}_{G}Y)\iso
JT^{\oL}_{G}X\sma_{JT^{\oL}_{G}A}JT^{\oL}_{G}Y
\end{multline*}
induce isomorphism of derived balanced smash products.
\end{cor}

We still need to compare the balanced smash product
$T^{\oC_{1}}_{G}X\sma^{\ooCo}_{T^{\oC_{1}}_{G}A}T^{\oC_{1}}_{G}Y$
in the context of $\ooCo$-algebras and modules with the balanced smash
product $T^{\bar M}X\sma_{T^{M}\!A}T^{\bar M}Y$ in the context of
associative orthogonal spectra and modules.  As in
Section~\ref{sec:Moore}, for any $\ooCo$-algebra $B$, we can form the
Moore construction $B^{M}$ as the pushout
\[
B^{M}:=(\bS \sma \bR^{\geq 0}_{+})\cup_{\bS\sma \bR^{>0}_{+}}(B\sma \bR^{>0}_{+}),
\]
which has the natural structure of an associative ring orthogonal
spectrum and also has the property that the canonical map $B^{M}\to B$
(induced by the map collapsing $\bR$ to a point) is a homotopy
equivalence of orthogonal spectra.  Given $X$ and $Y$ right and left
$B$-modules over $\ooCo$, the Moore construction $(-)^{\bar
M}:=(-)\sma \bR^{>0}_{+}$ converts $X$ and $Y$ to $B^{M}$-modules.  We
then have a natural map 
\[
X^{\bar M}\sma_{B^{M}}Y^{\bar M}\iso
X^{\bar M}\sma_{B^{M}}B^{M}\sma_{B^{M}}Y^{\bar M}\to
X\sma^{\ooCo}_{B}Y
\]
induced by the map $\mu_{3}$ (from Construction~\ref{cons:mu})
interpreted as a map 
\[
(X\sma \bR^{>0}_{+})\sma (B\sma \bR^{>0}_{+})\sma (Y\sma \bR^{>0}_{+})
\to \ooCo(3)_{+}\sma X\sma B\sma Y
\]
and the unital extension of $\mu_{3}$ viewed as a map
\[
(X\sma \bR^{>0}_{+})\sma (\bS\sma \bR^{>0}_{+})\sma (Y\sma \bR^{>0}_{+})
\to \ooCo(2)_{+}\sma X\sma Y
\]
composed with the maps in the defining colimit
\[
\ooCo(3)_{+}\sma X\sma B\sma Y\to X\sma^{\ooCo}_{B}Y
\qquad \text{and}\qquad
\ooCo(2)_{+}\sma X\sma Y\to X\sma^{\ooCo}_{B}Y.
\]

\begin{thm}
For $B$ a $\ooCo$-algebra and $X$ and $Y$ right and left $B$-modules
over $\ooCo$, the canonical map 
\[
X^{\bar M}\sma_{B^{M}}Y^{\bar M}\to X\sma^{\ooCo}_{B}Y
\]
induces an isomorphism on the derived balanced smash products.
\end{thm}

\begin{proof}
The proof follows the usual induction up the cellular filtration argument.
We may as well take $X$ and $Y$ to be cofibrant. Let $Y'\to Y^{M}$ be a
cofibrant approximation; we want to show that the composite point-set map
\[
X^{\bar M}\sma_{B^{M}}Y'\to 
X^{\bar M}\sma_{B^{M}}Y^{\bar M}\to X\sma^{\ooCo}_{B}Y
\]
is a weak equivalence.  Without loss of
generality, we can assume that $Y$ is a cellular left $B$-module,
$Y=\colim Y_{n}$ where each $Y_{n}$ is formed as the pushout of cell
attachments using the generating cofibrations~\cite[12.1]{MMSS} in the
model category of $U^{\ooCo}_{L}B$-modules as cells.   We can likewise
arrange that $Y'=\colim Y'_{n}$ where each $Y'_{n}\to
Y'_{n+1}$ is the inclusion of a subcomplex, and we have a system of
compatible weak equivalences $Y'_{n}\to Y^{M}_{n}$.  Since both
$X^{\bar M}\sma_{B^{M}}(-)$ and $X\sma^{\ooCo}_{B}(-)$ preserve
$h$-cofibrations and all colimits, it suffices to show that the maps
\[
X^{\bar M}\sma_{B^{M}}(Y'_{n}/Y'_{n-1})\to 
X^{\bar M}\sma_{B^{M}}(Y^{\bar M}_{n}/Y^{\bar M}_{n-1})
\to X\sma^{\ooCo}_{B}(Y_{n}/Y_{n-1})
\]
are weak equivalences, where we understand $Y'_{-1}=Y_{-1}=*$.  Since
both functors $X^{\bar M}\sma_{B^{M}}(-)$ and $X\sma^{\ooCo}_{B}(-)$ preserve
weak equivalences between cofibrant objects and coproducts, it
suffices to consider the case when $Y_{n}/Y_{n-1}=B\sma F_{n}S^{m}$ (for
some $m,n\in \bN$) and this case is clear.
\end{proof}

\begin{cor}
Let $A$ be an associative ring orthogonal $G$-spectrum, let $X$ be a
right $A$-module and let $Y$ be a left $A$-module.  Then the natural
transformation of balanced smash products
\[
T^{\bar M}X\sma_{T^{M}\!A}T^{\bar M}Y\to 
T^{\oC_{1}}_{G} X\sma_{T^{\oC_{1}}_{G}\!A}T^{\oC_{1}}_{G} Y
\]
induces an isomorphism on derived balanced smash products. 
\end{cor}

If the
underlying orthogonal $G$-spectra of $X$ and $Y$ come with structure
maps $\bS\to X$ and $\bS\to Y$, then the natural transformation of
balanced smash products
\[
T^{\bar M}X\sma_{T^{M}\!A}T^{\bar M}Y\to
T^{M}X\sma_{T^{M}\!A}T^{M}Y
\]
induces a weak equivalence of derived balanced smash products.  This
completes a comparison of the derived balanced smash product for our
filtered lax monoidal model and our (point-set) lax symmetric monoidal
model.  For the comparison of the maps 
\[
T^{M}X\smal_{T^{M}\!A}T^{M}Y\to T^{M}(X\smal_{A}Y)
\]
and
\[
JT^{\oL}_{G}X\smal_{JT^{\oL}_{G}A}JT^{\oL}_{G}Y\to JT^{\oL}_{G}(X\smal_{A}Y)
\]
we combine with the natural maps on derived functors induced by the
maps~\eqref{eq:bsmap} to obtain the following commuting diagram, where
all vertical arrows are isomorphisms in the stable category.
\[
\xymatrix@R-1pc{%
T^{\bar M}X\smal_{T^{M}\!A}T^{\bar M}Y
\ar[r]\ar[d]_-{\simeq}
&T^{\bar M}(X\smal_{A}Y)\ar[d]^-{\simeq}\\
\Tor^{T^{\oC_{1}}_{G}A,\ooCo}(T^{\oC_{1}}_{G}X,T^{\oC_{1}}_{G}Y)
\ar[r]
&T^{\oC_{1}}_{G}(X\smal_{A}Y)\\
\Tor^{T^{\oL\times \oC_{1}}_{G}A,\oL\times \ooCo}(T^{\oL\times \oC_{1}}_{G}X,T^{\oL\times \oC_{1}}_{G}Y)
\ar[r]\ar[d]_-{\simeq}\ar[u]^-{\simeq}
&T^{\oL\times \oC_{1}}_{G}(X\smal_{A}Y)
\ar[u]_-{\simeq}\ar[d]^-{\simeq}\\
\Tor^{T^{\oL\times \oCO}_{G}A,\oL\times \oCO}(T^{\oL\times \oCO}_{G}X,T^{\oL\times \oCO}_{G}Y)
\ar[r]
&T^{\oL\times \oCO}_{G}(X\smal_{A}Y)\\
JT^{\oL}_{G}X\smal_{JT^{\oL}_{G}A}JT^{\oL}_{G}Y
\ar[r]\ar[u]^-{\simeq}
&JT^{\oL}_{G}(X\smal_{A}Y)\ar[u]_-{\simeq}
}
\]

\section{Identification of the enveloping algebras and \texorpdfstring{$\Bal$}{Bal}}

In this section, we identify in more concrete terms the enveloping
algebras $U^{\ooO}_{L}B$, $U^{\ooO}_{R}B$, and balanced product object
$\Bal^{\ooO}B$ for an algebra $B$ over one of the particular
non-$\Sigma$ $A_{\infty}$ operads $\ooO=\ooCo$, $\oL\times \ooCo$,
$\oL\times \oCO$, and $\oL$.  As a consequence, we deduce
Theorems~\ref{thm:UB} and~\ref{thm:BHbs}.  The theorems only refer to the
underlying orthogonal spectra, but we have included a full description
of the ring and module structures for completeness.  We work one operad at a
time.  In each of the following subsections, we use $B$ to denote an
arbitrary $\ooO$-algebra for the given non-$\Sigma$ operad $\ooO$.

\subsection*{The operad \texorpdfstring{$\ooCo$}{C1}} 
For this operad, the left enveloping algebra was identified in
\cite[2.5]{Mandell-Smash} and we only need to review its description.

Let $D$ denote the subspace of $\ooCo(1)$
where the interval does not start at $0$; then $U^{\ooCo}_{L}B$ is the
pushout 
\[
U^{\ooCo}_{L}B=(\ooCo(1)_{+}\sma \bS)\cup_{D_{+}\sma \bS}(D_{+}\sma B)
\]
with unit induced by the map 
\[
\bS\iso \{1\}_{+}\sma \bS\to \ooCo(1)_{+}\sma \bS\to U^{\ooCo}_{L}B
\]
and product induced as follows.  Given elements $[x_{1},y_{1}]$ and
$[x_{2},y_{2}]$ of $D$, we get an element 
\[
[0,x_{1}/(x_{1}+(y_{1}-x_{1})x_{2})],
[x_{1}/(x_{1}+(y_{1}-x_{1})x_{2}),1]
\]
of $\ooCo(2)$ and an element
\[
[x_{1}+(y_{1}-x_{1})x_{2},x_{1}+(y_{1}-x_{1})y_{2}]
\]
of $D$. This can also be expressed in terms of $\circ_{2}$ where we
view $D$ as in \cite[\S2]{Mandell-Smash} as the subset of $\ooCo(2)$
where the first interval starts at $0$ and ends at the start of the second
interval, which is the given interval that does not start at 0.
The picture of
$\{[0,x_{1}],[x_{1},y_{1}]\}\circ_{2}\{[0,x_{2}],[x_{2},y_{2}]\}$ is
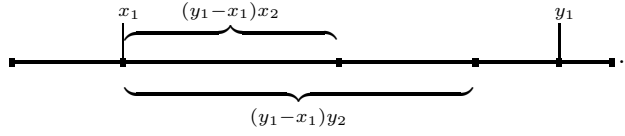
\[
\subsegnolabel{4em}\uplabel{x_{1}}%
\underbrace{\subsegnobeg{(y_{1}-x_{1})x_{2}}{8em}%
\subsegnobl{5em}}_{(y_{1}-x_{1})y_{2}}
\subsegnobl{3em}\uplabel{y_{1}}%
\subsegnobl{2em}\ .
\]
The first two segments translated and rescaled to begin at 0 and end
at 1 give the element of $\ooCo(2)$ and the third segment gives the
element of $D$. 
We use the element of $\ooCo(2)$ to multiply the copies of $B$ and the
element of $D$ above for the new element of $D$.  More formally,
writing $f\colon D\times D\to \ooCo(2)$ and $g\colon D\times D\to
\ooCo(1)$, the product is induced by
\[
(D_{+}\sma B)\sma (D_{+}\sma B)\iso (D\times D)_{+}\sma B\sma B
\overto{(g\times f)\sma \id} D_{+}\sma \ooCo(2)_{+}\sma B\sma B
\overto{\id\sma \xi_{2}} 
D_{+}\sma B
\]
where $\xi$ is the action of $\ooCo$ on $B$.

The right enveloping algebra has a similar description except that we
use the subspace $D'$ of $\ooCo(1)$ of intervals that do not end at
$1$.  The underlying orthogonal spectrum is then 
\[
U^{\ooCo}_{R}B=(\ooCo(1)_{+}\sma \bS)\cup_{D'_{+}\sma \bS}(D'_{+}\sma B)
\]
with unit induced by the map 
\[
\bS\iso \{1\}_{+}\sma \bS\to \ooCo(1)_{+}\sma \bS\to U^{\ooCo}_{R}B
\]
and product induced by the maps $D'\times D'\to \ooCo(2)$ and
$D'\times D'\to D'$ that send the pair of elements $[x_{1},y_{1}]$ and
$[x_{2},y_{2}]$ of $D$ to the element
\[
\left[0,\frac{(y_{2}-x_{2})(1-y_{1})}{(y_{2}-x_{2})(1-y_{1})+(1-y_{2})}\right],
\left[\frac{(y_{2}-x_{2})(1-y_{1})}{(y_{2}-x_{2})(1-y_{1})+(1-y_{2})},1\right]
\]
of $\ooCo(2)$ and the element
\[
[x_{2}+(y_{2}-x_{2})x_{1},x_{2}+(y_{2}-x_{2})y_{1}]
\]
of $D'$, respectively.  Here we are using the $\circ_{1}$ product
\[
\{[x_{2},y_{2}],[y_{2},1]\}\circ_{1} \{[x_{1},y_{1}],[y_{1},1]\},
\]
whose
picture is
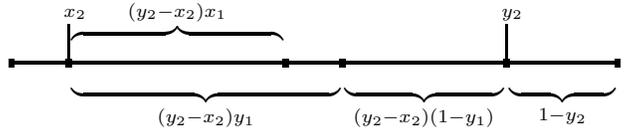
\[
\subsegnoel{2em}\mvuplabel{x_{2}}{.375pt}%
\underbrace{\subseg{(y_{2}-x_{2})x_{1}}{8em}%
\subsegnobl{2em}}_{(y_{2}-x_{2})y_{1}}%
\underbrace{\subsegnolabel{6em}}_{(y_{2}-x_{2})(1-y_{1})}%
\hskip-1pt
\mvuplabel{y_{2}}{.25pt}%
\underbrace{\subsegnolabel{4em}}_{1-y_{2}}\ .
\]
The last two segments translated and rescaled to start at 0 and end at
1 give the element of $\ooCo(2)$ and the third segment (whose length
is $(y_{2}-x_{2})(y_{1}-x_{1})$) gives the element of $D'$.

To identify $\Bal^{\ooCo}B$, let $C$ be the subset of $\ooCo(2)$
consisting of those pairs of intervals $[a,b],[c,d]$ with $b<c$.  Let
$Z$ be the pushout 
\[
Z=(\ooCo(2)_{+}\sma \bS)\cup_{C_{+}\sma \bS}(C_{+}\sma B).
\]
We define the right $U^{\ooCo}_{L}B$-module structure essentially
using $\circ_{2}$. Specifically, we have maps 
\[
C\times D\to \ooCo(2)\qquad \text{and}\qquad C\times D\to C
\]
sending the elements $[a,b],[c,d]$ in
$C$ and $[x,y]$ in $D$ to 
\[
[0,(c-b)/(c-b+x(d-c))],[(c-b)/(c-b+x(d-c)),1]
\]
in $\ooCo(2)$ and 
\[
[a,b],[c+(d-c)x,c+(d-c)y]
\]
in $C$, respectively; we use the element of $\ooCo(2)$ to multiply the
two factors of $B$ and the element of $C$ as the new element of $C$.
Pictorially, 
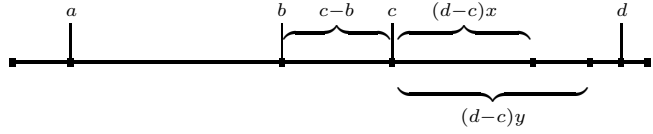
\[
\subsegnolabel{2em}\uplabel{a}
\subsegnobl{8em}\uplabel{b}
\subsegnobeg{c-b\mathstrut}{4em}\mvuplabel{c}{-0.875pt}
\underbrace{\subsegnobeg{(d-c)x}{5em}
\subsegnobl{2em}}_{(d-c)y}
\subsegnobl{1em}\uplabel{d}
\subsegnobl{1em}
\]
the element in $\ooCo(2)$ is the third and fourth segments translated
and rescaled to start at 0 and end at 1, while the element of $C$ is
the second and fifth segments.
We also need to describe what the action of the $\ooCo(1)_{+}\sma \bS$
part does; the same formulas apply (the first formula meaning to use
the isomorphism $B\sma \bS\iso B$), and so it is easy to see that this
describes a well-defined pairing $Z\sma U^{\ooCo}_{L}B\to Z$.
Examining the formula for $\{[0,1]\}_{+}\sma \bS$ shows that the
pairing is unital, and an tedious arithmetic check shows that it is
associative.

The left action of $U^{\ooCo}_{R}B$ is similar, using $\circ_{1}$
(on $C$) instead of $\circ_{2}$:  We have maps
\[
D'\times C\to \ooCo(2)\qquad \text{and}\qquad D'\times C\to C
\]
that take the pair of elements $[x,y]$ in $D'$ and $[a,b],[c,d]$ in
$C$ (with $b<c)$ to the element
\[
\left[0,\frac{(b-a)(1-y)}{(b-a)(1-y)+c-b}\right],\left[\frac{(b-a)(1-y)}{(b-a)(1-y)+c-b},1\right]
\]
in $\ooCo(2)$ and the element 
\[
[a+(b-a)x,a+(b-a)y],[c,d]
\]
of $C$.  Pictorially,
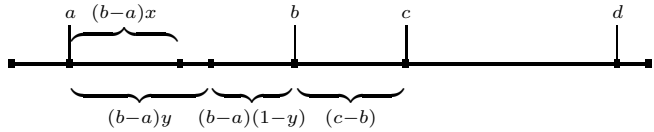
\[
\subsegnolabel{2em}\uplabel{a}%
\underbrace{\subsegnobeg{(b-a)x}{4em}%
\subsegnobl{1em}}_{(b-a)y}%
\underbrace{\subsegnobl{3em}}_{\hskip-1em (b-a)(1-y)\hskip-1em}\mvuplabel{b}{-.675pt}%
\hskip-1pt
\underbrace{\subsegnolabel{4em}}_{(c-b)}\mvuplabel{c}{-1pt}%
\subsegnoel{8em}\mvuplabel{d}{-1.25pt}%
\subsegnolabel{1em}
\]
the element of $\ooCo(2)$ is the fourth and fifth segments translated
and rescaled to start at 0 and end at 1, and the element of $C$ is the
third and sixth segments.

To see that the two actions commute, we note that the two maps
\[
U^{\ooCo}_{R}B\sma Z\sma U^{\ooCo}_{L}B
\]
can each be written in terms of maps
\[
D'\times C\times D\to \ooCo(3)\qquad \text{and}\qquad 
D'\times C\times D\to C
\]
using the element of $\ooCo(3)$ to multiply the three factors of $B$. 
Writing $[x',y']$ for the element of $D'$, $[a,b],[c,d]$ for the
element of $C$, and $[x_{2},y_{2}]$ for the element of $D$, both maps
give the same element
\[
\left[0,\tfrac{(b-a)(1-y')}{\ell}\right],
\left[\tfrac{(b-a)(1-y')}{\ell},
  \tfrac{(b-a)(1-y')+(b-c)}{\ell}\right],
\left[\tfrac{(b-a)(1-y')+(b-c)}{\ell},1\right]
\]
of $\ooCo(3)$, where $\ell=(b-a)(1-y')+c-b+(d-c)x$, and both maps give
the same element 
\[
[a+(b-a)x',a+(b-a)y'],[c+(d-c)x,c+(d-c)y]
\]
of $C$. Pictorially,
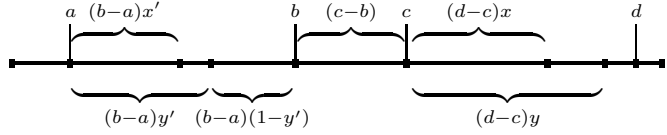
\[
\subsegnolabel{2em}\uplabel{a}%
\underbrace{\subsegnobeg{(b-a)x'}{4em}%
\subsegnobl{1em}}_{(b-a)y'}%
\underbrace{\subsegnobl{3em}}_{\hskip-1em (b-a)(1-y')\hskip-1em}\mvuplabel{b}{-.675pt}%
\hskip-1pt
\subseg{(c-b)}{4em}\mvuplabel{c}{-1pt}%
\underbrace{\subsegnobeg{(d-c)x}{5em}
\subsegnobl{2em}}_{(d-c)y}
\subsegnobl{1em}\uplabel{d}
\subsegnobl{1em}
\]
the element of $\ooCo(3)$ for both maps consists of the fourth, fifth,
and sixth segments translated and rescaled to start at 0 and end at 1,
and the element of $C$ consists of the third and seventh segments.

The map $C\to \ooCo(3)$ together with the maps $\ooCo(3)_{+}\sma B\to
\Bal^{\ooCo}B$ and $\ooCo(2)_{+}\sma \bS\to \Bal^{\ooCo}B$ induces a
map $Z\to \Bal^{\ooCo}B$.

\begin{thm}
The map $Z\to \Bal^{\ooCo}B$ is an isomorphism of right
$(U^{\ooCo}_{R}B)^{\op}\sma U^{\ooCo}_{L}B$-modules.
\end{thm}

\begin{proof}
We construct a map $\Bal^{\ooCo}B\to Z$ by constructing compatible
maps $\ooCo(n+2)_{+}\sma B^{(n)}\to Z$ as follows.  For $n=0$, we use
the map $\ooCo(2)_{+}\sma \bS\to Z$ from the construction of $Z$. For
$n>0$, we have an isomorphism 
\[
\ooCo(n+2)\to C\times \ooCo(n)
\]
defined by taking the element
\[
[x_{0},y_{0}],\ldots,[x_{n+1},y_{n+1}]
\]
in $\ooCo(n+2)$ (with $y_{i}\leq x_{i+1}$) to the element
$[x_{0},y_{0}],[x_{n+1},y_{n+1}]$ of $C$ and the element
\[
\left[\frac{x_{1}-y_{0}}{x_{n+1}-y_{0}},\frac{y_{1}-y_{0}}{x_{n+1}-y_{0}}\right],\ldots, 
\left[\frac{x_{n}-y_{0}}{x_{n+1}-y_{0}},\frac{y_{n}-y_{0}}{x_{n+1}-y_{0}}\right]
\]
in $\ooCo(n)$.  We then get a map $\ooCo(n+2)_{+}\sma B^{(n)}\to Z$
using the given element of $C$ and the element of $\ooCo(n)$ to
multiply the factors of $B$.  This obviously factors through the
coequalizer to define a map $\Bal^{\ooCo}B\to Z$.  Looking at the
composite maps 
\[
\ooCo(2)_{+}\sma \bS\to \Bal^{\ooCo}B\to Z
\quad \text{and}\quad
C_{+}\sma B\to \ooCo(3)_{+}\sma B\to \Bal^{\ooCo}B\to Z,
\]
we see that the composite map on $Z$ is the identity.  Likewise,
looking at the surjection $\circ_{2}\colon \ooCo(3)\times \ooCo(n)\to
\ooCo(n+2)$ for $n>0$ (which induces the inverse of the isomorphism
above when restricted to $C\times \ooCo(n)$), we can see that the
composite on $\Bal^{\ooCo}B$ is the identity.  To see that
$\Bal^{\ooCo}B\to Z$ is a map of bimodules, it suffices to check each
of the module structures separately.  The left action of
$U^{\ooCo}_{R}B$ on $\Bal^{\ooCo}B$ is induced by 
\begin{multline*}
(\ooCo(m+1)_{+}\sma B^{(m)})\sma (\ooCo(n+2)_{+}\sma B^{(n)})\\
\iso (\ooCo(n+2)\times \ooCo(m+1))_{+}\sma B^{(m+n)}
\overto{\circ_{1}\sma\id}\ooCo(m+n+2)_{+}\sma B^{(m+n)}.
\end{multline*}
It is now easy to check in the case $m=0,1$, $n=0,1$ that the
composite map to $Z$ is the same as the composite
\[
(\ooCo(m+1)_{+}\sma B^{(m)})\sma (\ooCo(n+2)_{+}\sma B^{(n)})
\to U^{\ooCo}_{R}B\sma Z
\]
with the action map $U^{\ooCo}_{R}B\sma Z$ defined above, and this
suffices to show the result. The case of the right action by
$U^{\ooCo}_{L}B$ is similar.
\end{proof}

We can now prove the case of Theorem~\ref{thm:UB} for $\ooO=\ooCo$,
which we state as the following proposition.

\begin{prop}\label{prop:UooCo}
The maps $\ooCo(2)_{+}\sma B\to U^{\ooCo}_{L}B$, $\ooCo(2)_{+}\sma
B\to U^{\ooCo}_{L}B$, and $\ooCo(3)_{+}\sma B\to \Bal^{\ooCo}B$ are
homotopy equivalences of orthogonal spectra.
\end{prop}

\begin{proof}
The case of the enveloping algebras are essentially proved
in~\cite[1.1]{Mandell-Smash}, but it is no extra work to include those
cases here.  We have maps 
\[
U^{\ooCo}_{L}B\to B, \qquad
U^{\ooCo}_{R}B\to B, \qquad
\Bal^{\ooCo}B\to B
\]
induced by the maps $\ooCo(1)\to *$, $\ooCo(2)\to *$, and $\ooCo(3)\to
*$.  Choosing elements in $\ooCo(2)$ and $\ooCo(3)$, we then get
composite maps
\[
U^{\ooCo}_{L}B\to \ooCo(2)_{+}\sma B, \qquad
U^{\ooCo}_{R}B\to \ooCo(2)_{+}\sma B, \qquad
\Bal^{\ooCo}B\to \ooCo(3)_{+}\sma B
\]
and we see that the composites on $\ooCo(2)_{+}\sma B$ and
$\ooCo(3)_{+}\sma B$
are homotopic to the identity since $\ooCo(2)$ and $\ooCo(3)$ are
contractible.  Likewise, it is straightforward to write explicit
formulas for contractions on the pairs $(\ooCo(1),D)$,
$(\ooCo(1),D')$, and $(\ooCo(2),C)$ that produce the homotopies for
the composites on $U^{\ooCo}_{L}B$, $U^{\ooCo}_{R}B$, and
$\Bal^{\ooCo}B$.
\end{proof}

\subsection*{The operad \texorpdfstring{$\oL$}{L}}

For the $\oL$ we also need to prove
Theorem~\ref{thm:BHbs} in addition to Theorem~\ref{thm:UB}.  For this
operad, we adapt the techniques of 
EKMM~\cite[\S I.5]{EKMM}.  We begin by identifying the enveloping algebras.

\begin{prop}
The underlying orthogonal spectra of $U^{\oL}_{L}B$ and $U^{\oL}_{R}B$
are naturally isomorphic to the pushout
\[
(\oL(1)_{+}\sma \bS)\cup_{\oL(2)_{+}\sma_{\oL(1)}\bS}\oL(2)\sma_{\oL(1)}B
\]
where for $U^{\oL}_{L}B$ the action of $\oL(1)$ on $\oL(2)$ is via
$\oL(1)\times \{1\}\subset \oL(1)\times \oL(1)$ and for $U^{\oL}_{L}B$ the action of $\oL(1)$ on $\oL(2)$ is via
$\{1\}\times \oL(1)\subset \oL(1)\times \oL(1)$.
\end{prop}

\begin{proof}
We treat the case of $U^{\oL}_{L}B$ as the other case is entirely
similar.  For the purpose of the proof, denote the pushout as $P$.
The defining coequalizer for $U^{\oL}_{L}B$ induces the map $P\to
U^{\oL}_{L}B$.  Using Hopkins' Lemma~\cite[I.5.4]{EKMM}, we have
isomorphisms 
\[
\oL(n+1)_{+}\sma B^{(n)}\iso \oL(2)_{+}\sma_{\oL(1)}(\oL(n)_{+}\sma B^{(n)})
\]
from which the $\oL$-action on $B$ induces a map to $P$.  These maps
glue over the coequalizer to construct a map $U^{\oL}_{L}B\to P$.
We see that the composite map is the identity on $P$ by looking at the
composite maps $\oL(2)_{+}\sma B\to P$ and $\oL(1)_{+}\sma \bS\to P$.
Likewise the map
\[
(\oL(1)_{+}\sma B)\vee \bigvee \oL(2)_{+}\sma (\oL(n)_{+}\sma B^{(n)}) \to U^{\oL}_{L}B
\]
is unchanged by composing with the composite $U^{\oL}_{L}B\to P\to
U^{\oL}_{L}B$, and we conclude that the composite is the identity on
$U^{\oL}_{L}B$ (since the coequalizer description of $U^{\oL}_{L}B$
plus Hopkins' Lemma implies that the displayed map is an epimorphism).
\end{proof}

For the pushout description of $U^{\oL}_{L}B$ in the previous
proposition, the unit $\bS\to U^{\oL}_{L}B$ is induced by the map
\[
\bS\iso \{1\}_{+}\sma \bS\to \oL(1)_{+}\sma \bS\to U^{\oL}_{L}B.
\]
The product is induced by the map 
\begin{multline*}
(\oL(2)_{+}\sma_{\oL(1)}B)\sma (\oL(2)_{+}\sma_{\oL(1)}B)
\to
(\oL(2)_{+}\sma_{\oL(1)}B)\sma_{\oL(1)} (\oL(2)_{+}\sma_{\oL(1)}B)\\
\iso
\oL(3)_{+}\sma_{\oL(1)\times \oL(1)}B\sma B
\iso
\oL(2)_{+}\sma_{\oL(1)}(\oL(2)_{+}\sma_{\oL(1)\times \oL(1)}B\sma B)\\
\to
\oL(2)_{+}\sma_{\oL(1)}B
\end{multline*}
and the straightforward modifications using $\oL(1)_{+}\sma \bS$ in
place of one or both copies of $\oL(2)_{+}\sma B$ with $\oL(1)_{+}\sma
\bS$. (The latter is the map
\[
(\oL(1)_{+}\sma \bS)\sma (\oL(1)_{+}\sma \bS)\iso (\oL(1)\times
\oL(1))_{+}\sma \bS\to \oL(1)_{+}\sma \bS
\]
induced by the operadic multiplication.)  The product for
$U^{\oL}_{R}B$ is similar except that the operad factors (but not the
$B$ factors) transpose, plugging the left $\oL(2)$ into the right
$\oL(2)$ with respect to the operadic multiplication.

We have the following concrete description of $\Bal^{\oL}B$.  The
proof uses the same techniques as the proof of the previous proposition.

\begin{prop}
The underlying orthogonal spectrum of $\Bal^{\oL}B$ is computed by the pushout 
\[
(\oL(2)_{+}\sma \bS)\cup_{\oL(3)_{+}\sma_{\oL(1)}\bS}(\oL(3)_{+}\sma_{\oL(1)}B)
\]
where the action of $\oL(1)$ on $\oL(3)$ is via $\{1\}\times
\oL(1)\times \{1\} \subset \oL(1)\times \oL(1)\times \oL(1)$. 
\end{prop}

The right action of $U^{\oL}_{L}B$ is induced by the map
\[
(\oL(3)_{+}\sma_{\oL(1)}B)\sma (\oL(2)_{+}\sma_{\oL(1)} B)
\to \oL(4)_{+}\sma_{\oL(1)\times \oL(1)} (B\sma B)
\to \oL(3)_{+}\sma_{\oL(1)}B
\]
again multiplying $\oL(2)_{+}\sma_{{\oL(1)}\times \oL(1)}B\sma B\to B$
on the inner factors of $B$, using Hopkins' Lemma (with similar
formulas for $\oL(2)_{+}\sma \bS$ and/or $\oL(1)_{+}\sma \bS$).  The
left action of $U^{\oL}_{R}B$ is similar but transposing the operad
spaces so that the operad space for $U^{\oL}_{R}B$ plugs in to the
operad space for $\Bal^{\oL}B$ for the operadic multiplication.

Turning to Theorem~\ref{thm:BHbs}, recall that $J(-)=(-)\sma_{\oL}\bS$
denotes the functor from $\oL(1)$-spectra in orthogonal spectra to
EKMM $\bS$-modules in orthogonal spectra.  One feature it has is that
$J$ turns the maps $\oL(n)_{+}\sma_{\oL}\bS\to \oL(n-1)_{+}\sma \bS$
into isomorphisms (another application of Hopkins' Lemma).  Because of
this we then have natural isomorphisms
\begin{equation}\label{eq:withJ}
\begin{aligned}
JU^{\oL}_{L}B&\iso J(\oL(2)_{+}\sma_{\oL(1)}B)\\
JU^{\oL}_{R}B&\iso J(\oL(2)_{+}\sma_{\oL(1)}B)\\
J\Bal^{\oL}B&\iso J(\oL(3)_{+}\sma_{\oL(1)}B)
\end{aligned}
\end{equation}
When $B$ is already an EKMM $\bS$-module in orthogonal spectra, we can
omit the $J$ on the right side.  We now prove Theorem~\ref{thm:BHbs}.

\begin{proof}[Proof of Theorem~\ref{thm:BHbs}]
It suffices to construct an isomorphism between $J(M\sma^{\oL}_{B}N)$
and $M\sma_{B}B\sma_{B}N$, where the smash products without the
superscript denote smash in the category of EKMM $\bS$-modules in
orthogonal spectra. With the simplification of~\eqref{eq:withJ}, we have
$J(M\sma^{\oL}_{B}N)$ as the coequalizer of
\[
\xymatrix@R-1pc{%
(\oL(3)_{+}\sma_{\oL(1)}B)\sma 
(U^{\oL}_{R}B)^{\op}\sma U^{\oL}_{L}B
\sma M\sma N
\ar@<.5ex>[d]\ar@<-.5ex>[d]\\
(\oL(3)_{+}\sma_{\oL(1)}B)\sma M\sma N.
}
\]
Because of the maps $\oL(1)_{+}\sma\bS\to (U^{\oL}_{R}B)^{\op}$ and
$\oL(1)_{+}\sma \bS\to U^{\oL}_{L}B$, we can replace the 
$(\oL(3)_{+}\sma_{\oL(1)}B)\sma M\sma N$ with
$(\oL(3)_{+}\sma_{\oL(1)^{3}}\sma (M\sma B\sma N)$.  Also
using~\eqref{eq:withJ} to replace the enveloping algebras with
$\oL(2)_{+}\sma_{\oL(1)}B$, we then identify $J(M\sma^{\oL}_{B}N)$ as
the coequalizer of
\[
\xymatrix@R-1pc{%
(\oL(3)_{+}\sma_{\oL(1)}B)\sma (\oL(2)_{+}\sma_{\oL(1)}B)\sma
(\oL(2)_{+}\sma_{\oL(1)}B)\sma M\sma N
\ar@<.5ex>[d]\ar@<-.5ex>[d]\\
\oL(3)_{+}\sma_{\oL(1)^{3}} (M\sma B\sma N).
}
\]
We can simplify the above coequalizer diagram to 
\[
\xymatrix@R-1pc{%
\oL(5)_{+}\sma_{\oL(1)^{5}}(M\sma B\sma B\sma B\sma N)
\ar@<.5ex>[d]\ar@<-.5ex>[d]\\
\oL(3)_{+}\sma_{\oL(1)^{3}} (M\sma B\sma N)
}
\]
using the categorical epimorphism
\begin{multline*}
(\oL(3)_{+}\sma_{\oL(1)}B)\sma (\oL(2)_{+}\sma_{\oL(1)}B)\sma
(\oL(2)_{+}\sma_{\oL(1)}B)\sma M\sma N\\
\to \oL(3)_{+}\sma_{\oL(1)^{3}} ((\oL(2)_{+}\sma_{\oL(1)}B)\sma B\sma
(\oL(2)_{+}\sma_{\oL(1)}B))\sma M\sma N\\
\iso \oL(5)_{+}\sma_{\oL(1)^{5}}(M\sma B\sma B\sma B\sma N).
\end{multline*}
The latter coequalizer is easily seen (using Hopkins' Lemma) to be
$M\sma_{B}B\sma_{B}N$. 
\end{proof}

The following proposition gives the instance of Theorem~\ref{thm:UB}
for the non-$\Sigma$ $A_{\infty}$ operad $\ooO=\oL$.

\begin{prop}\label{prop:UBL}
For any non-$\Sigma$ $\oL$-algebra $B$,
the maps $\oL(2)_{+}\sma B\to U^{\oL}_{L}B$, $\oL(2)_{+}\sma B\to
U^{\oL}_{R}B$, and $\oL(3)_{+}\sma B\to \Bal^{\oL}B$ are weak
equivalences of orthogonal spectra.
\end{prop}

\begin{proof}
It suffices to check that these maps are weak equivalences after
applying $J$.  Using~\eqref{eq:withJ}, each map is an instance of the
map 
\[
\oL(n)_{+}\sma_{\oL(1)}\bS\iso 
(\oL(n)\times_{\oL(1)}\oL(0))_{+}\sma \bS
\to \oL(n-1)_{+}\sma \bS.
\]
This map is a weak equivalence; see~\cite[XI.2.2]{EKMM}.
\end{proof}

\subsection*{The operad \texorpdfstring{$\oL\times \ooCo$}{LxC1}}
This operad combines the features of the previous two cases.  See the
case of $\ooCo$ for the definition of $D$, $D'$, and $C$.

\begin{prop}
The underlying orthogonal spectrum of $U^{\oL\times \ooCo}_{L}B$ is 
\[
((\oL(1)\times \ooCo)_{+}\sma \bS)\cup_{((\oL(2)\times D)_{+}\sma_{\oL(1)}\bS)}
((\oL(2)\times D)_{+}\sma B).
\]
The underlying orthogonal spectrum of $U^{\oL\times \ooCo}_{R}B$ is 
\[
((\oL(1)\times \ooCo)_{+}\sma \bS)\cup_{((\oL(2)\times D')_{+}\sma_{\oL(1)}\bS)}
((\oL(2)\times D')_{+}\sma B).
\]
The underlying orthogonal spectrum of $\Bal^{\oL\times \ooCo}B$ is 
\[
((\oL(2)\times \ooCo(2))_{+}\sma \bS)\cup_{((\oL(3)\times C)_{+}\sma_{\oL(1)}\bS)}
((\oL(3)\times C)_{+}\sma_{\oL(1)}B).
\]
\end{prop}

The proof is to use Hopkins' Lemma as in the previous subsection and
the homeomorphisms
\[
D\times \ooCo(n)\overto{\iso}\ooCo(n+1),\quad
D'\times \ooCo(n)\overto{\iso}\ooCo(n+1),\quad\text{and}\quad 
C\times \ooCo(n)\overto{\iso}\ooCo(n+2)
\]
(for $n>0$) induced by the maps $D\to \ooCo(2)$, $D'\to\ooCo(2)$, and
$C\to \ooCo(3)$ and the operadic products $\circ_{1}$,$\circ_{2}$, and
$\circ_{2}$, respectively, as in the first subsection.  The unit and
product on this model of $U^{\oL\times \ooCo}_{L}B$ and $U^{\oL\times
\ooCo}_{R}B$ are the evident generalization of the structure described
in the previous two subsections, as is the action on $\Bal^{\oL\times
\ooCo}B$. 

The following proposition is the instance of Theorem~\ref{thm:UB} for
the operad $\ooO=\oL\times \ooCo$.

\begin{prop}
The maps $(\oL(2)\times \ooCo(2))_{+}\sma B\to U^{\oL\times \ooCo}_{L}B$,
$(\oL(2)\times \ooCo(2))_{+}\sma B\to U^{\oL\times \ooCo}_{R}B$, and
$(\oL(3)\times \ooCo(3))_{+}\sma B\to \Bal^{\oL\times \ooCo}B$ are
weak equivalences of orthogonal spectra. 
\end{prop}

\begin{proof}
 We treat the case of $\Bal^{\oL\times
\ooCo}B$; the remaining cases are similar.  
The left action of $\oL(1)$ on $\oL(n)\times \ooCo(n)$ (as
$\oL(1)\times \{1\}$) and the models given in the previous proposition
allow us to view $\Bal^{\oL\times \ooCo}$ as $\oL(1)$-spectra; we can
then apply the functor $J$ to convert to EKMM $\bS$-modules (in
orthogonal spectra).   Applying $J$ to the pushout
\begin{equation}\label{eq:LCpo}
((\oL(3)\times \ooCo(2))\sma_{\oL(1)}\bS)
\cup_{((\oL(3)\times \ooCo(3))_{+}\sma_{\oL(1)} \bS)}
((\oL(3)\times \ooCo(3))_{+}\sma_{\oL(1)}B)
\end{equation}
is isomorphic to $J(\Bal^{\oL\times \ooCo}B)$, and so suffices to show
that the map from $\oL(3)_{+}\sma B$ to~\eqref{eq:LCpo} is a weak
equivalence. Using the weak equivalence
\[
\oL(3)_{+}\sma B\to \oL(3)_{+}\sma_{\oL(1)}B
\]
(see~\cite[XI.2.2]{EKMM}), it suffices to check that the inclusion of
$(\oL(3)\times \ooCo(3))_{+}\sma_{\oL(1)}B$ in the
pushout~\eqref{eq:LCpo} is a weak equivalence.  From here the proof is
identical to the proof of Proposition~\ref{prop:UooCo} using
$\oL(3)_{+}\sma_{\oL(1)}B$ in place of $B$ and
$\oL(3)_{+}\sma_{\oL(1)}\bS$ in place of $\bS$.
\end{proof}

\subsection*{The operad \texorpdfstring{$\oL\times \oCO$}{LxC1Xi}} 
The operad $\oCO$ satisfies $\oCO(n)=\oCO(1)^{n}$, the operad built
from the monoid $\oCO(1)$ using the diagonal multiplication.  In any
reasonable symmetric monoidal category tensored over spaces, a
non-$\Sigma$ $\oCO$-algebra $A$ is just an associative monoid object
together with a left action of the monoid $\oCO(1)$ on its underlying
object such that the multiplication is $\oCO(1)$-equivariant and the
unit is $\oCO(1)$-fixed. Likewise a left or right
$A$-module $M$ over $\oCO$ is just a left or right $A$-module (in the
usual sense) together with a left action of $\oCO(1)$ for which the
module structure map is $\oCO(1)$-equivariant.  
In particular, we see that the left enveloping
algebra of $A$ is $\oCO(1)_{+}\sma A$ (where we have written $(-)_{+}\sma (-)$ for
the tensor with a space) with $\oCO(1)$ acting diagonally.
Likewise, the right enveloping algebra is 
$\oCO(1)^{\op}_{+}\sma A$.  (Recall that our convention is for the right
enveloping algebra to act on the right; the alternative convention of
having it act on the left would yield $\oCO(1)_{+}\sma A^{\op}$ for the
right enveloping algebra.) Working in the weak symmetric monoidal
category of $\oL(1)$-spectra in orthogonal spectra, we obtain the
following analogous statement.

\begin{prop}
For any non-$\Sigma$ $(\oL\times \oCO)$-algebra $B$, we have natural
isomorphisms of associative ring orthogonal spectra
\begin{align*}
U^{\oL\times \oCO}_{L}B&\iso \oCO(1)_{+}\sma U^{\oL}_{L}B\\
U^{\oL\times \oCO}_{R}B&\iso \oCO(1)^{\op}_{+}\sma U^{\oL}_{R}B\\
\noalign{\noindent and an isomorphism of bimodules}
\Bal^{\oL\times \oCO}B&\iso (\oCO(1)\times \oCO(1))_{+}\sma \Bal^{\oL}B
\end{align*}
(where the action has both factors of $\oCO(1)$ in $\Bal$ always left
for the multiplication on $\oCO(1)$).
\end{prop}

The instance of Theorem~\ref{thm:UB} for $\oL\times \oCO$ now follows
from Proposition~\ref{prop:UBL}.

\section[A lax monoidal fibrant replacement functor]{A topologically
enriched lax symmetric monoidal fibrant replacement functor for
equivariant orthogonal spectra\sbreak (Proof of
Lemma~\ref{lem:smRG})}\label{sec:smRG}

In this section we construct a topologically enriched lax symmetric
monoidal fibrant replacement 
functor for the positive stable model category of 
orthogonal $G$-spectra.  We use an equivariant version of the construction
of Kro~\cite[3.3]{Kro}.

Before beginning the construction and the argument, it is useful to be
slightly more precise about the homotopy groups of an 
orthogonal $G$-spectrum.  Recall that a \term{complete $G$-universe} is an
infinite dimensional $G$-inner product space containing a
representative of each finite dimensional $G$-representation.  We use
the notation $V<U$ to denote that $V$ is a 
finite dimensional $G$-linear subspace of $U$.
Given a complete $G$-universe $U$, $V<U$, and $W$ an arbitrary finite
dimensional $G$-inner product space, for $H<G$ and $X$ a
orthogonal $G$-spectrum, define
\[
\pi^{H}_{W,V<U}X=\lcolim_{V<Z<U} [S^{W\oplus (Z-V)},X(Z)]^{H},
\]
where $[-,-]^{H}$ denotes the set of homotopy classes of maps of based
$H$-spaces and $Z-V$ denotes the orthogonal complement of $V$ in
$Z$. The following facts are well-established.
\begin{enumerate}
\item $\pi^{H}_{W,V<U}$ has the natural structure of an abelian group.
\item If $U'$ is a complete $G$-universe and $f\colon U\to U'$ is a
$G$-equivariant linear isometry (not necessarily isomorphism), then
the induced map 
\[
\pi^{H}_{W,V<U}X\to \pi^{H}_{W,f(V)<U'}X
\]
is an isomorphism.
\item For any finite-dimensional $G$-inner product space $W'$, the map
\[
\pi^{H}_{W,V<U}X\to \pi^{H}_{W\oplus W',V\oplus W'<U\oplus W'}
\]
(induced by $(-)\sma S^{W'}$ and the structure map on $X$) is an
isomorphism. 
\item A map $X\to Y$ of orthogonal $G$-spectra is a stable
equivalence if and only if the induced maps on $\pi^{H}_{W,V<U}$ are
isomorphisms for all $H<G$, $V<U$, and $W$.
\end{enumerate}
Indeed, the group $\pi^{H}_{W,V<U}X$ defined above is a specific model for the $RO(G)$-graded
homotopy group $\pi^{H}_{[W]-[V]}X$; (ii) and~(iii) are some minimal
invariance properties easily proved by comparison of colimit
arguments, while~(iv) follows from the fact that
a map in the stable category induces an isomorphism on integer-graded
homotopy groups if and only if it induces an isomorphism on
$RO(G)$-graded homotopy groups.  Another useful
observation is that when $U$ is a complete $G$-universe and $W$ is any
non-trivial finite dimensional $G$-inner product space, $U\otimes
W$ is also a complete $G$-universe \cite[IV.3.9]{MM}.

For $W$ any finite
dimensional $G$-inner product space, define 
\[
(R_{G}X)(W)=\lhocolim_{V<U}\; \Omega^{V\otimes W}X((\bR\oplus V)\otimes W)
\]
where for $V<V'$ the map in the hocolim system is induced by the
structure map for $X$
\begin{multline*}
X((\bR\oplus V)\otimes W)\sma S^{(V'-V)\otimes W}\\\to
X(((\bR\oplus V)\otimes W)\oplus ((V'-V)\otimes W))\iso
X((\bR\oplus V')\otimes W)
\end{multline*}
and the canonical isomorphism $V\oplus (V'-V)\iso V'$.
The structure map 
\[
(R_{G}X)(W)\sma S^{W'}\to (R_{G}X)(W\oplus W')
\]
is induced at the $V$ spot in the hocolim 
\[
\Omega^{V\otimes W}X((\bR\oplus V)\otimes W)\sma S^{W'}
\to
\Omega^{V\otimes (W\oplus W')}X((\bR\oplus V)\otimes (W\oplus W'))
\]
as the adjoint under the $(\Sigma^{W\oplus W'}$,$\Omega^{W\oplus
W'})$-adjunction to the map
\begin{multline*}
X((\bR\oplus V)\otimes W)\sma S^{W'}\sma S^{V\otimes W'}
\iso
X((\bR\oplus V)\otimes W)\sma S^{(\bR \oplus V)\otimes W'}\\
\to
X((\bR\oplus V)\otimes W\oplus (\bR \oplus V)\otimes W')
\iso
X((\bR\oplus V)\otimes (W\oplus W'))
\end{multline*}
coming from the structure map for $X$.  The check that the structure
map is well-defined works exactly as in the non-equivariant case, and
this together with the evident $G$ and $\aI(W,W')$-action maps make
$R_{G}$ into an endofunctor on orthogonal $G$-spectra.
The inclusion of $X(W)$ as $\Omega^{0\otimes W}X((\bR\oplus 0)\otimes
W)$ induces a natural transformation $\Id\to R_{G}$.  

\begin{prop}
For any orthogonal $G$-spectrum $X$, $R_{G}X$ is a positive
$G$-$\Omega$-spectrum and $X\to R_{G}X$ is a stable equivalence.
\end{prop}

\begin{proof}
The formula for $(R_{G}X)(W)$ gives a canonical isomorphism
\begin{align*}
\pi^{H}_{n}((R_{G}X)(W))&\iso 
\colim_{(V<U)} [S^{\bR^{n}\oplus (V\otimes W)},X((\bR\oplus V)\otimes W)]^{H}\\
&\iso
\colim_{(\bR\otimes W<Z<(\bR \oplus U)\otimes W)} [S^{\bR^{n}\oplus
(Z-\bR\otimes W)},X(Z)]^{H}
\end{align*}
with the second isomorphism by cofinality.  This gives a canonical
isomorphism 
\[
\pi^{H}_{n}(R_{G}X(W))\iso 
\pi^{H}_{\bR^{n},\bR\otimes W<(\bR\oplus U)\otimes W}X
\]
when $W$ is non-trivial (we have not defined the righthand side when
$W$ is trivial).   Similarly, we have a canonical isomorphism
\[
\pi^{H}_{n}(\Omega^{W'}(R_{G}X(W\oplus W')))
\iso \pi^{H}_{\bR^{n}\oplus W',\bR\otimes (W\oplus W')<(\bR\oplus U)\otimes (W\oplus W')}X
\]
and the adjoint structure map $R_{G}X(W)\to \Omega^{W'}(R_{G}X(W\oplus
W'))$ is a weak equivalence whenever $W$ is nontrivial as an instance
of properties~(ii) and~(iii) of homotopy groups listed above.  It
follows from this calculation (and property~(iv) of homotopy groups
above) that $R_{G}$ takes stable equivalences of orthogonal
$G$-spectra to positive level equivalences.  In particular, to prove
that the natural transformation $\Id\to R_{G}$ is a stable
equivalence, it suffices to check that it is a positive level
equivalence on a positive $G$-$\Omega$-spectrum $X$.  In this case,
for $W$ nontrivial, $\pi^{H}_{n}(X(W))\to \pi^{H}_{\bR^{n},W<U}X$ is
an isomorphism, and we can identify the induced map on homotopy groups
$\pi^{H}_{n}(X(W))\to \pi^{H}_{n}(R_{G}X(W))$ as an isomorphism once
again using property~(ii) of homotopy groups above.
\end{proof}

The functor $R_{G}$ is clearly continuous on mapping spaces.  Thus,
to complete the proof of Lemma~\ref{lem:smRG}, we need to construct
the lax symmetric monoidal structure on $R_{G}$ and prove that $\Id\to
R_{G}$ is a symmetric monoidal transformation.  This works essentially
just as in the non-equivariant case, using internal sum $+$ of finite
dimensional subspaces of $U$ in place of $\max$ of natural numbers.
(Of course, $\bR^{k}+\bR^{\ell}<\bR^{\infty}$ is
$\bR^{\max(k,\ell)}<\bR^{\infty}$.) Denoting by $\aI_{U}$ the
partially ordered set of $V<U$, $+$ defines a functor $\aI_{U}\times
\aI_{U}\to \aI_{U}$ that is strictly symmetric ($V+V'=V'+V$) and 
strictly associative ($(V+V')+V''=V+(V'+V'')$).  For specified 
orthogonal $G$-spectra $X$ and $X'$ and finite dimensional $G$-inner 
product spaces $W$ and $W'$, we have a natural transformation 
\[
\xymatrix@R-1.25pc{
\relax\Omega^{V\otimes W}X((\bR\oplus V)\otimes W)\sma 
\Omega^{V'\otimes W'}X((\bR\oplus V')\otimes W')\ar[d]\\
\relax\hspace{-1em}\Omega^{(V+V')\otimes W}X((\bR\oplus (V+V'))\otimes W)\sma
\Omega^{(V+V')\otimes W'}X'((\bR\oplus (V+V'))\otimes W')\hspace{-1em}\ar[d]\\
\relax\Omega^{(V+V')\otimes (W\oplus W')}((X\sma X')%
((\bR\oplus (V+V'))\otimes (W\oplus W')))
}
\]
subordinate to $+$ and inducing a map of based $G$-spaces
\[
R_{G}X(W)\sma R_{G}X'(W')\to (R_{G}(X\sma X'))(W\oplus W').
\]
The check that this assembles to a natural transformation 
\[
R_{G}X\sma R_{G}X'\to R_{G}(X\sma X')
\]
is now straightforward and essentially the same as in the non-equivariant
case, as are the remaining checks of the associativity and symmetry properties
and the check of lax symmetric monoidality of the natural transformation
$\Id\to R_{G}$. 


\bibliographystyle{plain}
\bibliography{thhtc}

\ifx\nolinkurl\undefined\def\nolinkurl#1{\texttt{\chardef~126\relax #1}}\fi
\begin{thebibliography}{10}

\bibitem{ABGHLM}
Vigleik Angeltveit, Andrew~J. Blumberg, Teena Gerhardt, Michael~A. Hill, Tyler
  Lawson, and Michael~A. Mandell.
\newblock Topological cyclic homology via the norm.
\newblock {\em Doc. Math.}, 23:2101--2163, 2018.

\bibitem{BGT}
Andrew~J. Blumberg, David Gepner, and Gon{\c{c}}alo Tabuada.
\newblock A universal characterization of higher algebraic {$K$}-theory.
\newblock {\em Geom. Topol.}, 17(2):733--838, 2013.

\bibitem{BGT2}
Andrew~J. Blumberg, David Gepner, and Gon{\c{c}}alo Tabuada.
\newblock Uniqueness of the multiplicative cyclotomic trace.
\newblock {\em Adv. Math.}, 260:192--232, 2014.

\bibitem{BlumbergHill-NormsTransfers}
Andrew~J. Blumberg and Michael~A. Hill.
\newblock Operadic multiplications in equivariant spectra, norms, and
  transfers.
\newblock {\em Adv. Math.}, 285:658--708, 2015.

\bibitem{BlumbergHill-GSymmetric}
Andrew~J. Blumberg and Michael~A. Hill.
\newblock {$G$}-symmetric monoidal categories of modules over equivariant
  commutative ring spectra.
\newblock {\em Tunis. J. Math.}, 2(2):237--286, 2020.

\bibitem{BM-tc}
Andrew~J. Blumberg and Michael~A. Mandell.
\newblock Localization theorems in topological {H}ochschild homology and
  topological cyclic homology.
\newblock {\em Geom. Topol.}, 16(2):1053--1120, 2012.

\bibitem{Boardman-SpectralSequences}
J.~Michael Boardman.
\newblock Conditionally convergent spectral sequences.
\newblock In {\em Homotopy invariant algebraic structures ({B}altimore, {MD},
  1998)}, volume 239 of {\em Contemp. Math.}, pages 49--84. Amer. Math. Soc.,
  Providence, RI, 1999.

\bibitem{ConnesConsani}
Alain Connes and Caterina Consani.
\newblock Cyclic homology, {S}erre's local factors and {$\lambda$}-operations.
\newblock {\em J. K-Theory}, 14(1):1--45, 2014.

\bibitem{Deninger-Main}
Christopher Deninger.
\newblock Motivic {$L$}-functions and regularized determinants.
\newblock In {\em Motives ({S}eattle, {WA}, 1991)}, volume~55 of {\em Proc.
  Sympos. Pure Math.}, pages 707--743. Amer. Math. Soc., Providence, RI, 1994.

\bibitem{EKMM}
A.~D. Elmendorf, I.~Kriz, M.~A. Mandell, and J.~P. May.
\newblock {\em Rings, modules, and algebras in stable homotopy theory},
  volume~47 of {\em Mathematical Surveys and Monographs}.
\newblock American Mathematical Society, Providence, RI, 1997.
\newblock With an appendix by M. Cole.

\bibitem{Greenlees-Tate}
J.~P.~C. Greenlees.
\newblock Representing {T}ate cohomology of {$G$}-spaces.
\newblock {\em Proc. Edinburgh Math. Soc. (2)}, 30(3):435--443, 1987.

\bibitem{Greenlees-SegalToric}
J.~P.~C. Greenlees.
\newblock The geometric equivariant {S}egal conjecture for toral groups.
\newblock {\em J. London Math. Soc. (2)}, 48(2):348--364, 1993.

\bibitem{GreenleesMay-Tate}
J.~P.~C. Greenlees and J.~P. May.
\newblock Generalized {T}ate cohomology.
\newblock {\em Mem. Amer. Math. Soc.}, 113(543):viii+178, 1995.

\bibitem{Hesselholt-Periodic}
Lars Hesselholt.
\newblock Topological {H}ochschild homology and the {H}asse-{W}eil zeta
  function.
\newblock In {\em An alpine bouquet of algebraic topology}, volume 708 of {\em
  Contemp. Math.}, pages 157--180. Amer. Math. Soc., Providence, RI, 2018.

\bibitem{HMAnnals}
Lars Hesselholt and Ib~Madsen.
\newblock On the {$K$}-theory of local fields.
\newblock {\em Ann. of Math. (2)}, 158(1):1--113, 2003.

\bibitem{Hirschhorn}
Philip~S. Hirschhorn.
\newblock {\em Model categories and their localizations}, volume~99 of {\em
  Mathematical Surveys and Monographs}.
\newblock American Mathematical Society, Providence, RI, 2003.

\bibitem{Illusie-DRW}
Luc Illusie.
\newblock Complexe de de\thinspace {R}ham-{W}itt et cohomologie cristalline.
\newblock {\em Ann. Sci. \'Ecole Norm. Sup. (4)}, 12(4):501--661, 1979.

\bibitem{Janssen-MotivesNumerical}
Uwe Jannsen.
\newblock Motives, numerical equivalence, and semi-simplicity.
\newblock {\em Invent. Math.}, 107(3):447--452, 1992.

\bibitem{Kontsevich-talk}
M.~Kontsevich.
\newblock Noncommutative motives.
\newblock Talk at the IAS on the occasion of the 61st birthday of Pierre
  Deligne. Available at
  \nolinkurl{http://video.ias.edu/Geometry-and-Arithmetic}, 2005.

\bibitem{Kro}
Tore~August Kro.
\newblock Model structure on operads in orthogonal spectra.
\newblock {\em Homology Homotopy Appl.}, 9(2):397--412, 2007.

\bibitem{LewisMandell2}
L.~Gaunce Lewis, Jr. and Michael~A. Mandell.
\newblock Modules in monoidal model categories.
\newblock {\em J. Pure Appl. Algebra}, 210(2):395--421, 2007.

\bibitem{Lindenstrauss-GaussianIntegers}
Ayelet Lindenstrauss.
\newblock The topological {H}ochschild homology of the {G}aussian integers.
\newblock {\em Amer. J. Math.}, 118(5):1011--1036, 1996.

\bibitem{MM}
M.~A. Mandell and J.~P. May.
\newblock Equivariant orthogonal spectra and {$S$}-modules.
\newblock {\em Mem. Amer. Math. Soc.}, 159(755):x+108, 2002.

\bibitem{MMSS}
M.~A. Mandell, J.~P. May, S.~Schwede, and B.~Shipley.
\newblock Model categories of diagram spectra.
\newblock {\em Proc. London Math. Soc. (3)}, 82(2):441--512, 2001.

\bibitem{Mandell-Smash}
Michael~A. Mandell.
\newblock The smash product for derived categories in stable homotopy theory.
\newblock {\em Adv. Math.}, 230(4-6):1531--1556, 2012.

\bibitem{McClure-EinftyTate}
J.~E. McClure.
\newblock {$E_\infty$}-ring structures for {T}ate spectra.
\newblock {\em Proc. Amer. Math. Soc.}, 124(6):1917--1922, 1996.

\bibitem{PatchkoriaSagave}
Irakli Patchkoria and Steffen Sagave.
\newblock Topological {H}ochschild homology and the cyclic bar construction in
  symmetric spectra.
\newblock {\em Proc. Amer. Math. Soc.}, 144(9):4099--4106, 2016.

\bibitem{SSAlgMod}
Stefan Schwede and Brooke~E. Shipley.
\newblock Algebras and modules in monoidal model categories.
\newblock {\em Proc. London Math. Soc. (3)}, 80(2):491--511, 2000.

\bibitem{ShipleyD}
Brooke Shipley.
\newblock Symmetric spectra and topological {H}ochschild homology.
\newblock {\em $K$-Theory}, 19(2):155--183, 2000.

\bibitem{Tabuada-Num}
Gon{\c{c}}alo Tabuada.
\newblock Finite generation of the numerical {G}rothendieck group.
\newblock Preprint, arXiv:1704.06252, 2017.

\bibitem{Tabuada-Weil}
Gon{\c{c}}alo Tabuada.
\newblock Noncommutative {W}eil conjecture.
\newblock Preprint, arXiv:1808.00950, 2018.

\bibitem{Tabuada-ULECT}
Gon\c{c}alo Tabuada.
\newblock {\em Noncommutative motives}, volume~63 of {\em University Lecture
  Series}.
\newblock American Mathematical Society, Providence, RI, 2015.
\newblock With a preface by Yuri I. Manin.

\bibitem{Toen-Lectures}
Bertrand To\"en.
\newblock Lectures on dg-categories.
\newblock In {\em Topics in algebraic and topological {$K$}-theory}, volume
  2008 of {\em Lecture Notes in Math.}, pages 243--302. Springer, Berlin, 2011.

\bibitem{ToenVaquie}
Bertrand To\"en and Michel Vaqui\'e.
\newblock Moduli of objects in dg-categories.
\newblock {\em Ann. Sci. \'Ecole Norm. Sup. (4)}, 40(3):387--444, 2007.

\end{thebibliography}

\end{document}